\documentclass[11pt,final]{amsart}

\usepackage{amssymb,amsmath,latexsym}
\usepackage{amsthm}
\usepackage{amsfonts}
\usepackage{enumerate}
\usepackage{booktabs}
\usepackage{mathrsfs}

\usepackage{hyperref}
\usepackage{xcolor}
\hypersetup{
    colorlinks, linkcolor={red!80!black},
    citecolor={blue!80!black}, urlcolor={blue}
}
\usepackage[color]{showkeys}
\definecolor{refkey}{gray}{.75}
\definecolor{labelkey}{gray}{.75}

\usepackage{fancyhdr}

\newtheorem{theorem}{Theorem}[section]
\newtheorem{corollary}{Corollary}[section]

\newtheorem{proposition}{Proposition}[section]
\newtheorem{lemma}{Lemma}[section]
\newtheorem{remark}{Remark}[section]

\newcommand{\aiminset}[1]{\left\{#1 \right\}}

\newcommand{\aiminabs}[1]{\lvert #1 \rvert }
\newcommand{\aiminnorm}[1]{\| #1 \| }
\newcommand{\aimininner}[2]{\langle #1, #2 \rangle}

\numberwithin{equation}{section}
\numberwithin{figure}{section}

\let\oldtocsection=\tocsection
\let\oldtocsubsection=\tocsubsection
\let\oldtocsubsubsection=\tocsubsubsection
\renewcommand{\tocsection}[2]{\hspace{0em}\oldtocsection{#1}{#2}}
\renewcommand{\tocsubsection}[2]{\hspace{2em}\oldtocsubsection{#1}{#2}}
\renewcommand{\tocsubsubsection}[2]{\hspace{4em}\oldtocsubsubsection{#1}{#2}}


\begin{document}
\title[The global global attractor for the 2D Boussinesq system]{The global well-posedness and global attractor for the solutions to the 2D Boussinesq system with variable viscosity and thermal diffusivity}
\author{Aimin Huang}
\address{The Institute for Scientific Computing and Applied Mathematics, Indiana University, 831 East Third Street, Rawles Hall, Bloomington, Indiana 47405, U.S.A.}
\email{huangepn@gmail.com}
\date{\today}

\begin{abstract}
Global well-posedness of strong solutions and existence of the global attractor to the initial and boundary value problem of 2D Boussinesq system in a periodic channel with non-homogeneous boundary conditions for the temperature and viscosity and thermal diffusivity depending on the temperature are proved.
\end{abstract}

\maketitle

\setcounter{tocdepth}{2}
\tableofcontents
\addtocontents{toc}{~\hfill\textbf{Page}\par}

\section{Introduction}\label{sec1}
We study the evolution problem of the 2D Boussinesq system governing the coupled mass and heat flow of a viscous incompressible fluid by assuming that viscosity and heat conductivity are temperature dependent. The system reads in the non-dimensional form:
\begin{equation}\begin{cases}\label{eq1.1.1}
	\partial_t\boldsymbol u - {\mathrm{div}\,}(\nu(T) \nabla \boldsymbol u) + \boldsymbol u\cdot \nabla \boldsymbol u + \nabla p = T \boldsymbol e_2,\quad \boldsymbol e_2=(0,1),\\
	{\mathrm{div}\,}\boldsymbol u=0,\\
	\partial_t T - {\mathrm{div}\,}(\kappa(T)\nabla T) + \boldsymbol u\cdot \nabla T =0,
\end{cases}\end{equation}
and the physical domain in consideration is the region
$$\Omega:=\mathbb{T}_x\times(0,1)_y,$$
where we assume the periodic boundary conditions in the $x$-direction and $\mathbb{T}_x$ is the 1-torus, with the endpoints of the interval $[0,1]$ identified; $\boldsymbol u=(u_1,u_2)$ and $T$ denote the velocity and the temperature of the fluid respectively. The viscosity $\nu$ and the thermal diffusivity $\kappa$ depend on the temperature, which is the key issue in this article.  We associate with \eqref{eq1.1.1} the following initial and boundary conditions (free-slip boundary conditions for $\boldsymbol u$ and Dirichlet boundary condition for $T$), beside the periodicity in $x$-direction of all functions:
\begin{equation}\begin{cases}\label{eq1.1.2}
	\boldsymbol u(0,x,y)=\boldsymbol u_0(x,y),\quad\quad T(0,x,y)=T_0(x,y),\\
	 u_2(t,x,y)\big|_{y=0,1}=0,\quad\quad \frac{\partial u_1}{\partial y}(t,x,y)\big|_{y=0,1}=0,\\
	T(t,x)\big|_{y=0}=1,\quad\quad\quad\quad T(t,x,y)\big|_{y=1}=0.
\end{cases}\end{equation}
Note that the boundary conditions \eqref{eq1.1.2}$_2$ can also be rewritten as
\[
\boldsymbol u\cdot \boldsymbol n = 0,\quad\quad \frac{\partial \boldsymbol u}{\partial \boldsymbol n}\times\boldsymbol n=0,
\]
where $\boldsymbol n$ is the outward unit normal vector to $\partial\Omega$.

Following a traditional approach (see e.g. \cite{Tem97, Wang07}), we recast the Boussinesq system in terms of the perturbative variable (perturbation away from the pure conduction state $(0,1-y)$); namely we set
\begin{equation}
	(\boldsymbol u,\theta)=(\boldsymbol u, T-(1-y)).
\end{equation}
In the perturbative variables, the Boussinesq system \eqref{eq1.1.1} reads
\begin{equation}\begin{cases}\label{eq1.1.3}
	\partial_t\boldsymbol u - {\mathrm{div}\,}(\nu(\theta) \nabla \boldsymbol u) + \boldsymbol u\cdot \nabla \boldsymbol u + \nabla p = \theta \boldsymbol e_2 + (1-y)\boldsymbol e_2,\quad \boldsymbol e_2=(0,1),\\
	{\mathrm{div}\,}\boldsymbol u=0,\\
	\partial_t \theta - {\mathrm{div}\,}(\kappa(\theta)\nabla \theta) + \boldsymbol u\cdot \nabla \theta - u_2 = -\kappa'(\theta)\theta_y,
\end{cases}\end{equation}
with initial and boundary conditions
\begin{equation}\begin{cases}\label{eq1.1.4}
	\boldsymbol u(0,x,y)=\boldsymbol u_0(x,y),\quad\quad \theta(0,x,y)=\theta_0(x,y):=T_0(x,y)-(1-y),\\
	u_2(t,x,y)\big|_{y=0,1}=0,\quad\quad \frac{\partial u_1}{\partial y}(t,x,y)\big|_{y=0,1}=0,\quad\quad \theta(t,x,y)\big|_{y=0,1}=0.
\end{cases}\end{equation}
We note that we already write
\[
\nu(\theta)=\nu(T),\quad\quad\kappa(\theta)=\kappa(T).
\]

\noindent\textbf{Conditions on the viscosity $\nu$ and thermal diffusivity $\kappa$.} We assume that $\nu(\cdot)$ and $\kappa(\cdot)$ are $\mathcal C^3(\mathbb R)$-functions of $\tau$ satisfying
\begin{equation}\label{asp1.1.1}
		\nu(\tau)\geq \underline\nu,\quad\kappa(\tau)\geq \underline\kappa,\quad\forall\;\tau\in\mathbb R,
\end{equation}
for some positive constants $\underline\nu,\;\underline\kappa>0$.
We remark here that the condition \eqref{asp1.1.1} is sufficient for showing the global well-posedness of the strong solutions to \eqref{eq1.1.3}-\eqref{eq1.1.4} as we will indicate in Section~\ref{sec2} (see Remark~\ref{rmk2.0.1}). However, since we are mainly interested in the global attractor in this article, we assume additional control on the growth rates of $\nu$ and $\kappa$:
\begin{equation}\label{asp1.1.2}
\begin{cases}
	\aiminabs{\nu'(\tau)},\;\aiminabs{\kappa'(\tau)}\leq c_0(\aiminabs{\tau}^r + 1),\\
	\nu(\tau),\;\kappa(\tau)\leq  c_0(\aiminabs{\tau}^{r+1}+1),	
\end{cases} \quad\forall\;\tau\in\mathbb R,
\end{equation}
for some positive constants $c_0>0$ and $r\geq 0$, and the control
\begin{equation}\label{asp1.1.3}
	\frac{ \aiminabs{\nu'(\tau)} }{\kappa(\tau)},\;\frac{\aiminabs{\kappa'(\tau)}}{\kappa(\tau)},\;\frac{\aiminabs{\kappa''(\tau)}}{\kappa(\tau)}\leq \tilde c_0,\quad\forall\;\tau\in\mathbb R,
\end{equation}
for some positive constant $\tilde c_0>0$.
We note that \eqref{asp1.1.2}$_1$ implies \eqref{asp1.1.2}$_2$ by the fundamental theorem of calculus and Young's inequality.

In the case when $\nu$ and $\kappa$ are positive constants, the global well-posedness of the 2D Boussinesq system is classical, and the existence of the global attractor and its finite dimensionality  have been studied in e.g. \cite{FMT87, MZ97}. The asymptotic behavior of the global attractors to the 2D Boussinesq system at large Prandtl Number has been showed in \cite{Wang05,Wang07}. Recently, there are many works devoted to the study of the Boussinesq system with partial viscosity (that is either $\nu>0$ and $\kappa=0$; or $\nu=0$ and $\kappa>0$), see for example \cite{HL05,Cha06, LPZ11}; and there are also many works devoted to the case when only either the horizontal viscosity or vertical viscosity is present, see for examples \cite{ACW11, CW13, MZ13}.

In some realistic applications, the variation of the fluid viscosity and thermal conductivity with temperature can not be disregarded (see for example \cite{LB96} and the references therein). The dependence on the temperature of viscosity $\nu$ and diffusivity $\kappa$ introduces a strong nonlinear coupling between the equations and the problem becomes more complicated. In the case of no-slip boundary conditions for the velocity $\boldsymbol u$, S.A. Lorca and J.L. Boldrini \cite{LB99} first proved the existence of weak solutions and also well-posedness of local strong solutions to \eqref{eq1.1.1}-\eqref{eq1.1.2} (or \eqref{eq1.1.3}-\eqref{eq1.1.4}) with general non-homogeneous boundary conditions for the temperature. Recently, in \cite{WZ11,SZ13}, the authors proved the global well-posedness result of strong solutions to the system \eqref{eq1.1.1}-\eqref{eq1.1.2} in $\mathbb R^2$ and in a bounded domain with homogeneous boundary condition for the temperature. In \cite{LPZ13}, the authors showed the global well-posedness of strong solutions in the case when only thermal diffusivity is present and $\nu=0$ and provided some decay estimates, which actually inspired this work.

In this article, we consider the free-slip boundary conditions for the velocity $\boldsymbol u$ and the non-homogeneous (physical) boundary condition for the temperature $T$ and prove a global well-posedness result as well as the existence of a global attractor for the Boussinesq system \eqref{eq1.1.3}-\eqref{eq1.1.4}. The free-slip boundary conditions are suitable for many geophysical applications where the top and bottom boundaries are somewhat artificial and hence free-slip may be more appropriate to avoid artificial boundary layer (see \cite{TS82}).
We point out that although the statements of the results in Theorems~\ref{thm1.2}-\ref{thm1.3} are quite expected, their proofs depart sharply from and are much more involved than the classical ones when the viscosity and diffusivity are positive constants, requiring extra and harder estimates (see Sections~\ref{sec2}-\ref{sec3}), and we have made great effort to simplify the proof.

The rest of this article is organized as follows.
At the end of this introduction, we introduce the functional setting and state our main results.
The main steps to prove the main results are that we first show the time-uniform a priori estimates,  then prove the continutiy of the solutions with respect to the initial data, and finally obtain the existence of the global attractor of the Boussinesq system.
In Section~\ref{sec2}, we first collect some basic facts which will be used throughout this article and then derive the maximum principle for $\theta$ and the time-uniform $L^2$, $H^1$, and $H^2$-estimates successively. We remark that the uniform $H^1$-estimate for $\theta$ is the most difficult one to obtain where we need to utilize two auxiliary functions. Section~\ref{sec3} is devoted to prove the continuous dependence of the strong solutions in $H^2$ with respect to the initial data. Finally, in Section~\ref{sec4}, the existence of the global attractor is proved, where we show the compactness result via a special technique using an Aubin-Lions compactness lemma combined with the use of the Riesz lemma and a continuity argument. This method has already been used in \cite{Ju01, Ju07}.

\subsection{Functional setting}\label{subsec1.1}
We introduce the following function spaces corresponding to the free-slip boundary conditions of $\boldsymbol u$ (see \cite{FMT87, MZ97, Zia98}):
\[
H_1=\aiminset{\boldsymbol u\in L^2(\Omega)^2,\; {\mathrm{div}\,} \boldsymbol u=0,\; u_2\big|_{y=0,1}=0},
\]
\[
V_1=\aiminset{\boldsymbol u\in H^1(\Omega)^2,\; {\mathrm{div}\,} \boldsymbol u=0,\; u_2\big|_{y=0,1}=0}.
\]
Here and in the following, the periodicity in the $x$-direction is implicitly embedded in the definition of the domain $\Omega=\mathbb T_x\times(0,1)_y$.

The Stokes operator $A_1$ is then defined as a linear unbounded operator in $H_1$ and given by
\[
A_1\boldsymbol v = -P_{H_1} \Delta\boldsymbol v,\quad\forall\,\boldsymbol v\in\mathcal D(A_1),
\]
where $P_{H_1}\, :\, L^2(M)\mapsto H_1$ is the projection operator and the domain $\mathcal D(A_1)$:
\[
\mathcal D(A_1)=\aiminset{\boldsymbol u\in H^2(\Omega)^2,\; {\mathrm{div}\,} \boldsymbol u=0,\; (u_2, \frac{\partial u_1}{\partial y})\big|_{y=0,1}=0 }.
\]
It has been shown in  \cite[Proposition~2.1]{Zia98} that in the case of free-slip boundary conditions, the Stokes operator $A_1$ coincides, on its domain, with the negative-Laplacian operator; that is
\begin{equation}\label{eq1.1.6}
	A_1\boldsymbol u=-\Delta \boldsymbol u,\quad\forall\,\boldsymbol u\in\mathcal D(A_1).
\end{equation}
For the sake of completeness, we sketch the proof of \eqref{eq1.1.6}. To prove \eqref{eq1.1.6}, we note that one only need to show that $\Delta \boldsymbol u$ belongs to $H_1$, which is equivalent to verify that
\begin{equation}\label{eq1.1.7}
\Delta u_2 =0,\quad\text{ at }y=0,1.
\end{equation}
To verify \eqref{eq1.1.7}, we first deduce from $u_2|_{y=0,1}=0$ that $u_{2,xx}|_{y=0,1}=0$ and note that
\[
u_{2,yy}= \partial_y({\mathrm{div}\,}\boldsymbol u) - \partial_x(u_{1,y}),
\]
which implies that $u_{2,yy}|_{y=0,1}=0$ thanks to the divergence free condition of $\boldsymbol u$. Therefore, we proved \eqref{eq1.1.7} and hence \eqref{eq1.1.6}.

Thanks to \eqref{eq1.1.6} and the periodicity in $x$-direction, we obtain
\begin{equation}\label{eq1.1.10}
	-\int_\Omega \nabla p\cdot \Delta \boldsymbol u{\mathrm{d}x}{\mathrm{d}y} = \int_\Omega \nabla p\cdot A_1\boldsymbol u{\mathrm{d}x}{\mathrm{d}y} = 0,\quad\forall\,u\in \mathcal D(A_1),\;p\in H^1(\Omega).
\end{equation}

\begin{remark}
	We remark that the identity \eqref{eq1.1.10} is essential for proving the uniform $H^1$-estimate for the velocity $\boldsymbol u$ in Section~\ref{subsec2.4}. This identity has also been used to prove the global strong solutions for the 3D Primitive Equations, see \cite[Section~3.3.1]{CT06}. Showing the existence of global attractor for Boussinesq system \eqref{eq1.1.3} with no-slip boundary conditions for the velocity $\boldsymbol u$ will necessitate taking into account the pressure term.
\end{remark}

Observe that the equation satisfied by $u_1$ is
	\begin{equation}\label{eq1.1.20}
	\frac{\partial u_1}{\partial t} - {\mathrm{div}\,}(\nu(\theta)\nabla u_1) + (\boldsymbol u\cdot\nabla)u_1 + p_x = 0,	
	\end{equation}
and if we integrate this equation in $\Omega$, use the periodicity in the $x$-direction and the free-slip boundary conditions for $\boldsymbol u$, we arrive at
\[
\frac{\partial}{\partial t}\int_\Omega u_1 {\mathrm{d}x}{\mathrm{d}y} =0.
\]
Hence,
\[
\int_\Omega u_1(x,y,t) {\mathrm{d}x}{\mathrm{d}y} = \int_\Omega u_1(x,y,0) {\mathrm{d}x}{\mathrm{d}y} =: K,
\]
for some constant $K\in\mathbb R$. Therefore, since this component of $u_1$ is known (fixed) during the evolution, we make a variable change $u_1'=u_1 - K$ and see that $u_1'$ satisfies the same equation \eqref{eq1.1.20} as $u_1$. From now on, dropping the prime, we still consider \eqref{eq1.1.20} as the equation for $u_1$ and assume that  $u_1$ satisfies
\[
\int_\Omega u_1{\mathrm{d}x}{\mathrm{d}y}=0.
\]
We then set
\begin{equation}
	\begin{split}
		H&=H_1 \cap\aiminset{\int_\Omega u_1{\mathrm{d}x}{\mathrm{d}y} = 0 }\times L^2(\Omega),\\
		V&=V_1\cap\aiminset{\int_\Omega u_1{\mathrm{d}x}{\mathrm{d}y} = 0 }\times H_0^1(\Omega),
	\end{split}
\end{equation}
and define the unbounded linear operator $A$ from $H$ into $H$ by setting
\[
A(\boldsymbol u, \theta)=(A_1\boldsymbol u, -\Delta \theta),\quad\forall\,(\boldsymbol u,\theta)\in \mathcal D(A),
\]
where the domain
\[
\mathcal D(A) = \mathcal D(A_1)\cap\aiminset{\int_\Omega u_1{\mathrm{d}x}{\mathrm{d}y} = 0 }\times (H_0^1(\Omega)\cap H^2(\Omega)).
\]
Note that the extra condition $\int_\Omega u_1{\mathrm{d}x}{\mathrm{d}y}=0$ for the spaces $H$, $V$ and $\mathcal D(A)$ also ensures the coercivity of the Stokes operator $A_1$ and hence the operator $A$.

Classically, $A$ is self-adjoint positive and $A^{-1}$ is a compact (self-adjoint) linear operator on $H$. Hence, we can find an orthogonal eigen-fuctions of $A$, which are complete in the spaces $H$, $V$ and $\mathcal D(A)$ with respect to the inner products $\aimininner{\cdot}{\cdot}$, $\aimininner{\nabla\cdot}{\nabla\cdot}$ and $\aimininner{A\cdot}{A\cdot}$, respectively. This guarantees us that we are able to implement the Galerkin approximation for the system \eqref{eq1.1.3}-\eqref{eq1.1.4}.

\subsection{Main results}
\begin{theorem}\label{thm1.1}
Assume that \eqref{asp1.1.1} holds and that $(\boldsymbol u_0,\theta_0)\in H$.
Then for any $t_1>0$, the Boussinesq system \eqref{eq1.1.3}-\eqref{eq1.1.4} possesses a weak solution $(\boldsymbol u,  \theta)$ satisfying
\begin{equation}
	(\boldsymbol u,\theta) \in \, \mathcal C([0,t_1], H)\cap L^2(0,t_1; V).
\end{equation}
\end{theorem}

Following \cite[Section~4]{LB99} line by line and using the standard Galerkin approximation, we are able to conclude the existence of weak solutions for system \eqref{eq1.1.3}-\eqref{eq1.1.4}, that is Theorem~\ref{thm1.1}, since the difference between no-slip boundary conditions and free-slip boundary conditions will not make any difference for the $L^2$-estimates. Actually, the a priori estimates will be easier for us than those in \cite[Section~4]{LB99} since the boundary condition for the temperature $T$ is simpler in our case and the simplification of the external force, where the external force $g$ in \cite{LB99} is non-constant while it is chosen to be a constant (the gravitational acceleration) in our case and omitted in \eqref{eq1.1.3}. We omit the proof of Theorem~\ref{thm1.1} here.

We now state the result about the global strong solutions of Boussinesq system \eqref{eq1.1.3}-\eqref{eq1.1.4}, which is proved by Galerkin approximation and the uniform estimates in Section~\ref{sec2}. Let us denote by $W$ the domain $\mathcal D(A)$, that is
\begin{equation*}\begin{split}
	 W:= \aiminset{ (\boldsymbol u,\theta)\in H^2(\Omega)^3\,:\,{\mathrm{div}\,}\boldsymbol u=0,\; (u_2,\frac{\partial u_1}{\partial y},\theta )\big|_{y=0,1}=0,\text{ and }\int_\Omega u_1{\mathrm{d}x}{\mathrm{d}y} = 0 }.
\end{split}\end{equation*}

\begin{theorem}\label{thm1.2}
Assume that \eqref{asp1.1.1}-\eqref{asp1.1.3} holds and that  $(\boldsymbol u_0,\theta_0)\in W$. Then for any $ t_1>0$,
the Boussinesq system \eqref{eq1.1.3}-\eqref{eq1.1.4} has a unique strong solution $(\boldsymbol u,  \theta)$ satisfying
\begin{equation}\begin{split}
	(\boldsymbol u,\theta) &\in \,\mathcal C([0, t_1], W)\cap L^2(0, t_1; H^3(\Omega)),\\
	(\boldsymbol u_t,\theta_t) &\in \mathcal C([0, t_1], H)\cap L^2(0, t_1; V).\\
\end{split}\end{equation}
\end{theorem}

The existence of the global attractor for the Boussinesq system \eqref{eq1.1.3}-\eqref{eq1.1.4} is given by the following Theorem.
\begin{theorem}\label{thm1.3}
	Assume that \eqref{asp1.1.1}-\eqref{asp1.1.3} holds.
	Then the solution operator $\aiminset{S(t)}_{t\geq 0}$ of the 2D Boussinesq system \eqref{eq1.1.3}-\eqref{eq1.1.4}: $S(t)(\boldsymbol u_0,\theta_0)=(\boldsymbol u(t),\theta(t))$ defines a  semigroup in the space $W$ for all $t\in\mathbb R_+$. Moreover, the following statements are valid:
	\begin{enumerate}
		\item for any $(\boldsymbol u_0,\theta_0)\in W$, $t\mapsto S(t)(\boldsymbol u_0,\theta_0)$ is a continuous function from $\mathbb R_+$ into $W$;
		\item for any fixed $t>0$, $S(t)$ is a continuous and compact map in $W$;
		\item $\aiminset{S(t)}_{t\geq 0}$ possesses a global attractor $\mathcal A$ in the space $W$. The global attractor $\mathcal A$ is compact and connected in $W$ and is the maximal bounded attractor in $W$ in the sense of the set inclusion relation; $\mathcal A$ attracts all bounded subsets in $W$ in the norm of $H^2(\Omega)$-norm.
	\end{enumerate}
\end{theorem}

\begin{remark}
	We actually do not need the assumptions \eqref{asp1.1.2}-\eqref{asp1.1.3} to prove Theorem~\ref{thm1.2}; see also Remark~\ref{rmk2.0.1}. However, as we said before, since we are mainly concerned about the global attractor, we do not provide an explicit proof of Theorem~\ref{thm1.2} when \eqref{asp1.1.2}-\eqref{asp1.1.3} are not satisfied.
\end{remark}


\section{Uniform estimates and absorbing balls}\label{sec2}

\subsection{Preliminary results}
Here and throughout this article, we will not distinguish the notations for vector and scalar function spaces whenever they are self-evident from the context. Recall that $\Omega=\mathbb{T}_x\times(0,1)_y$, and denote by $H^s(\Omega)$ the classical Sobolev space of order $s$ on $\Omega$ with periodicity in the $x$-direction, and by $L^p(\Omega)$  ($1\leq p\leq\infty$) the classical $L^p$-Lebesgue space with norm $\|\cdot\|_p$. For simplicity, we also use $\aiminnorm{\cdot}$ for the $L^2$-norm.

In the following, we denote by $C$ a positive constant which may depend on the constants $\underline\nu$, $\underline\kappa$, $c_0$, $\tilde c_0$, $r$, $\Omega$ and the constants $c_i$'s ($i=1,2,3,4$) appearing in Lemma~\ref{lem2.1.1} below, but is independent of time $t$ and of the initial data $\boldsymbol u_0$ and $\theta_0$. The constant $C$  may vary from line to line.

We now recall and prove some classical and useful inequalities which will be used frequently in this article. First, for our convenience, we will interchangeably use the following equivalent norms:
\[
\aiminnorm{\nabla f}\backsimeq \aiminnorm{f}_{H^1},\;\forall\,f\in H_0^1(\Omega);\quad\quad
\aiminnorm{\Delta f}\backsimeq \aiminnorm{f}_{H^2},\;\forall\,f\in H_0^1(\Omega)\cap H^2(\Omega),
\]
and
\[
\aiminnorm{\nabla \boldsymbol v}\backsimeq \aiminnorm{\boldsymbol v}_{H^1},\;\forall\boldsymbol v\in V_1\cap \{\int_\Omega v_1{\mathrm{d}x}{\mathrm{d}y} = 0 \};\;
\aiminnorm{\Delta \boldsymbol v}\backsimeq \aiminnorm{\boldsymbol v}_{H^2},\;\forall\boldsymbol v\in \mathcal D(A_1)\cap \{\int_\Omega v_1{\mathrm{d}x}{\mathrm{d}y} = 0 \},
\]
where $V_1$ and $\mathcal D(A_1)$ were defined in Section~\ref{subsec1.1}.

Lemmas~\ref{lem2.1.1}-\ref{lem2.1.2} below provide essentially well-known results which we recall because of their frequent use below.
\begin{lemma}\label{lem2.1.1}
There holds\footnote[1]{see e.g. \cite{Tem97}.}

\noindent Poincar\'e's inequality:
\[
\aiminnorm{f} \leq \aiminnorm{\nabla f},\quad\forall\,f\in H_0^1(\Omega);
\]

\noindent Sobolev embedding:
\[
\aiminnorm{f}_{p}\leq c_{1}(p) \aiminnorm{f}_{H^{1}},\quad\forall\,1\leq p<+\infty,\quad\forall\, f\in H^{1}(\Omega);
\]

\noindent Ladyzhenskaya's inequalities:
\begin{equation*}\begin{split}
\aiminnorm{f}_4^2&\leq c_2\aiminnorm{f}_2\aiminnorm{ f}_{H^1},\quad\forall\,f\in H^1(\Omega),\\
\aiminnorm{f}_4^2&\leq c_3\aiminnorm{f}_2\aiminnorm{\nabla f}_2,\quad  \forall\,f\in H_0^1(\Omega).
\end{split}\end{equation*}

\noindent Agmon's inequality:
\[
\aiminnorm{f}_\infty \leq c_4 \aiminnorm{f}_{H^1}^{1/2}\aiminnorm{f}_{H^2}^{1/2},\quad\forall\,f\in H^2(\Omega)\cap H_0^1(\Omega).
\]
\end{lemma}

\begin{lemma}\label{lem2.1.3}
Let $f,g\in H_0^1(\Omega)$. Then for any $\epsilon>0$:
\begin{equation}\label{eq2.0.3}
\int_\Omega \aiminabs{f}\aiminabs{g}\aiminabs{\nabla g}{\mathrm{d}x}{\mathrm{d}y} \leq \frac{27c_3^8}{64\epsilon^3} \aiminnorm{f}^2\aiminnorm{\nabla f}^2\aiminnorm{g}^2 + \epsilon\aiminnorm{\nabla g}^2.
\end{equation}
\end{lemma}
\begin{proof}
By H\"older's inequality and Ladyzhenskaya's inequality, one has
\[
\int_\Omega \aiminabs{f}\aiminabs{g}\aiminabs{\nabla g}{\mathrm{d}x}{\mathrm{d}y}
\leq \aiminnorm{f}_4\aiminnorm{g}_4\aiminnorm{\nabla g}
\leq c_3^2\aiminnorm{f}^{1/2}\aiminnorm{\nabla f}^{1/2}\aiminnorm{g}^{1/2}\aiminnorm{\nabla g}^{3/2}.
\]
We can then obtain the inequality \eqref{eq2.0.3} by using Young's inequality.
We will choose $\epsilon>0$ appropriately when applying Lemma~\ref{lem2.1.3} in the following context.
\end{proof}
\begin{lemma}\label{lem2.1.4}
Let $f,g\in H^1(\Omega)$. Then for any $\epsilon>0$:
\begin{equation*}
	\aiminnorm{fg}^2=\int_\Omega \aiminabs{f}^2\aiminabs{g}^2{\mathrm{d}x}{\mathrm{d}y}
	\leq c_2^2\aiminnorm{f}\aiminnorm{f}_{H^1}\aiminnorm{g}\aiminnorm{ g}_{H^1}
	\leq \begin{cases}
	\displaystyle	c_2^4 \aiminnorm{f}^2\aiminnorm{f}^2_{H^1} + c_2^4 \aiminnorm{g}^2\aiminnorm{ g}^2_{H^1},\\
	\displaystyle	\frac{c_2^4}{4\epsilon} \aiminnorm{f}^2\aiminnorm{f}^2_{H^1}\aiminnorm{g}^2 + \epsilon \aiminnorm{ g}_{H^1}^2.
	\end{cases}
\end{equation*}
In the case when $f,g\in H^1_0(\Omega)$, we have
\begin{equation*}
	\aiminnorm{fg}^2=\int_\Omega \aiminabs{f}^2\aiminabs{g}^2{\mathrm{d}x}{\mathrm{d}y}
	\leq \begin{cases}
	\displaystyle	c_2^4 \aiminnorm{f}^2\aiminnorm{\nabla f}^2 + c_2^4 \aiminnorm{g}^2\aiminnorm{\nabla g}^2,\\
	\displaystyle	\frac{c_2^4}{4\epsilon} \aiminnorm{f}^2\aiminnorm{\nabla f}^2\aiminnorm{g}^2 + \epsilon \aiminnorm{ \nabla g}^2.
	\end{cases}
\end{equation*}
\end{lemma}
The proof of Lemma~\ref{lem2.1.4} can also be achieved by using H\"older's inequality, Ladyzhenskaya's inequality, and Young's inequality.

\begin{lemma}\label{lem2.1.5}
	Let $f,g\in H^1(\Omega)$. Then for any $\epsilon>0$:
\[
\aiminnorm{fg}_{H^1}^2 \leq \max\big(c_1^4(4), \frac{c_2c_1^2(4)}{\epsilon}\big) \aiminnorm{f}_{H^1}^2\aiminnorm{g}_{H^1}^2\big(1+\aiminnorm{f}_{H^1}^2+\aiminnorm{g}_{H^1}^2\big) + \epsilon(\aiminnorm{f}_{H^2}^2+\aiminnorm{g}_{H^2}^2),
\]
where the constant $c_1(4)$ is the Sobolev embedding norm appearing in Lemma~\ref{lem2.1.1}.
\end{lemma}
\begin{proof}
We first write
\begin{equation}\begin{split}\label{eq2.0.5}
	\aiminnorm{fg}_{H^1}^2
	= \aiminnorm{fg}^2 + \aiminnorm{(\nabla f)g + f(\nabla g)}^2
	\leq \aiminnorm{fg}^2 + 2\aiminnorm{(\nabla f)g}^2  +2\aiminnorm{ f (\nabla g)}^2.
\end{split}\end{equation}
By H\"older's inequality and the Sobolev embedding, we find
\begin{equation*}
	\aiminnorm{fg}^2  \leq \aiminnorm{f}_4^2\aiminnorm{g}_4^2 \leq c_1^4(4) \aiminnorm{f}_{H^1}^2\aiminnorm{g}_{H^1}^2.
\end{equation*}
By H\"older's inequality, the Sobolev embedding, and Ladyzhenskaya's inequality, we find
\begin{equation*}\begin{split}
	\aiminnorm{(\nabla f)g}^2
	&\leq \aiminnorm{\nabla f}_4^2\aiminnorm{g}_4^2
	\leq c_2c_1^2(4) \aiminnorm{\nabla f}\aiminnorm{\nabla f}_{H^1}  \aiminnorm{g}_{H^1}^2\\
	&\leq \frac{c_2c_1^2(4)}{2\epsilon} \aiminnorm{\nabla f}^2  \aiminnorm{g}_{H^1}^4 + \frac\epsilon2 \aiminnorm{\nabla f}_{H^1}^2\\
	&\leq \frac{c_2c_1^2(4)}{2\epsilon} \aiminnorm{f}_{H^1}^2  \aiminnorm{g}_{H^1}^4 + \frac\epsilon2 \aiminnorm{f}_{H^2}^2;\\
\end{split}\end{equation*}
similarly,
\begin{equation*}
	\aiminnorm{ f(\nabla g)}^2
	\leq  \frac{c_2c_1^2(4)}{2\epsilon} \aiminnorm{ f}_{H^1}^4  \aiminnorm{g}_{H^1}^2 + \frac\epsilon2 \aiminnorm{g}_{H^2}^2.
\end{equation*}
Inserting these estimates into \eqref{eq2.0.5} readily yields the result. This ends the proof of Lemma~\ref{lem2.1.5}.
\end{proof}

\begin{lemma}\label{lem2.1.6}
	Let $f,g,h\in H^1(\Omega)$. Then for any $\epsilon>0$:
	\[
	\int_\Omega\aiminabs{f}\aiminabs{g}\aiminabs{h}{\mathrm{d}x}{\mathrm{d}y} \leq \frac{c_2^4}{8\epsilon}\big( \aiminnorm{f}^{2}\aiminnorm{f}_{H^1}^{2}+\aiminnorm{g}^{2}\aiminnorm{g}_{H^1}^{2}\big) + \epsilon\aiminnorm{h}^2.
	\]
\end{lemma}
\begin{proof}
By H\"older's inequality, Ladyzhenskaya's inequality, and Young's inequality, we find
\begin{equation*}\begin{split}
	\int_\Omega\aiminabs{f}\aiminabs{g}\aiminabs{h}{\mathrm{d}x}{\mathrm{d}y} \leq \aiminnorm{f}_4\aiminnorm{g}_4\aiminnorm{h}
	&\leq c_2^2\aiminnorm{f}^{1/2}\aiminnorm{f}_{H^1}^{1/2}\aiminnorm{g}^{1/2}\aiminnorm{g}_{H^1}^{1/2}\aiminnorm{h}\\
	&\leq \frac{c_2^4}{8\epsilon}\big( \aiminnorm{f}^{2}\aiminnorm{f}_{H^1}^{2}+\aiminnorm{g}^{2}\aiminnorm{g}_{H^1}^{2}\big) + \epsilon\aiminnorm{h}^2.
\end{split}\end{equation*}
This completes the proof of Lemma~\ref{lem2.1.6}.
\end{proof}

Finally, we recall the Uniform Gronwall Lemma which is the key to prove the uniform estimates and is used extensively in this section. One can refer to \cite[pp. 91]{Tem97} and \cite{FP67} for its proof.
\begin{lemma}[Uniform Gronwall Lemma]\label{lem2.1.2}
Let $g,h$ and $y$ be three non-negative locally integrable functions on $(t_0,+\infty)$ such that
\[
\frac{{\mathrm{d}y}}{{\mathrm{d}t}}\leq gy+h,\quad\forall\,t\geq t_0,
\]
and
\[
\int_t^{t+\gamma}g(s)\,\text{d}s \leq a_1,\quad \int_t^{t+\gamma}h(s)\,\text{d}s \leq a_2,\quad \int_t^{t+\gamma}y(s)\,\text{d}s \leq a_3,\quad\forall\,t\geq t_0,
\]
where $\gamma$, $a_1$, $a_2$ and $a_3$ are positive constants. Then
\[
y(t)\leq \big(\frac{a_3}{\gamma}+a_2)e^{a_1},\quad\forall\,t\geq t_0+\gamma.
\]
\end{lemma}

\subsection{Maximum principle}
We need a variant of the maximum principle for $\theta$ and for this purpose we introduce the truncation operators that associate with a function $\psi$, the functions $\psi_+$ and $\psi_-$:
\[
\psi_+(x,y) = \max(\psi(x,y), 0),\quad\quad \psi_-(x,y)=\max(-\psi(x,y), 0).
\]
\begin{proposition}\label{prop2.2.1}
For $p\geq 2$, let $\theta\in L^\infty(0,  t_1; L^p(\Omega))\cap L^2(0, t_1; H_0^1(\Omega))$ and $u\in L^2(0, t_1; V_1)$  be a weak solution of \eqref{eq1.1.3}-\eqref{eq1.1.4}.
We additionally assume that
\begin{equation}\label{eq2.1.4}
	-1\leq \theta_0(x,y)\leq 1,\quad \text{a.e. } (x,y)\in\Omega.
\end{equation}
Then
\begin{equation}\label{eq2.1.5}
	-1\leq \theta(t,x,y) \leq 1,\quad \text{a.e. }(x,y)\in\Omega,\text{ a.e. }t\geq 0.
\end{equation}

If $\aiminset{\boldsymbol u,\theta}$ are defined for all $t>0$ and \eqref{eq2.1.4} is not assumed, then
\begin{equation}\label{eq2.1.5-6}
	\theta(t,x,y)=\tilde\theta(t,x,y) + \bar\theta(t,x,y),
\end{equation}
where $-1\leq\tilde\theta(t,x,y)\leq 1$ a.e., and
\begin{equation}\label{eq2.1.6}
	\aiminnorm{\bar\theta(t)}_p \leq \big( \aiminnorm{(\theta-1)_+(0)}_p + \aiminnorm{(\theta+1)_-(0)}_p \big) \exp\big( -\frac{4(p-1)}{p^2}\underline\kappa t\big).
\end{equation}
\end{proposition}
\begin{proof}
Although the arguments below follow line by line as in \cite[pp. 136 - 137]{Tem97} and it is useful to briefly recall them here because of the dependence of $\kappa$ on $T$. The results are proved naturally on the equation for $T$, that is \eqref{eq1.1.1}$_3$. In terms of $T$, \eqref{eq2.1.4} and \eqref{eq2.1.5} amount to
\begin{equation}\label{eq2.1.7}
\begin{split}
	&0\leq T(0,x,y)\leq 1,\quad\text{a.e. }(x,y)\in\Omega,\\
	&0\leq T(t,x,y)\leq 1,\quad\text{a.e. }(x,y)\in\Omega,\quad\text{a.e. }t.
\end{split}
\end{equation}
In order to establish the second inequality in \eqref{eq2.1.7}$_2$, we observe that $(T-1)_+$ belongs to $L^2(0, t_1; H_0^1(\Omega))$; hence, we multiply \eqref{eq1.1.1}$_3$ by $(T-1)_+\aiminabs{(T-1)_+}^{p-2}$ and integrate over $\Omega$; we arrive, after utilization of Green's formula, at
\begin{equation}\label{eq2.1.8}
	\frac 1 p\frac{\mathrm{d}}{\mathrm{d}t} \aiminnorm{(T-1)_+}_p^p + \int_\Omega \kappa(T)(p-1) \aiminabs{(T-1)_+}^{p-2} \aiminabs{\nabla ((T-1)_+)}^2 {\mathrm{d}x}{\mathrm{d}y} = 0.
\end{equation}
We set $g=\aiminabs{(T-1)_+}^{p/2-1}(T-1)_+$; then by direct computations, we find
\[
\aiminnorm{g}^2 = \aiminnorm{(T-1)_+}_p^p,\quad\quad \nabla g=\frac p 2\aiminabs{(T-1)_+}^{p/2-1}\nabla ((T-1)_+).
\]
Using the assumption $\kappa(T)\geq \underline\kappa$, we infer from \eqref{eq2.1.8} that
\[
\frac{\mathrm{d}}{\mathrm{d}t} \aiminnorm{g}^2 + \underline\kappa\frac{4(p-1)}{p}	\aiminnorm{\nabla g}^2\leq 0,
\]
which, by using the Poincar\'e's inequality, implies that
\begin{equation}\label{eq2.1.9}
	\frac{\mathrm{d}}{\mathrm{d}t} \aiminnorm{g}^2 + \underline\kappa\frac{4(p-1)}{p}	\aiminnorm{g}^2\leq 0.
\end{equation}
From \eqref{eq2.1.9}, we observe that $\aiminnorm{g(t)}$ is a decreasing function of $t$ that vanishes at $t=0$ and, therefore, it vanishes for all later time $t>0$; thus $T(t,x,y)\leq 1$ for all $t\geq 0$. For the proof of the first inequality in \eqref{eq2.1.7}$_2$, we consider $T_-$ and proceed similarly.

If we do not assume \eqref{eq2.1.4}, we conclude from \eqref{eq2.1.9} that $\aiminnorm{g(t)}$ decreases exponentially,
\[
\aiminnorm{g(t)}^2 \leq \aiminnorm{g(0)}^2 \exp\big( -\frac{4(p-1)}{p}\underline\kappa t\big),
\]
which is equivalent to
\begin{equation}
	\aiminnorm{ (T-1)_+(t)}_p \leq \aiminnorm{ (T-1)_+(0)}_p \exp\big( -\frac{4(p-1)}{p^2}\underline\kappa t\big).
\end{equation}
Similarly, we can prove that
\begin{equation}
	\aiminnorm{ T_-(t)}_p \leq \aiminnorm{ T_-(0)}_p \exp\big( -\frac{4(p-1)}{p^2}\underline\kappa t\big).
\end{equation}
We then set
\[
T=\tilde T + \bar T,\quad \quad \bar T=(T-1)_+- T_-,
\]
and we see that $0\leq \tilde T\leq 1$ almost everywhere and $\bar T(t,\cdot)\rightarrow 0$ in $L^p(\Omega)$ as $t\rightarrow \infty$:
\begin{equation}
	\aiminnorm{\bar T(t)}_p \leq \big( \aiminnorm{ (T-1)_+(0)}_p + \aiminnorm{ T_-(0)}_p\big) \exp\big( -\frac{4(p-1)}{p^2}\underline\kappa t\big).
\end{equation}
Then \eqref{eq2.1.5-6} and \eqref{eq2.1.6} are just a rephrasing of the results in terms of $\theta$.
\end{proof}

\begin{remark}\label{rmk2.0.1}
	By taking $p\rightarrow+\infty$ in \eqref{eq2.1.6}, we immediately find that
	\[
	\aiminnorm{\bar\theta(t)}_\infty \leq  \aiminnorm{(\theta-1)_+(0)}_\infty + \aiminnorm{(\theta+1)_-(0)}_\infty
	\leq 2\aiminnorm{\theta_0}_\infty + 2.
	\]
	Therefore
	\begin{equation}\label{eq2.1.11}
		\aiminnorm{\theta}_{L^\infty( (0,T)\times\Omega )} \leq \aiminnorm{\tilde\theta}_{L^\infty( (0,T)\times\Omega )} + \aiminnorm{\bar\theta}_{L^\infty( (0,T)\times\Omega )}
		\leq 2\aiminnorm{\theta_0}_\infty + 3.
	\end{equation}
	We remark that the estimate \eqref{eq2.1.11} yields the upper bounds for $\nu(\theta)$ and $\kappa(\theta)$ and also their derivatives. These bounds will ensure us to obtain the global well-posedness of the strong solutions. But as indicated in the introduction, since we are more interested in the global attractor, the dependence on the initial data $\theta_0$ for the estimate \eqref{eq2.1.11} does not suit our goal. To continue towards our objective, we need more delicate estimates as we will show below as well as the assumptions \eqref{asp1.1.1}-\eqref{asp1.1.3}.
\end{remark}

\subsection{$L^p$-estimates}
The uniform $L^p$-estimate for $\theta$ and $L^2$-estimate for $\boldsymbol u$ follow the similar arguments as in \cite[pp. 137 - 138]{Tem97} and we will briefly explain it below. However, at this stage, we can not prove the uniform $L^p$-estimate for $\boldsymbol u$ because of the pressure term in the velocity equation.

\begin{proposition}[Existence of absorbing balls in $L^2$ and $L^p$]\label{prop2.3.1}
Under the assumptions of Theorem~\ref{thm1.3}, there exists $ t_0>0$  depending on the initial data $\boldsymbol u_0$ and $\theta_0$ such that
\begin{equation}\label{eq2.2.8}
	\aiminnorm{\boldsymbol u(t)},\;\aiminnorm{\theta(t)}_p \leq C_0 := \frac{1}{\underline\nu} + \frac{\aiminabs{\Omega}^{1/2}}{\underline\nu}+\aiminabs{\Omega}^{1/p}+1,\quad\forall\,t\geq  t_0,
\end{equation}
for all $p$ satisfying $2\leq p\leq p_0<+\infty$ for some $p_0$ large enough (for example, we could take $p_0=10r+10$ where $r$ is the one in assumption \eqref{asp1.1.2}).
\end{proposition}
\begin{proof}
	For any fixed $p\geq 2$, Proposition~\ref{prop2.2.1} already provides an uniform estimate for $\aiminnorm{\theta(t)}_p$:
\begin{equation}\begin{split}\label{eq2.2.1}
	\aiminnorm{\theta(t)}_p &\leq \aiminnorm{\tilde\theta(t)}_p + \aiminnorm{\bar\theta(t)}_p \\
	&\leq \aiminabs{\Omega}^{1/p} + \big( \aiminnorm{(\theta-1)_+(0)}_p + \aiminnorm{(\theta+1)_-(0)}_p \big)\exp\big( -\frac{4(p-1)}{p^2}\underline\kappa t\big).
\end{split}\end{equation}
Hence, we have
\begin{equation}\label{eq2.2.3}
	\limsup_{t\rightarrow\infty} \aiminnorm{\theta(t)}_p \leq \aiminabs{\Omega}^{1/p}.
\end{equation}
This gives us the desired uniform estimate which yields an absorbing ball for $\theta$ in $L^p(\Omega)$.

Taking the inner product of \eqref{eq1.1.3}$_1$ with $\boldsymbol u$ in $L^2(\Omega)$, we classically find the energy estimate
\begin{equation}\label{eq2.2.6}
	\frac12\frac{{\mathrm{d}}}{{\mathrm{d}t}}\aiminnorm{\boldsymbol u}^2  + \int_\Omega\nu(\theta)\aiminabs{\nabla\boldsymbol u}^2{\mathrm{d}x}{\mathrm{d}y}  \leq \aiminnorm{\theta}\aiminnorm{\boldsymbol u} + \aiminnorm{\boldsymbol u},
\end{equation}
which, by using the assumption \eqref{asp1.1.1} and Poincar\'e's inequality, implies that
\[
	\frac12\frac{{\mathrm{d}}}{{\mathrm{d}t}}\aiminnorm{\boldsymbol u}^2  + \underline\nu\aiminnorm{\boldsymbol u}^2 \leq \aiminnorm{\theta}\aiminnorm{\boldsymbol u} + \aiminnorm{\boldsymbol u},
\]
that is
\[
	\frac{{\mathrm{d}}}{{\mathrm{d}t}}\aiminnorm{\boldsymbol u}  + \underline\nu\aiminnorm{\boldsymbol u} \leq \aiminnorm{\theta} + 1.
\]
Therefore, by a direct integration, we find that for all $t\geq 0$:
\begin{equation}\label{eq2.2.7}
\begin{split}
	\aiminnorm{\boldsymbol u(t)} &\leq e^{-\underline\nu t}\aiminnorm{\boldsymbol u_0} + \int_0^t e^{\underline\nu (s-t)}(\aiminnorm{\theta(s)}+1){\mathrm{d}s} \\
	&\leq	e^{-\underline\nu t}\aiminnorm{\boldsymbol u_0} +\big(\frac{1}{\underline\nu} + \frac{\aiminabs{\Omega}^{1/2}}{\underline\nu}\big)(1-e^{-\underline\nu t}) \\
	&\quad +\big( \aiminnorm{(\theta-1)_+(0)}+ \aiminnorm{(\theta+1)_-(0)} \big)\frac{e^{-\underline\kappa t} - e^{-\underline\nu t}}{\underline\nu-\underline\kappa},
\end{split}
\end{equation}
where we have used \eqref{eq2.2.1} for $\aiminnorm{\theta(s)}$ with $p=2$ for the last inequality. Note that in the case when $\underline\nu=\underline\kappa$, the last term in the right-hand side of \eqref{eq2.2.7} should be viewed as a limit when $\underline\nu$ approaches $\underline\kappa$, that is $te^{-\underline\nu t}$. Finally, we find from \eqref{eq2.2.7} that
\begin{equation}\label{eq2.2.8p}
	\limsup_{t\rightarrow \infty}\aiminnorm{\boldsymbol u(t)} \leq \frac{1}{\underline\nu} + \frac{\aiminabs{\Omega}^{1/2}}{\underline\nu}.
\end{equation}
This gives us the desired uniform estimate which yields an absorbing ball for $\boldsymbol u$ in $L^2(\Omega)$.
In conclusion, we obtain \eqref{eq2.2.8} from the estimates \eqref{eq2.2.3} and \eqref{eq2.2.8p}.
\end{proof}

As a consequence of Proposition~\ref{prop2.3.1}, we are going to obtain the control on the time averages of $\aiminnorm{\nabla \boldsymbol u}^2$ and $\aiminnorm{\nabla\theta}^2$.
Integrating \eqref{eq2.2.6} in time on $(t, t+1)$, we find
\[
\int_t^{t+1} \aiminnorm{\sqrt{\nu(\theta)}\nabla\boldsymbol u(s)}^2{\mathrm{d}s} \leq \frac12\aiminnorm{\boldsymbol u(t)}^2 + \int_t^{t+1}\aiminnorm{\theta(s)}\aiminnorm{\boldsymbol u(s)}{\mathrm{d}s} + \int_t^{t+1} \aiminnorm{\boldsymbol u(s)}{\mathrm{d}s},
\]
and by employing the uniform estimates \eqref{eq2.2.8}, we obtain
\begin{equation}\begin{split}\label{eq2.2.10p}
	\int_t^{t+1} \aiminnorm{\sqrt{\nu(\theta)}\nabla\boldsymbol u(s)}^2{\mathrm{d}s} \leq \frac{C_0}{2}\big(3C_0+2\big),\quad\forall\,t\geq  t_0.
\end{split}\end{equation}
In particular, we also have
\begin{equation}\label{eq2.2.10}
	\int_t^{t+1} \aiminnorm{\nabla\boldsymbol u(s)}^2{\mathrm{d}s} \leq \frac{C_0}{2\underline\nu}\big(3C_0+2\big),\quad\forall\,t\geq  t_0.
\end{equation}

Multiplying \eqref{eq1.1.3}$_3$ by $\aiminabs{\theta}^{p-2}\theta$, integrating over $\Omega$, and using the assumptions \eqref{asp1.1.1}-\eqref{asp1.1.2}$_1$, we obtain
\begin{equation*}\begin{split}
	\frac1p\frac{{\mathrm{d}}}{{\mathrm{d}t}}\aiminnorm{\theta}_p^p + &\underline\kappa (p-1)\int_\Omega \aiminabs{\theta}^{p-2}\aiminabs{ \nabla\theta}^2 {\mathrm{d}x}{\mathrm{d}y}\\
	&\leq \aiminnorm{u_2}\aiminnorm{\aiminabs{\theta}^{p-1}} + c_0\aiminnorm{\aiminabs{\theta}^{p/2-1}\aiminabs{\theta_y}}\big(\aiminnorm{\aiminabs{\theta}^{p/2}} + \aiminnorm{\aiminabs{\theta}^{p/2+r}}\big)\\
	&\leq \aiminnorm{\boldsymbol u}\aiminnorm{\theta}_{2p-2}^{p-1} + \frac{\underline\kappa}{4}\aiminnorm{\aiminabs{\theta}^{p/2-1}\nabla\theta}^2 + \frac{c_0^2}{\underline\kappa}\big(\aiminnorm{\theta}_p^{p/2} + \aiminnorm{\theta}_{p+2r}^{p/2+r}\big)^2,
\end{split}\end{equation*}
which, after integrating in time on $(t,t+1)$ and using the uniform estimates \eqref{eq2.2.8}, implies the control of the time average of $\aiminnorm{\aiminabs{\theta}^{p/2-1}\nabla\theta}^2$:
\begin{equation}\label{eq2.2.12}
	\int_t^{t+1}\int_\Omega \aiminabs{\theta}^{p-2}\aiminabs{ \nabla\theta}^2 {\mathrm{d}x}{\mathrm{d}y}{\mathrm{d}s} \leq \frac{2}{\underline\kappa(p-1)}\bigg[\big(\frac1p+1\big)C_0^p + \frac{c_0^2}{\underline\kappa}\big(C_0^{p/2} + C_0^{p/2+r}\big)^2\bigg],\;\;\forall\,t\geq  t_0.
\end{equation}
In the particular case when $p=2$, we have
\begin{equation}\label{eq2.2.5}
	\int_t^{t+1} \aiminnorm{\nabla \theta(s)}^2{\mathrm{d}s} \leq \frac{3C_0^2}{\underline\kappa} + \frac{2c_0^2}{\underline\kappa^2}(C_0+C_0^{r+1})^2,\quad\forall\,t\geq  t_0.
\end{equation}

\subsection{$H^1$-estimates}\label{subsec2.4} We now turn to proving the uniform estimates for $(\boldsymbol u,\theta)$ in $H^1(\Omega)$. In order to prove the estimates for $\theta$, we need to  introduce two auxiliary quantities
\begin{equation}\label{eq2.3.1}
	\breve\theta = \int_0^\theta \sqrt{\kappa(\tau) }\text{d}\tau,\quad\quad
\hat\theta = \int_0^\theta \kappa(\tau)\,\text{d}\tau.
\end{equation}
By explicit calculations, we obtain that $\breve\theta$ and $\hat\theta$ satisfy the following equations
\begin{equation}\label{eq2.3.2}
\begin{split}
	&\breve\theta_t -\sqrt{\kappa(\theta)}{\mathrm{div}\,}(\sqrt{\kappa(\theta)}\nabla\breve\theta) + \boldsymbol u\cdot\nabla\breve\theta - \sqrt{\kappa(\theta)}u_2 = -\kappa'(\theta)\breve\theta_y,\\
	&\hat\theta_t -\kappa(\theta)\Delta\hat\theta + \boldsymbol u\cdot \nabla\hat\theta - \kappa(\theta)u_2 = -\kappa'(\theta)\hat\theta_y,
\end{split}
\end{equation}
with the following homogeneous boundary conditions
\begin{equation}\label{eq2.3.3}
	\breve\theta(t,x,y)=\hat\theta(t,x,y)=0,\quad\text{ on }\partial\Omega.
\end{equation}
By direct calculations again, we immediately find the following relations
\begin{equation}\label{eq2.3.30.1}
	\nabla\breve\theta=\sqrt{\kappa(\theta)}\nabla\theta,\quad\quad \breve\theta_t=\sqrt{\kappa(\theta)}\theta_t,\quad\quad
	\nabla\hat\theta=\kappa(\theta)\nabla\theta,\quad\quad\hat\theta_t=\kappa(\theta)\theta_t.
\end{equation}
\begin{equation}\label{eq2.3.30.2}
	\sqrt{\underline\kappa}\aiminnorm{\nabla\theta}\leq \aiminnorm{\nabla\breve\theta},\quad\quad
	\sqrt{\underline\kappa}\aiminnorm{\theta_t}\leq \aiminnorm{\breve\theta_t},\quad\quad
	\underline\kappa\aiminnorm{\nabla\theta}\leq \aiminnorm{\nabla\hat\theta},\quad\quad
	\underline\kappa\aiminnorm{\theta_t}\leq \aiminnorm{\hat\theta_t}.
\end{equation}

In order to derive the uniform $H^1$-estimates of $\theta$, we actually first derive uniform $H^1$-estimates of $\hat\theta$ and then extend the estimates to $\theta$. For that reason, we need a refined version of \eqref{eq2.2.5}, which would provide us a uniform control on the time average of $\aiminnorm{\nabla\hat\theta}^2$ (see Lemma~\ref{lem2.4.0} below), and we prove the result through $\breve\theta$. Recall that here and in the following, $C$ denotes a positive constant, independent of the time $t$ and initial data $\boldsymbol u_0$ and $\theta_0$, that may vary from line to line.
\begin{lemma}\label{lem2.4.0}
	Under the assumptions of Theorem~\ref{thm1.3}, we have for all $t\geq t_0$:
\begin{equation}\label{eq2.3.35}
	\int_t^{t+1} \aiminnorm{ \nabla\hat\theta(s)}^2{\mathrm{d}s} = \int_t^{t+1}\aiminnorm{\kappa(\theta)\nabla\theta(s)}^2{\mathrm{d}s} =
	\int_t^{t+1}\aiminnorm{ \sqrt{\kappa(\theta)}\nabla\breve\theta(s)}^2{\mathrm{d}s}
	\leq C.
\end{equation}
\end{lemma}
\begin{proof}
As a preliminary,
we utilize the Cauchy-Schwarz inequality and \eqref{asp1.1.2}$_2$ to obtain
\begin{equation}\label{eq2.3.31}
	\aiminabs{\breve\theta}^2\leq \big(\int_0^{\aiminabs{\theta}} \sqrt{\kappa(\tau)}\text{d}\tau\big)^2
	\leq \int_0^{\aiminabs{\theta}} 1 \text{d}\tau\int_0^{\aiminabs{\theta}} c_0(\aiminabs{\tau}^{r+1}+1) \text{d}\tau
	\leq c_0\big(\frac{\aiminabs{\theta}^{r+3}}{r+2} + \aiminabs{\theta}^2),
\end{equation}
which immediately yields that
\begin{equation}\label{eq2.3.31p}
	\aiminnorm{\breve\theta}^2 \leq c_0 \big(\frac1{r+2} \aiminnorm{\theta}_{r+3}^{r+3} + \aiminnorm{\theta}^2\big).
\end{equation}

Now, taking the inner product of \eqref{eq2.3.2}$_1$ with $\breve\theta$ in $L^2(\Omega)$ and integrating by parts, we obtain
\begin{equation}\begin{split}\label{eq2.3.32}
	\frac12\frac{\mathrm{d}}{\mathrm{d}t}\aiminnorm{\breve\theta}^2 &+ \int_\Omega \kappa(\theta)\aiminabs{\nabla\breve\theta}^2 {\mathrm{d}x}{\mathrm{d}y} \\
	&= - \frac12\int_\Omega \nabla\breve\theta\cdot \kappa'(\theta) \nabla\theta\cdot\breve\theta{\mathrm{d}x}{\mathrm{d}y} + \int_\Omega u_2 \sqrt{\kappa(\theta)}\breve\theta{\mathrm{d}x}{\mathrm{d}y} - \int_\Omega \kappa'(\theta)\breve\theta_y \breve\theta{\mathrm{d}x}{\mathrm{d}y}\\
	&=:I_1 + I_2 + I_3.
\end{split}\end{equation}
We estimate the right-hand side of \eqref{eq2.3.32} term by term. Using \eqref{asp1.1.2}$_1$, \eqref{eq2.3.31} and Young's inequality, we estimate $I_1$ as
\begin{equation*}
	\begin{split}
		I_1 &\leq C\int_\Omega \aiminabs{\nabla\breve\theta} \aiminabs{\nabla\theta}(\aiminabs{\theta}^{r+3}+\aiminabs{\theta}^2)^{1/2}(\aiminabs{\theta}^r+1){\mathrm{d}x}{\mathrm{d}y}\\
		&\leq \int_\Omega\bigg[\frac{\underline\kappa}4\aiminabs{\nabla\breve\theta}^2 + C\aiminabs{\nabla\theta}^2 (\aiminabs{\theta}^{r+3}+\aiminabs{\theta}^2)(\aiminabs{\theta}^r+1)^2\bigg]{\mathrm{d}x}{\mathrm{d}y}\\
		&\leq \frac{\underline\kappa}4 \aiminnorm{\nabla\breve\theta}^2 + C\int_\Omega (\aiminabs{\theta}^{3r+3} +1 )\aiminabs{\nabla\theta}^2{\mathrm{d}x}{\mathrm{d}y}.
	\end{split}
\end{equation*}
Using the Cauchy-Schwarz inequality, Young's inequality and \eqref{asp1.1.2}$_2$,\eqref{eq2.3.31}, $I_2$ is estimated as
\begin{equation*}
	\begin{split}
		I_2 &\leq 2\int_\Omega\big[\aiminabs{u_2}^2 + \kappa(\theta)\aiminabs{\breve\theta}^2\big]{\mathrm{d}x}{\mathrm{d}y}\\
		&\leq 2\aiminnorm{\boldsymbol u}^2 + C\int_\Omega (\aiminabs{\theta}^{r+1}+1)( \aiminabs{\theta}^{r+3}+\aiminabs{\theta}^2 ){\mathrm{d}x}{\mathrm{d}y}\\
		&\leq 2\aiminnorm{\boldsymbol u}^2 +C( \aiminnorm{\theta}_{2r+4}^{2r+4} + \aiminnorm{\theta}^2 ).
	\end{split}
\end{equation*}
The last term $I_3$ is easier and simpler than $I_1$ and we actually have
\begin{equation*}\begin{split}
	I_3 \leq \frac{\underline\kappa}4 \aiminnorm{\nabla\breve\theta}^2 + C\int_\Omega (\aiminabs{\theta}^{3r+3} +1 ){\mathrm{d}x}{\mathrm{d}y}\leq \frac{\underline\kappa}4 \aiminnorm{\nabla\breve\theta}^2 + C (\aiminnorm{\theta}_{3r+3}^{3r+3} + 1).
\end{split}\end{equation*}

Inserting these estimates for $I_1,\,I_2$ and $I_3$ into \eqref{eq2.3.32} implies that
\begin{equation}\begin{split}\label{eq2.3.34}
	\frac{\mathrm{d}}{\mathrm{d}t}\aiminnorm{\breve\theta}^2 + \int_\Omega \kappa(\theta)\aiminabs{\nabla\breve\theta}^2 {\mathrm{d}x}{\mathrm{d}y}
	\leq & C\bigg[\int_\Omega (\aiminabs{\theta}^{3r+3} +1 )\aiminabs{\nabla\theta}^2{\mathrm{d}x}{\mathrm{d}y}  + \aiminnorm{\boldsymbol u}^2\\
	&\quad +  (\aiminnorm{\theta}_{2r+4}^{2r+4} + \aiminnorm{\theta}_{3r+3}^{3r+3} + \aiminnorm{\theta}^2+ 1)\bigg].
\end{split}\end{equation}
We now integrate \eqref{eq2.3.34} in time on $(t,t+1)$, combine the estimate \eqref{eq2.3.31p}, and employ the uniform estimates \eqref{eq2.2.8},\,\eqref{eq2.2.12}, and \eqref{eq2.2.5}; we arrive at
\[
\int_t^{t+1}\aiminnorm{ \sqrt{\kappa(\theta)}\nabla\breve\theta(s)}^2{\mathrm{d}s}
	\leq C,\quad\forall\, t\geq t_0,
\]
which implies \eqref{eq2.3.35} by the relation \eqref{eq2.3.30.1}. We thus completed the proof of Lemma~\ref{lem2.4.0}.
\end{proof}

It is convenient to observe the following result.
\begin{lemma}\label{lem2.4.1}
	For any fixed $q\geq 1$, we have
	\begin{equation}
	\aiminnorm{\kappa'(\theta)}_q,\;\aiminnorm{\kappa(\theta)}_q \leq C,\quad\forall\;t\geq  t_0.
	\end{equation}
\end{lemma}
The proof of Lemma~\ref{lem2.4.1} directly follows from the uniform estimate \eqref{eq2.2.8} and the assumption \eqref{asp1.1.2}.

With the uniform estimate \eqref{eq2.3.35} in hand, we are now able to derive the uniform $H^1$-estimate of $\hat\theta$ and hence of $\theta$.
\begin{proposition}\label{prop2.4.1}
		Under the assumptions of Theorem~\ref{thm1.3}, we have
\begin{equation}\label{eq2.3.8}
	\aiminnorm{\nabla\hat\theta(t)},\;\aiminnorm{\nabla\theta(t)} \leq C,\quad\forall\,t\geq  t_0+1,
\end{equation}
and
\begin{equation}\label{eq2.3.9}
	\int_t^{t+1}\aiminnorm{\Delta\hat\theta(s)}^2{\mathrm{d}s} \leq C,\quad\forall\,t\geq  t_0+1.	
\end{equation}
\end{proposition}
\begin{proof}
Taking the inner product of \eqref{eq2.3.2}$_2$ with $-\Delta\hat\theta$ in $L^2(\Omega)$ and using the Cauchy-Schwarz inequality and H\"older's inequality, we find that
\begin{equation}\begin{split}\label{eq2.3.4}
	\frac12\frac{\mathrm{d}}{\mathrm{d}t}\aiminnorm{\nabla\hat\theta}^2 + \underline\kappa\aiminnorm{\Delta\hat\theta}^2 &\leq \int_\Omega \big[\aiminabs{\boldsymbol u\cdot\nabla\hat\theta}\aiminabs{\Delta\hat\theta}+  \aiminabs{\kappa(\theta) u_2}\aiminabs{\Delta\hat\theta}+ \aiminabs{\kappa'(\theta)\hat\theta_y}\aiminabs{\Delta\hat\theta}\big]{\mathrm{d}x}{\mathrm{d}y}\\
	\leq &\frac{2\underline\kappa}{8}\aiminnorm{\Delta\hat\theta}^2 + \frac2{\underline\kappa}\aiminnorm{\boldsymbol u\cdot\nabla\hat\theta}^2 + \frac2{\underline\kappa}\aiminnorm{\kappa(\theta)\boldsymbol u}^2 + \aiminnorm{\kappa'(\theta)}_4\aiminnorm{\nabla\hat\theta}_4\aiminnorm{\Delta\hat\theta}.
\end{split}\end{equation}
We are now dealing with the last three terms in the right-hand side of \eqref{eq2.3.4} successively.
Applying Lemma~\ref{lem2.1.4} with $f=\boldsymbol u$ and $g=\nabla\hat\theta$, we obtain that
\begin{equation}\begin{split}\label{eq2.3.5.1}
	\aiminnorm{\boldsymbol u\cdot\nabla\hat\theta}^2 \leq
	  C\aiminnorm{\boldsymbol u}^2\aiminnorm{\nabla\boldsymbol u}^2\aiminnorm{\nabla\hat\theta}^2 + \frac{\underline\kappa^2}{16}\aiminnorm{\Delta\hat\theta}^2;
\end{split}\end{equation}
using H\"older's inequality and the Sobolev embedding, we have
\begin{equation}\label{eq2.3.5.2}
	\aiminnorm{\kappa(\theta)\boldsymbol u}^2 \leq \aiminnorm{\kappa(\theta)}_4^2\aiminnorm{\boldsymbol u}_4^2
	\leq C\aiminnorm{\kappa(\theta)}_4^2\aiminnorm{\nabla\boldsymbol u}^2;
\end{equation}
using Ladyzhenskaya's inequality and Young's inequality, we find that
\begin{equation}\begin{split}\label{eq2.3.5.3}
	\aiminnorm{\kappa'(\theta)}_4\aiminnorm{\nabla\hat\theta}_4\aiminnorm{\Delta\hat\theta}
	\leq C\aiminnorm{\kappa'(\theta)}_4\aiminnorm{\nabla\hat\theta}^{1/2}\aiminnorm{\Delta\hat\theta}^{3/2}
	\leq  C\aiminnorm{\kappa'(\theta)}^4_4\aiminnorm{\nabla\hat\theta}^2 + \frac{\underline\kappa}{8}\aiminnorm{\Delta\hat\theta}^2.
\end{split}\end{equation}

Inserting the three estimates \eqref{eq2.3.5.1}-\eqref{eq2.3.5.3} into \eqref{eq2.3.4} yields that
\begin{equation}\begin{split}\label{eq2.3.6}
	\frac{\mathrm{d}}{\mathrm{d}t}\aiminnorm{\nabla\hat\theta}^2 + \underline\kappa\aiminnorm{\Delta\hat\theta}^2 \leq C\big(\aiminnorm{\boldsymbol u}^2\aiminnorm{\nabla\boldsymbol u}^2 +  \aiminnorm{\kappa'(\theta)}^4_4\big)\aiminnorm{\nabla\hat\theta}^2
	+C\aiminnorm{\kappa(\theta)}_4^2\aiminnorm{\nabla\boldsymbol u}^2.
\end{split}\end{equation}
We temporarily ignore the second term involving $\Delta\hat\theta$ in the left-hand side of \eqref{eq2.3.6},
and we then employ the Uniform Gronwall Lemma~\ref{lem2.1.2} with $y=\aiminnorm{\nabla\hat\theta(t)}^2$ and utilize the estimates \eqref{eq2.2.8}, \eqref{eq2.2.10} and Lemmas~\ref{lem2.4.0}-\ref{lem2.4.1}; we find the uniform estimate \eqref{eq2.3.8} for $\aiminnorm{\nabla\hat\theta}$. Using the third inequality in \eqref{eq2.3.30.2}, we then obtain the desired uniform estimate \eqref{eq2.3.8} for $\aiminnorm{\nabla\theta}$ and hence the absorbing ball for $\theta$ in $H^1(\Omega)$.
Now, integrating \eqref{eq2.3.6} in time on $(t,t+1)$ and using the uniform estimates \eqref{eq2.2.8},\eqref{eq2.2.10} and \eqref{eq2.3.8}, we obtain the control \eqref{eq2.3.9} of the time average for $\aiminnorm{\Delta\hat\theta}^2$.
We thus have completed the proof of Proposition~\ref{prop2.4.1}.
\end{proof}

As a consequence of Proposition~\ref{prop2.4.1}, we obtain the control of the time average of $\aiminnorm{\Delta\theta}^2$.
\begin{corollary}\label{cor3.4.1}
Under the assumptions of Theorem~\ref{thm1.3}, we have
\begin{equation}\label{eq2.3.9p}
	\int_t^{t+1}\aiminnorm{\kappa(\theta)\Delta\theta(s)}^2{\mathrm{d}s} \leq C,\quad\forall\,t\geq  t_0+1.	
\end{equation}
\end{corollary}
\begin{proof}	
Note that, by the definition \eqref{eq2.3.1} of $\hat\theta$, and direct calculations, we have
\begin{equation}\label{eq2.3.7}
	\Delta\hat\theta = \kappa(\theta)\Delta\theta + \kappa'(\theta)\aiminabs{\nabla\theta}^2,
\end{equation}
which, with the third identity in \eqref{eq2.3.30.1}, yields
\begin{equation}\begin{split}\label{eq2.3.7.1}
	\aiminnorm{ \kappa(\theta)\Delta\theta }^2
	\leq 2&\aiminnorm{ \Delta\hat\theta }^2 + 2 \aiminnorm{\frac{\kappa'(\theta)}{\kappa(\theta)^2} \aiminabs{\nabla\hat\theta}^2 }^2\\
	&\leq 2\aiminnorm{ \Delta\hat\theta }^2+ C \aiminnorm{ \nabla\hat\theta }_4^4
	\leq 2\aiminnorm{ \Delta\hat\theta }^2 + C\aiminnorm{ \nabla\hat\theta }^2 \aiminnorm{\Delta\hat\theta}^2,
\end{split}\end{equation}
where we used the assumptions \eqref{asp1.1.3},\eqref{asp1.1.1} in the second inequality.
Hence, the uniform estimate \eqref{eq2.3.9p} then follows from \eqref{eq2.3.7.1} and Proposition~\ref{prop2.4.1}.
\end{proof}

We now aim to obtain a control of the time average of $\hat\theta_t$ and hence of $\theta_t$.
Taking the inner product of \eqref{eq2.3.2}$_2$ with $\hat\theta_t$ in $L^2(\Omega)$ and integrating by parts, we obtain
\begin{equation}\label{eq2.3.10}
	\begin{split}
		\frac12\frac{\mathrm{d}}{\mathrm{d}t}\int_\Omega \kappa(\theta)\aiminabs{\nabla\hat\theta}^2{\mathrm{d}x}{\mathrm{d}y} + \aiminnorm{\hat\theta_t}^2
		&= \frac12\int_\Omega\kappa'(\theta)\theta_t\aiminabs{\nabla\hat\theta}^2{\mathrm{d}x}{\mathrm{d}y}
		-\int_\Omega \boldsymbol u\cdot\nabla\hat\theta\cdot\hat\theta_t{\mathrm{d}x}{\mathrm{d}y}\\
		&\quad\quad+\int_\Omega \kappa(\theta)u_2\hat\theta_t{\mathrm{d}x}{\mathrm{d}y}
		-\int_\Omega \kappa'(\theta)\hat\theta_y\hat\theta_t{\mathrm{d}x}{\mathrm{d}y}\\
		&=:I_1+I_2+I_3+I_4.
	\end{split}
\end{equation}
The four terms in the right-hand side of \eqref{eq2.3.10} are estimated as follows.
Using assumption \eqref{asp1.1.3}, H\"older's inequality, Ladyzhenskaya's inequality, and Young's inequality, the first term $I_1$ is estimated as:
\begin{equation*}
	\begin{split}
		I_1
		\leq \frac12\int_\Omega\frac{\aiminabs{\kappa'(\theta)}}{\kappa(\theta)}\hat\theta_t\aiminabs{\nabla\hat\theta}^2{\mathrm{d}x}{\mathrm{d}y}
		\leq C\aiminnorm{\hat\theta_t}\aiminnorm{\nabla\hat\theta}_4^2
		\leq C\aiminnorm{\nabla\hat\theta}^2\aiminnorm{\Delta\hat\theta}^2 + \frac18\aiminnorm{\hat\theta_t}^2;
	\end{split}
\end{equation*}
applying Lemma~\ref{lem2.1.6} with $f=\boldsymbol u$, $g=\nabla\hat\theta$, and $h=\hat\theta_t$ yields
\begin{equation*}\begin{split}
	I_2 \leq\int_\Omega \aiminabs{\boldsymbol u}\aiminabs{\nabla\hat\theta}\aiminabs{\hat\theta_t} {\mathrm{d}x}{\mathrm{d}y}
	\leq C\aiminnorm{\boldsymbol u}^2\aiminnorm{\nabla\boldsymbol u}^2 + C\aiminnorm{\nabla\hat\theta}^{2}\aiminnorm{\Delta\hat\theta}^2 +\frac18\aiminnorm{\hat\theta_t}^2;
\end{split}\end{equation*}
using H\"older's inequality, the Sobolev embedding, and Young's inequality, the last two terms $I_3$ and $I_4$ are estimated as:
\begin{equation*}
	\begin{split}
		I_3\leq \aiminnorm{\kappa(\theta)}_4\aiminnorm{u_2}_4\aiminnorm{\hat\theta_t}
		\leq C\aiminnorm{\kappa(\theta)}_4^2 \aiminnorm{\nabla \boldsymbol u}^2 + \frac18\aiminnorm{\hat\theta_t}^2;
	\end{split}
\end{equation*}
\begin{equation*}\begin{split}
	I_4\leq\aiminnorm{\kappa'(\theta)}_4 \aiminnorm{\hat\theta_y}_4\aiminnorm{\hat\theta_t}
	\leq C\aiminnorm{\kappa'(\theta)}_4^2  \aiminnorm{\Delta\hat\theta}^2 + \frac18\aiminnorm{\hat\theta_t}^2.
\end{split}\end{equation*}

Combining the estimates for $I_1,\cdots,I_4$, we derive from \eqref{eq2.3.10} that
\begin{equation}\begin{split}\label{eq2.3.12}
	\frac{\mathrm{d}}{\mathrm{d}t}\aiminnorm{\sqrt{\kappa(\theta)}\nabla\hat\theta}^2 + \aiminnorm{\hat\theta_t}^2
	\leq
	C\big( (\aiminnorm{\nabla\hat\theta}^{2}+ \aiminnorm{\kappa'(\theta)}_4^2 )\aiminnorm{\Delta\hat\theta}^2 +(\aiminnorm{\boldsymbol u}^2 + \aiminnorm{\kappa(\theta)}_4^2) \aiminnorm{\nabla \boldsymbol u}^2 \big).
\end{split}\end{equation}
We note that, by H\"older's inequality and the Sobolev embedding:
\[
 \aiminnorm{\sqrt{\kappa(\theta)}\nabla\hat\theta}^2 \leq \aiminnorm{\sqrt{\kappa(\theta)}}_4^2 \aiminnorm{\nabla\hat\theta}_4^2
 \leq C\aiminnorm{\kappa(\theta)}_2^2 \aiminnorm{\Delta\hat\theta}^2,
\]
which, together with Lemma~\ref{lem2.4.1} and Proposition~\ref{prop2.4.1}, implies that
\begin{equation}\label{eq2.3.13}
	\int_t^{t+1}\aiminnorm{\sqrt{\kappa(\theta)}\nabla\hat\theta(s)}^2{\mathrm{d}s} \leq C,\quad\forall\,t\geq  t_0+1.
\end{equation}
We then apply the Uniform Gronwall Lemma~\ref{lem2.1.2} with $y=\aiminnorm{\sqrt{\kappa(\theta)}\nabla\hat\theta(t)}^2$ and utilize the estimates \eqref{eq2.3.8}, \eqref{eq2.2.8},\eqref{eq2.2.10},\eqref{eq2.3.13} and Lemma~\ref{lem2.4.1}; we arrive at
\begin{equation}
	\aiminnorm{\sqrt{\kappa(\theta)}\nabla\hat\theta}\leq C,\quad\forall\,t\geq t_0+2.
\end{equation}
Integrating \eqref{eq2.3.12} in time on $(t,t+1)$, we then obtain the control of the time average of $\hat\theta_t$:
\begin{equation}\label{eq2.3.15}
	\int_t^{t+1}\aiminnorm{\hat\theta_t(s)}^2{\mathrm{d}s} \leq C,\quad\forall\,t\geq t_0+2.
\end{equation}

We now turn to the uniform estimates of $\boldsymbol u$ in $H^1(\Omega)$; here we could not mimic the procedure for estimating $\theta$ in $H^1(\Omega)$ by defining a new quantity $\hat\theta$ because of the presence of the pressure term in the velocity equation. However, Proposition~\ref{prop2.4.1} provides us enough regularity for $\theta$ to show the uniform estimate of $\boldsymbol u$ in $H^1(\Omega)$.
\begin{proposition}\label{prop2.4.2}
	Under the assumptions of Theorem~\ref{thm1.3}, we have
\begin{equation}\label{eq2.3.18}
	\aiminnorm{\nabla\boldsymbol u(t)} \leq C,\quad\quad \forall\,t\geq  t_0+2,
\end{equation}
and
\begin{equation}\label{eq2.3.20}
	\int_t^{t+1}\aiminnorm{\Delta\boldsymbol u(s)}^2{\mathrm{d}s} \leq C,\quad\quad \forall\,t\geq  t_0+2.
\end{equation}
\end{proposition}
\begin{proof}
Taking the inner product of \eqref{eq1.1.3}$_1$ with $-\Delta \boldsymbol u$ in $L^2(\Omega)$ yields
\begin{equation}\begin{split}\label{eq2.3.41}
\frac12\frac{\mathrm{d}}{\mathrm{d}t}\aiminnorm{\nabla\boldsymbol u}^2 + &\int_\Omega \nu(\theta)\aiminabs{\Delta\boldsymbol u}^2{\mathrm{d}x}{\mathrm{d}y} = -\int_\Omega \nu'(\theta)\nabla\theta\cdot\nabla\boldsymbol u\cdot\Delta\boldsymbol u{\mathrm{d}x}{\mathrm{d}y} \\
& + \int_\Omega\boldsymbol u\cdot\nabla\boldsymbol u\cdot\Delta \boldsymbol u{\mathrm{d}x}{\mathrm{d}y} - \int_\Omega\theta\boldsymbol e_2\cdot\Delta\boldsymbol u{\mathrm{d}x}{\mathrm{d}y} - \int_\Omega (1-y)\boldsymbol e_2\cdot\Delta \boldsymbol u{\mathrm{d}x}{\mathrm{d}y},
\end{split}\end{equation}
where the pressure term disappears thanks to the free-slip boundary conditions for $\boldsymbol u$, see \eqref{eq1.1.10}.

With the assumption \eqref{asp1.1.3} and applying Lemma~\ref{lem2.1.3} with $f=\nabla\hat\theta$ and $g=\nabla\boldsymbol u$, we find that
\begin{equation}\begin{split}\label{eq2.3.42.1}
\int_\Omega \aiminabs{\nu'(\theta)\nabla\theta\cdot\nabla\boldsymbol u\cdot\Delta\boldsymbol u}{\mathrm{d}x}{\mathrm{d}y}
&\leq \int_\Omega \frac{\aiminabs{\nu'(\theta)}}{\kappa(\theta)}\aiminabs{\nabla\hat\theta\cdot\nabla\boldsymbol u\cdot\Delta\boldsymbol u}{\mathrm{d}x}{\mathrm{d}y}\\
&\leq C\aiminnorm{\nabla\hat\theta}^{2}\aiminnorm{\Delta\hat\theta}^{2}\aiminnorm{\nabla\boldsymbol u}^{2} + \frac{\underline\nu}{8}\aiminnorm{\Delta\boldsymbol u}^2;
\end{split}\end{equation}
applying Lemma~\ref{lem2.1.3} with $f=\boldsymbol u$ and $g=\nabla\boldsymbol u$ yields
\begin{equation}\begin{split}
\int_\Omega\aiminabs{\boldsymbol u\cdot\nabla\boldsymbol u\cdot\Delta \boldsymbol u}{\mathrm{d}x}{\mathrm{d}y}
\leq C\aiminnorm{\boldsymbol u}^2\aiminnorm{\nabla\boldsymbol u}^4 +   \frac{\underline\nu}{8}\aiminnorm{\Delta\boldsymbol u}^2;
\end{split}\end{equation}
the last two terms in the right-hand side of \eqref{eq2.3.41} are easy:
\begin{equation}\begin{split}\label{eq2.3.42.3}
 \int_\Omega\aiminabs{\theta\boldsymbol e_2\cdot\Delta\boldsymbol u}{\mathrm{d}x}{\mathrm{d}y} &\leq \aiminnorm{\theta}\aiminnorm{\Delta\boldsymbol u} \leq C\aiminnorm{\theta}^2 +   \frac{\underline\nu}{8}\aiminnorm{\Delta\boldsymbol u}^2,\\
 \int_\Omega \aiminabs{(1-y)\boldsymbol e_2\cdot\Delta \boldsymbol u}{\mathrm{d}x}{\mathrm{d}y}&\leq C\aiminnorm{\Delta\boldsymbol u}\leq C+   \frac{\underline\nu}{8}\aiminnorm{\Delta\boldsymbol u}^2.
\end{split}\end{equation}
By inserting the estimates \eqref{eq2.3.42.1}-\eqref{eq2.3.42.3} into \eqref{eq2.3.41}, we arrive at
\begin{equation}\begin{split}\label{eq2.3.44}
\frac{\mathrm{d}}{\mathrm{d}t}\aiminnorm{\nabla\boldsymbol u}^2 +  \underline\nu\aiminnorm{\Delta\boldsymbol u}^2
\leq C+C\aiminnorm{\theta}^2 + C\big( \aiminnorm{\nabla\hat\theta}^{2}\aiminnorm{\Delta\hat\theta}^{2} +  \aiminnorm{\boldsymbol u}^2\aiminnorm{\nabla\boldsymbol u}^2\big)\aiminnorm{\nabla\boldsymbol u}^2.
\end{split}\end{equation}
We temporarily ignore the second term involving $\Delta\boldsymbol u$ in the left-hand side of \eqref{eq2.3.44}, combine the estimates \eqref{eq2.2.8}, \eqref{eq2.2.10}, \eqref{eq2.3.8}, and \eqref{eq2.3.9}, and employ the Uniform Gronwall Lemma~\ref{lem2.1.2} with $y=\aiminnorm{\nabla\boldsymbol u(t)}^2$;
we obtain the uniform estimate \eqref{eq2.3.18} for $\aiminnorm{\nabla\boldsymbol u}$ and hence the absorbing ball for $\boldsymbol u$ in $H^1(\Omega)$. Integrating \eqref{eq2.3.44} in time on $(t,t+1)$, we then obtain the control \eqref{eq2.3.20} of the time average of $\aiminnorm{\Delta\boldsymbol u}^2$. This completes the proof of Proposition~\ref{prop2.4.2}.
\end{proof}

We now aim to obtain a control of the time average of $\aiminnorm{\boldsymbol u_t}^2$.
Taking the inner product of \eqref{eq1.1.3}$_1$ with $\boldsymbol u_t$ in $L^2(\Omega)$ yields
\begin{equation}\label{eq2.3.16}
	\begin{split}
		\frac12\frac{\mathrm{d}}{\mathrm{d}t} \int_\Omega \nu(\theta)\aiminabs{\nabla\boldsymbol u}^2{\mathrm{d}x}{\mathrm{d}y} + \aiminnorm{\boldsymbol u_t}^2&=\frac12\int_\Omega \nu'(\theta)\theta_t \aiminabs{\nabla\boldsymbol u}^2{\mathrm{d}x}{\mathrm{d}y} - \int_\Omega \boldsymbol u\cdot\nabla\boldsymbol u\cdot\boldsymbol u_t{\mathrm{d}x}{\mathrm{d}y} \\
		&\quad\; + \int_\Omega \theta\boldsymbol e_2\cdot\boldsymbol u_t{\mathrm{d}x}{\mathrm{d}y} +\int_\Omega (1-y)\boldsymbol e_2\cdot\boldsymbol u_t{\mathrm{d}x}{\mathrm{d}y} \\
		&=:I_1+I_2+I_3+I_4.
	\end{split}
\end{equation}
With the assumption \eqref{asp1.1.3} and using H\"older's inequality and  Ladyzhenskaya's inequality, we estimate the first term $I_1$ in the right-hand side of \eqref{eq2.3.16} as:
\begin{equation*}
	\begin{split}
I_1\leq\frac12\int_\Omega\aiminabs{ \nu'(\theta)\theta_t }\aiminabs{\nabla\boldsymbol u}^2{\mathrm{d}x}{\mathrm{d}y}
\leq \frac12\int_\Omega \frac{\aiminabs{\nu'(\theta)} }{\kappa(\theta)} \aiminabs{\hat\theta_t}\aiminabs{\nabla\boldsymbol u}^2{\mathrm{d}x}{\mathrm{d}y}
 \leq C\aiminnorm{\hat\theta_t}\aiminnorm{\nabla\boldsymbol u}_4^2 \\
 \leq C\aiminnorm{\hat\theta_t}\aiminnorm{\nabla\boldsymbol u}  \aiminnorm{\Delta\boldsymbol u}
 \leq \aiminnorm{\hat\theta_t}^2 + C\aiminnorm{\nabla\boldsymbol u}^2  \aiminnorm{\Delta\boldsymbol u}^2;		
	\end{split}
\end{equation*}
using H\"older's inequality and the Sobolev embedding, we estimate $I_2$ as:
\begin{equation*}
	\begin{split}
		I_2\leq\int_\Omega \aiminabs{\boldsymbol u\cdot\nabla\boldsymbol u\cdot\boldsymbol u_t}{\mathrm{d}x}{\mathrm{d}y}
		&\leq \aiminnorm{\boldsymbol u_t}\aiminnorm{\boldsymbol u}_4 \aiminnorm{\nabla\boldsymbol u}_4
		\leq C\aiminnorm{\boldsymbol u_t} \aiminnorm{\nabla\boldsymbol u}\aiminnorm{\Delta\boldsymbol u}\\
		&\leq \frac18\aiminnorm{\boldsymbol u_t}^2 + C\aiminnorm{\nabla\boldsymbol u}^2 \aiminnorm{\Delta\boldsymbol u}^2;
	\end{split}
\end{equation*}
the last two terms $I_3$ and $I_4$ in the right-hand side of \eqref{eq2.3.16} are easy:
\begin{equation*}
	I_3\leq \int_\Omega \aiminabs{\theta\boldsymbol e_2\cdot\boldsymbol u_t}{\mathrm{d}x}{\mathrm{d}y}\leq \aiminnorm{\theta}\aiminnorm{\boldsymbol u_t}\leq \frac18\aiminnorm{\boldsymbol u_t}^2 + \aiminnorm{\theta}^2,
\end{equation*}
\begin{equation*}
	I_4\leq \int_\Omega \aiminabs{(1-y)\boldsymbol e_2\cdot\boldsymbol u_t}{\mathrm{d}x}{\mathrm{d}y} \leq \aiminnorm{\boldsymbol u_t} \leq \frac18\aiminnorm{\boldsymbol u_t}^2 + 2.
\end{equation*}
By inserting these estimates for $I_1,\cdots, I_4$ into \eqref{eq2.3.16}, we arrive at
\begin{equation}\begin{split}\label{eq2.3.17}
	\frac{\mathrm{d}}{\mathrm{d}t} \int_\Omega \nu(\theta)\aiminabs{\nabla\boldsymbol u}^2{\mathrm{d}x}{\mathrm{d}y} + \aiminnorm{\boldsymbol u_t}^2
	\leq 4+2 \aiminnorm{\theta}^2 + 2\aiminnorm{\hat\theta_t}^2+  C\aiminnorm{\nabla\boldsymbol u}^2\aiminnorm{\Delta\boldsymbol u}^2.
\end{split}\end{equation}
We apply the Uniform Gronwall Lemma~\ref{lem2.1.2} with $y=\aiminnorm{\sqrt{\nu(\theta)}\nabla\boldsymbol u}^2$ and employ the uniform estimates \eqref{eq2.2.8}, \eqref{eq2.2.10p}, \eqref{eq2.3.15}, and \eqref{eq2.3.18}-\eqref{eq2.3.20}; we find
\begin{equation}
	\aiminnorm{\sqrt{\nu(\theta)}\nabla\boldsymbol u}  \leq C,\quad\quad \forall\,t\geq  t_0+3.
\end{equation}
Integrating \eqref{eq2.3.17} in time on $(t,t+1)$, we obtain the control of the time average of $\boldsymbol u_t$:
\begin{equation}\label{eq2.3.19}
	\int_t^{t+1}\aiminnorm{\boldsymbol u_t(s)}^2{\mathrm{d}s} \leq C,\quad\quad \forall\,t\geq  t_0+3.
\end{equation}

\subsection{$H^2$-estimates}
Uniform $H^2$-estimates for $(\boldsymbol u,\theta)$ can be obtained once we derive uniform $L^2$-estimates for $(\boldsymbol u_t,\theta_t)$. Actually, we will work with $\hat\theta$ instead $\theta$; for that reason, we take the time derivative of equations \eqref{eq1.1.3}$_1$ and \eqref{eq2.3.2}$_2$, and find that
\begin{equation}\label{eq2.4.1}
	\begin{cases}
		\partial_t\boldsymbol u_t -{\mathrm{div}\,}(\nu(\theta)\nabla\boldsymbol u_t) + \boldsymbol u\cdot\nabla\boldsymbol u_t  + \nabla p_t = \theta_t\boldsymbol e_2 - \boldsymbol u_t\cdot\nabla\boldsymbol u + {\mathrm{div}\,}(\nu'(\theta)\theta_t\nabla\boldsymbol u),\\
		\partial_t\hat\theta_t - \kappa(\theta)\Delta\hat\theta_t + \kappa'(\theta)\theta_t\Delta\hat\theta + \boldsymbol u_t\cdot\nabla\hat\theta + \boldsymbol u\cdot\nabla\hat\theta_t\\
		\hspace{100pt}= \kappa'(\theta)\theta_tu_2 +\kappa(\theta)u_{2,t} -\kappa'(\theta)\hat\theta_{t,y} - \kappa''(\theta)\theta_t\hat\theta_y
	\end{cases}
\end{equation}

\begin{proposition}\label{prop2.5.1}
	Under the assumptions of Theorem~\ref{thm1.3}, we have
\begin{equation}\label{eq2.4.35}
	\aiminnorm{\hat\theta_t(t)},\;\aiminnorm{\theta_t(t)} \leq C,\quad\quad \forall\,t\geq  t_0+4,
\end{equation}
and
\begin{equation}\label{eq2.4.36}
	\int_t^{t+1}\aiminnorm{\sqrt{\kappa(\theta)}\nabla\hat\theta_t(s) }^2 {\mathrm{d}s}\leq C,\quad\quad \forall\,t\geq  t_0+4.
\end{equation}	
\end{proposition}
\begin{proof}	
Taking the inner product of \eqref{eq2.4.1}$_2$ with $\hat\theta_t$ in $L^2(\Omega)$ and integrating by parts, we obtain that
\begin{equation}\begin{split}\label{eq2.4.31}
	\frac12\frac{\mathrm{d}}{\mathrm{d}t}\aiminnorm{\hat\theta_t}^2 + \int_\Omega\kappa(\theta)\aiminabs{\nabla\hat\theta_t}^2{\mathrm{d}x}{\mathrm{d}y}
	=
	&-\int_\Omega \kappa'(\theta)\nabla\theta\nabla\hat\theta_t\hat\theta_t{\mathrm{d}x}{\mathrm{d}y}
	-\int_\Omega \kappa'(\theta)\theta_t\Delta\hat\theta\hat\theta_t{\mathrm{d}x}{\mathrm{d}y}\\
	&-\int_\Omega \boldsymbol u_t\cdot \nabla\hat\theta\hat\theta_t{\mathrm{d}x}{\mathrm{d}y}
	-\int_\Omega \boldsymbol u\cdot\nabla\hat\theta_t\hat\theta_t{\mathrm{d}x}{\mathrm{d}y}\\
	&+\int_\Omega \kappa'(\theta)\theta_tu_2\hat\theta_t{\mathrm{d}x}{\mathrm{d}y}
	+\int_\Omega \kappa(\theta)u_{2,t}\hat\theta_t{\mathrm{d}x}{\mathrm{d}y}\\
	&-\int_\Omega \kappa'(\theta)\hat\theta_{t,y}{\mathrm{d}x}{\mathrm{d}y}
	-\int_\Omega \kappa''(\theta)\theta_t\hat\theta_y{\mathrm{d}x}{\mathrm{d}y}\\
	=&I_1+\cdots+I_8.
\end{split}\end{equation}
With the assumption \eqref{asp1.1.3} and applying Lemma~\ref{lem2.1.3} with $f=\nabla\hat\theta$ and $g=\hat\theta_t$, we find that
\begin{equation*}\begin{split}
	I_1\leq \int_\Omega \frac{ \aiminabs{\kappa'(\theta)} }{\kappa(\theta)} \aiminabs{\nabla\hat\theta\nabla\hat\theta_t\hat\theta_t} {\mathrm{d}x}{\mathrm{d}y}
	\leq C\aiminnorm{\nabla\hat\theta}^2\aiminnorm{\Delta\hat\theta}^2\aiminnorm{\hat\theta_t}^2 + \frac{\underline\kappa}{16}\aiminnorm{\nabla\hat\theta_t}^2.
\end{split}\end{equation*}
With the assumption \eqref{asp1.1.3} and using H\"older's inequality, Ladyzhenskaya's inequality, and Young's inequality, we estimate $I_2$ and $I_3$ as follows:
\begin{equation*}\begin{split}
	I_2
	&\leq \int_\Omega \frac{ \aiminabs{\kappa'(\theta)} }{\kappa(\theta)}\aiminabs{\hat\theta_t}^2\aiminabs{\Delta\hat\theta}{\mathrm{d}x}{\mathrm{d}y}
	\leq C \aiminnorm{\Delta\hat\theta}\aiminnorm{\hat\theta_t}_4^2\\
	&\leq C\aiminnorm{\Delta\hat\theta}\aiminnorm{\hat\theta_t}\aiminnorm{\nabla\hat\theta_t}
	\leq C\aiminnorm{\Delta\hat\theta}^2\aiminnorm{\hat\theta_t}^2 +  \frac{\underline\kappa}{16}\aiminnorm{\nabla\hat\theta_t}^2;
\end{split}\end{equation*}
\begin{equation*}\begin{split}
	I_3
	&\leq \aiminnorm{\boldsymbol u_t}\aiminnorm{\nabla\hat\theta}_4\aiminnorm{\hat\theta_t}_4
	\leq C\aiminnorm{\boldsymbol u_t}\aiminnorm{\nabla\hat\theta}^{1/2}\aiminnorm{\Delta\hat\theta}^{1/2}\aiminnorm{\hat\theta_t}^{1/2}\aiminnorm{\nabla\hat\theta_t}^{1/2}\\
	&\leq \aiminnorm{\boldsymbol u_t}^2 + C \aiminnorm{\nabla\hat\theta}^{2}\aiminnorm{\Delta\hat\theta}^2\aiminnorm{\hat\theta_t}^2 +\frac{\underline\kappa}{16}\aiminnorm{\nabla\hat\theta_t}^2;
\end{split}\end{equation*}
by Green's formula, we have
\begin{equation*}
	I_4 = 0;
\end{equation*}
with the assumption \eqref{asp1.1.3}, we estimate $I_5$ as
\begin{equation*}\begin{split}
	I_5
	&\leq \int_\Omega \frac{ \aiminabs{\kappa'(\theta)} }{\kappa(\theta)}\aiminabs{\hat\theta_t}^2\aiminabs{u_2}{\mathrm{d}x}{\mathrm{d}y}
	\leq C\aiminnorm{u_2}\aiminnorm{\hat\theta_t}_4^2\\
	&\leq C\aiminnorm{\boldsymbol u}\aiminnorm{\hat\theta_t}\aiminnorm{\nabla\hat\theta_t}
	\leq C\aiminnorm{\boldsymbol u}^2\aiminnorm{\hat\theta_t}^2 +\frac{\underline\kappa}{16}\aiminnorm{\nabla\hat\theta_t}^2;
\end{split}\end{equation*}
by H\"older's inequality and the Sobolev embedding, we have
\begin{equation*}\begin{split}
	I_6
	&\leq \aiminnorm{\kappa(\theta)}_4 \aiminnorm{u_{2,t}}\aiminnorm{\hat\theta_t}_4
	\leq C\aiminnorm{\kappa(\theta)}_4 \aiminnorm{\boldsymbol u_{t}} \aiminnorm{\nabla\hat\theta_t}\\
	&\leq C\aiminnorm{\kappa(\theta)}_4^2 \aiminnorm{\boldsymbol u_{t}}^2  +\frac{\underline\kappa}{16}\aiminnorm{\nabla\hat\theta_t}^2;
\end{split}\end{equation*}
by the Cauchy-Schwarz inequality and Young's inequality, we have
\begin{equation*}
	I_7\leq \aiminnorm{\kappa'(\theta)} \aiminnorm{\hat\theta_{t,y}}
	\leq C\aiminnorm{\kappa'(\theta)}^2  +\frac{\underline\kappa}{16}\aiminnorm{\nabla\hat\theta_t}^2;
\end{equation*}
by the assumption \eqref{asp1.1.3} and the Cauchy-Schwarz inequality, we have
\begin{equation*}\begin{split}
	I_8\leq \int_\Omega \frac{ \aiminabs{\kappa''(\theta)} }{\kappa(\theta)}\hat\theta_t\hat\theta_y{\mathrm{d}x}{\mathrm{d}y}
	\leq C\aiminnorm{\hat\theta_t}\aiminnorm{\hat\theta_y}
	\leq C\aiminnorm{\nabla\hat\theta}^2 + \aiminnorm{\hat\theta_t}^2.
\end{split}\end{equation*}
Combining all the estimates for $I_1,\cdots,I_8$, we infer from \eqref{eq2.4.31} that
\begin{equation}\begin{split}\label{eq2.4.34}
	\frac{\mathrm{d}}{\mathrm{d}t}\aiminnorm{\hat\theta_t}^2 + \aiminnorm{\sqrt{\kappa(\theta)}\nabla\hat\theta_t }^2
	&\leq C\big(  \aiminnorm{\nabla\hat\theta}^{2}\aiminnorm{\Delta\hat\theta}^2+\aiminnorm{\Delta\hat\theta}^2+\aiminnorm{\boldsymbol u}^2+1 \big)\aiminnorm{\hat\theta_t}^2\\
	&\quad+C\big[(1+\aiminnorm{\kappa(\theta)}_4^2 )\aiminnorm{\boldsymbol u_{t}}^2  + \aiminnorm{\nabla\hat\theta}^2 + \aiminnorm{\kappa'(\theta)}^2\big].
\end{split}\end{equation}
We then apply the Uniform Gronwall Lemma~\ref{lem2.1.2} with $y=\aiminnorm{\hat\theta_t(t)}^2$ and employ the uniform estimates \eqref{eq2.3.19},\eqref{eq2.3.15}, Propositions~\ref{prop2.3.1}-\ref{prop2.4.1}, and Lemma~\ref{lem2.4.1}; we arrive at the uniform estimate \eqref{eq2.4.35}.
Finally, integrating \eqref{eq2.4.34} in time on $(t,t+1)$ readily yields \eqref{eq2.4.36}. We thus have completed the proof of Proposition~\ref{prop2.5.1}.
\end{proof}

\begin{proposition}\label{prop2.5.2}
	Under the assumptions of Theorem~\ref{thm1.3}, we have
\begin{equation}\label{eq2.4.6}
	\aiminnorm{\boldsymbol u_t}\leq C,\quad\quad \forall\,t\geq  t_0+5,
\end{equation}
and
\begin{equation}\label{eq2.4.5}
	\int_t^{t+1}\aiminnorm{ \nabla\boldsymbol u_t(s) }^2 {\mathrm{d}s}\leq C,\quad\quad \forall\,t\geq  t_0+5.
\end{equation}	
\end{proposition}
\begin{proof}
Taking the inner product of  \eqref{eq2.4.1}$_1$ with $\boldsymbol u_t$ in $L^2(\Omega)$ and integrating by parts, we obtain
\begin{equation}\begin{split}\label{eq2.4.2}
	\frac12\frac{\mathrm{d}}{\mathrm{d}t}\aiminnorm{\boldsymbol u_t}^2 +\int_\Omega\nu(\theta)\aiminabs{\nabla\boldsymbol u_t}^2{\mathrm{d}x}{\mathrm{d}y}
	\leq \int_\Omega &\aiminabs{\nu'(\theta)\theta_t}\aiminabs{\nabla\boldsymbol u}\aiminabs{\nabla\boldsymbol u_t}{\mathrm{d}x}{\mathrm{d}y}
	+ \int_\Omega\aiminabs{\nabla\boldsymbol u}\aiminabs{\boldsymbol u_t}^2{\mathrm{d}x}{\mathrm{d}y} \\
	&+ \int_\Omega\aiminabs{\theta_t}\aiminabs{\boldsymbol u_t}{\mathrm{d}x}{\mathrm{d}y}=:I_1+I_2+I_3.
\end{split}\end{equation}
Using Young's inequality and Lemma~\ref{lem2.1.1}, we estimate each term in the right-hand side of \eqref{eq2.4.2} as follows.
With the assumption \eqref{asp1.1.3} and applying Lemma~\ref{lem2.1.6} with $f=\hat\theta_t$, $g=\nabla\boldsymbol u$, and $h=\nabla\boldsymbol u_t$, we estimate $I_1$ as
\begin{equation*}\begin{split}
	I_1
	&\leq \int_\Omega \frac{\aiminabs{\nu'(\theta)}}{\kappa(\theta)}\aiminabs{\hat\theta_t}\aiminabs{\nabla\boldsymbol u}\aiminabs{\nabla\boldsymbol u_t}{\mathrm{d}x}{\mathrm{d}y}
	\leq C\int_\Omega \aiminabs{\hat\theta_t}\aiminabs{\nabla\boldsymbol u}\aiminabs{\nabla\boldsymbol u_t}{\mathrm{d}x}{\mathrm{d}y}\\
	&\leq   C\aiminnorm{\hat\theta_t}^2\aiminnorm{\nabla\hat\theta_t}^2 +C\aiminnorm{\nabla\boldsymbol u}^2\aiminnorm{\Delta\boldsymbol u}^2 + \frac{\underline\nu}{8}\aiminnorm{\nabla\boldsymbol u_t}^2;
\end{split}\end{equation*}
using H\"older's inequality and Young's inequality, we estimate $I_2$ and $I_3$ as
\begin{equation*}
	\begin{split}
	I_2
		\leq \aiminnorm{\nabla\boldsymbol u}\aiminnorm{\boldsymbol u_t}_4^2
		\leq C\aiminnorm{\nabla\boldsymbol u}\aiminnorm{\boldsymbol u_t}\aiminnorm{\nabla\boldsymbol u_t}
		\leq C\aiminnorm{\nabla\boldsymbol u}^2\aiminnorm{\boldsymbol u_t}^2 + \frac{\underline\nu}{8}\aiminnorm{\nabla\boldsymbol u_t}^2,
	\end{split}
\end{equation*}
\[
I_3=\int_\Omega\aiminabs{\theta_t}\aiminabs{\boldsymbol u_t}{\mathrm{d}x}{\mathrm{d}y}
\leq \frac12\aiminnorm{\theta_t}^2 + \frac12\aiminnorm{\boldsymbol u_t}^2.
\]
Combining all these estimates for $I_1$, $I_2$ and $I_3$, it then follows from \eqref{eq2.4.2} that
\begin{equation}\label{eq2.4.4}
	\begin{split}
		\frac{\mathrm{d}}{\mathrm{d}t}\aiminnorm{\boldsymbol u_t}^2 + \underline\nu\aiminnorm{\nabla\boldsymbol u_t}^2
		\leq C\aiminnorm{\hat\theta_t}^2\aiminnorm{\nabla\theta_t}^2 +\aiminnorm{\theta_t}^2 + C\aiminnorm{\nabla\boldsymbol u}^2\aiminnorm{\Delta\boldsymbol u}^2\\
		 + C\big(1+\aiminnorm{\nabla\boldsymbol u}^2\big) \aiminnorm{\boldsymbol u_t}^2.
	\end{split}
\end{equation}
Applying then the Uniform Gronwall Lemma~\ref{lem2.1.2} with $y=\aiminnorm{\boldsymbol u_t}^2$  and employing Propositions~\ref{prop2.4.1}-\ref{prop2.5.1} and the uniform estimates \eqref{eq2.3.19}, we obtain the uniform estimate \eqref{eq2.4.6}. As before, integrating \eqref{eq2.4.4} in time on $(t,t+1)$ yields \eqref{eq2.4.5}. We thus completed the proof of Proposition~\ref{prop2.5.2}.
\end{proof}

We are now ready to derive the uniform estimate for $\theta$ in $H^2(\Omega)$.
\begin{proposition}\label{prop2.5.3}
	Under the assumptions of Theorem~\ref{thm1.3}, we have
\begin{equation}\label{eq2.4.9}
	\aiminnorm{\Delta\hat\theta(t)},\;\aiminnorm{\Delta\theta(t)}\leq C,\quad\quad \forall\,t\geq  t_0+5.
\end{equation}	
\end{proposition}
\begin{proof}
	
 We infer from equation \eqref{eq2.3.2}$_2$ that
\begin{equation}\label{eq2.4.7}
	\underline\kappa\aiminnorm{\Delta\hat\theta}\leq \aiminnorm{\hat\theta_t} + \aiminnorm{\boldsymbol u\cdot\nabla\hat\theta} + \aiminnorm{\kappa(\theta)u_2} + \aiminnorm{\kappa'(\theta)\hat\theta_y}.
\end{equation}
The last three terms in the right-hand side of \eqref{eq2.4.7} can be estimated as follows:
\begin{equation*}
	\aiminnorm{\boldsymbol u\cdot\nabla\hat\theta}\leq\aiminnorm{\boldsymbol u}_4\aiminnorm{\nabla\hat\theta}_4
	\leq C\aiminnorm{\nabla\boldsymbol u}\aiminnorm{\nabla\hat\theta}^{1/2}\aiminnorm{\Delta\hat\theta}^{1/2}
	\leq C\aiminnorm{\nabla\boldsymbol u}^2\aiminnorm{\nabla\hat\theta} + \frac{\underline\kappa}{8}\aiminnorm{\Delta\hat\theta};
\end{equation*}
\begin{equation*}
	\aiminnorm{\kappa(\theta)u_2}\leq \aiminnorm{\kappa(\theta)}_4\aiminnorm{ u_2}_4\leq \aiminnorm{\kappa(\theta)}_4\aiminnorm{\nabla\boldsymbol u};
\end{equation*}
\begin{equation*}
	 \aiminnorm{\kappa'(\theta)\hat\theta_y}\leq \aiminnorm{\kappa'(\theta)}_4\aiminnorm{\hat\theta_y}_4
	 \leq C \aiminnorm{\kappa'(\theta)}_4^2\aiminnorm{\nabla\hat\theta} + \frac{\underline\kappa}{8}\aiminnorm{\Delta\hat\theta}.
\end{equation*}
By inserting these estimates into \eqref{eq2.4.7} and using the uniform estimates \eqref{eq2.4.35},\eqref{eq2.3.8},\eqref{eq2.3.18} and Lemma~\ref{lem2.4.1}, we conclude that
 \eqref{eq2.4.9} is valid for $\Delta\hat\theta$. We then infer from the inequality \eqref{eq2.3.7.1} that \eqref{eq2.4.9} is also valid for $\Delta\theta$, which gives us the desired uniform estimates of $\aiminnorm{\theta}_{H^2}$ and the absorbing ball of $\theta$ in $H^2(\Omega)$.
\end{proof}

As a consequence of Proposition~\ref{prop2.5.3} and the Sobolev embedding $H^2(\Omega)\hookrightarrow L^\infty(\Omega)$, we obtain that
\begin{equation}\label{eq2.4.41}
	\aiminnorm{\theta(t)}_\infty,\;\aiminnorm{\hat\theta(t)}_\infty\leq C,\quad\forall\,t\geq  t_0+5,
\end{equation}
which implies that
\begin{equation}\begin{split}\label{eq2.4.42}
	&\underline\nu \leq \nu(\theta) \leq \bar\nu,\quad\quad \underline\kappa \leq \kappa(\theta) \leq \bar\kappa,\quad\forall\,t\geq  t_0+5,\\
	&\aiminabs{\nu'(\theta)},\;\aiminabs{\kappa'(\theta)} \leq C,\hspace{72pt}\forall\,t\geq  t_0+5,\\
\end{split}\end{equation}
for some constants $\bar\nu,\bar\kappa>0$, independent of the time $t$ and initial data $\boldsymbol u_0$ and $\theta_0$.

We are now going  to utilize the regularity results for Stokes system (see e.g. \cite[Chapter I]{Tem01}) to obtain a uniform estimate of $\aiminnorm{\boldsymbol u}_{H^2}$.
\begin{proposition}\label{prop2.5.4}
	Under the assumptions of Theorem~\ref{thm1.3}, we have
\begin{equation*}
	\aiminnorm{\boldsymbol u(t)}_{H^2} \leq C,\quad\forall\,t\geq  t_0+5.
\end{equation*}	
\end{proposition}
\begin{proof}	
We first need to establish some uniform estimates on the pressure term $p$ and for that reason, we rewrite \eqref{eq1.1.3}$_{1,2}$ as
\begin{equation}\begin{cases}\label{eq2.4.43}
	 - {\mathrm{div}\,}(\nu(\theta) \nabla \boldsymbol u)   + \nabla p = \boldsymbol f:=\boldsymbol u_t -\boldsymbol u\cdot \nabla \boldsymbol u+ \theta \boldsymbol e_2 + (1-y)\boldsymbol e_2,\quad \boldsymbol e_2=(0,1),\\
	{\mathrm{div}\,}\boldsymbol u=0.\\	
\end{cases}\end{equation}
For $\boldsymbol f\in H^{-1}(\Omega)$ and $\theta$ and $\nu(\theta)$ satisfying \eqref{eq2.4.41} and \eqref{eq2.4.42}$_1$, the Lax-Milgram theorem and the Helmholtz decomposition ensure that there exists a unique solution $(\boldsymbol u,p)\in H_0^1(\Omega)\times L^2(\Omega)/\mathbb R$ to the system \eqref{eq2.4.43} such that
\begin{equation}\label{eq2.4.44.1}
\aiminnorm{\boldsymbol u}_{H^1} + \aiminnorm{p}_{L^2(\Omega)/\mathbb R} \leq C(\underline\nu,\bar\nu,\Omega)\aiminnorm{\boldsymbol f}_{H^{-1}},	
\end{equation}
where the space $L^2(\Omega)/\mathbb R$ is isomorphic to the subspace of $L^2(\Omega)$ orthogonal to the constants:
\[
L^2(\Omega)/\mathbb R = \aiminset{p\in L^2(\Omega),\; \int_\Omega p(x,y){\mathrm{d}x}{\mathrm{d}y} =0}.
\]
We note that $\boldsymbol u\cdot\nabla\boldsymbol u=(u_1\boldsymbol u)_x+(u_2\boldsymbol u)_y$ since $\boldsymbol u$ is divergence-free (see \eqref{eq2.4.43}$_2$). Therefore, for the $\bf f$ in \eqref{eq2.4.43}, we have
\begin{equation}\begin{split}\label{eq2.4.44.2}
	\aiminnorm{\boldsymbol f}_{H^{-1}}
	&\leq \aiminnorm{\boldsymbol u_t}_{H^{-1} } +\aiminnorm{ (u_1\boldsymbol u)_x+(u_2\boldsymbol u)_y}_{H^{-1}}+ \aiminnorm{\theta }_{H^{-1}} + \aiminnorm{ (1-y) }_{H^{-1} }\\
	&\leq \aiminnorm{\boldsymbol u_t} + \aiminnorm{ \aiminabs{\boldsymbol u}^2 } + \aiminnorm{\theta } + \aiminnorm{ (1-y) }
	\leq \aiminnorm{\boldsymbol u_t} + \aiminnorm{ \boldsymbol u }\aiminnorm{\nabla\boldsymbol u} + \aiminnorm{\theta } + 1.
\end{split}\end{equation}
With the uniform estimates \eqref{eq2.2.8}, \eqref{eq2.3.18}, \eqref{eq2.4.6}, we infer from \eqref{eq2.4.44.1}-\eqref{eq2.4.44.2} that
\begin{equation}\label{eq2.4.45}
	\aiminnorm{p}_{L^2(\Omega)/\mathbb R} \leq C,\quad\forall\,t\geq  t_0+5.
\end{equation}

In order to use the regularity results for the Stokes system,  we rewrite \eqref{eq1.1.3}$_{1,2}$ as
\begin{equation}\begin{split}\label{eq2.4.46}
	-\Delta\boldsymbol u + \nabla\tilde p &= \boldsymbol f:=
	-\frac{\nu'(\theta) }{\nu(\theta)}\tilde p\nabla\theta + \frac{\nu'(\theta) }{\nu(\theta)}\nabla\theta\nabla\boldsymbol u -\frac{1}{\nu(\theta)}\boldsymbol u\cdot\nabla\boldsymbol u \\
	&\hspace{40pt}-\frac{1}{\nu(\theta)}\boldsymbol u_t + \frac{1}{\nu(\theta)}\theta \boldsymbol e_2 + \frac{1}{\nu(\theta)}(1-y)\boldsymbol e_2,\\
	{\mathrm{div}\,}\boldsymbol u&=0,\\
\end{split}\end{equation}
where
\[
\tilde p=\frac{p}{\nu(\theta)}.
\]
Applying \cite[Chapter I, Proposition 2.2]{Tem01} to \eqref{eq2.4.46}, we obtain that
\begin{equation}\label{eq2.4.47.1}
	\aiminnorm{\boldsymbol u}_{H^2}^2 + \aiminnorm{  \tilde p }_{H^1}^2
	\leq \aiminnorm{\boldsymbol f}^2.
\end{equation}
We remark that the results in \cite[Chapter I, Proposition 2.2]{Tem01} concern the case of Dirichlet boundary conditions in a smooth domain, but can be also applied to a periodic channel domain with free-slip boundary conditions and the proofs are easier.

To estimate $\aiminnorm{\boldsymbol f}$, we estimate each term in the right-hand side of \eqref{eq2.4.46}$_1$ as follows.

\noindent By \eqref{eq2.4.42} and Lemma~\ref{lem2.1.4} with $f=\nabla\theta$ and $g=\tilde p$, we have
\begin{equation}\begin{split}\label{eq2.4.47.2}
	\aiminnorm{ \frac{\nu'(\theta) }{\nu(\theta)}\tilde p\nabla\theta  }^2
	&\leq C \aiminnorm{\tilde p\nabla\theta }^2
	\leq C\aiminnorm{\tilde p}^2\aiminnorm{\nabla\theta}^2\aiminnorm{\nabla\theta}_{H^1}^2 + \frac14 \aiminnorm{\tilde p }_{H^1}^2\\
	&\leq C\aiminnorm{p}^2\aiminnorm{\Delta\theta}^4 +  \frac14 \aiminnorm{\tilde p}_{H^1}^2.
\end{split}\end{equation}
By \eqref{eq2.4.42} and Lemma~\ref{lem2.1.4} with $f=\nabla\theta$ and $g=\nabla\boldsymbol u$, we have
\begin{equation}\begin{split}
	\aiminnorm{ \frac{\nu'(\theta) }{\nu(\theta)}\nabla\theta\nabla\boldsymbol u }^2
	\leq C\aiminnorm{ \nabla\theta\nabla\boldsymbol u }^2
	\leq C\aiminnorm{\nabla\theta}^2\aiminnorm{\nabla\theta}_{H^1}^2\aiminnorm{\nabla\boldsymbol u}^2 + \frac14\aiminnorm{\nabla\boldsymbol u}_{H^1}^2.
\end{split}\end{equation}
By \eqref{eq2.4.42} and Lemma~\ref{lem2.1.4} with $f=\boldsymbol u$ and $g=\nabla\boldsymbol u$, we have
\begin{equation}\begin{split}
	\aiminnorm{ \frac{1}{\nu(\theta)}\boldsymbol u\cdot\nabla\boldsymbol u }^2
	\leq C\aiminnorm{\boldsymbol u\cdot\nabla\boldsymbol u}
	\leq C\aiminnorm{\boldsymbol u}^2\aiminnorm{\nabla\boldsymbol u}^2\aiminnorm{\nabla\boldsymbol u}^2 + + \frac14\aiminnorm{\nabla\boldsymbol u}_{H^1}^2.
\end{split}\end{equation}
By \eqref{eq2.4.42}, we have
\begin{equation}\label{eq2.4.47.5}
	\aiminnorm{ \frac{1}{\nu(\theta)}\boldsymbol u_t }^2 +\aiminnorm{ \frac{1}{\nu(\theta)}\theta \boldsymbol e_2}^2 +\aiminnorm{ \frac{1}{\nu(\theta)}(1-y)\boldsymbol e_2 }^2\leq C(\aiminnorm{\boldsymbol u_t}^2 + \aiminnorm{\theta}^2 +1).
\end{equation}

Combining all the estimates \eqref{eq2.4.47.2}-\eqref{eq2.4.47.5}, we infer from \eqref{eq2.4.47.1} that
\begin{equation}\label{eq2.4.11p}
	\aiminnorm{\boldsymbol u}_{H^2}^2 + \aiminnorm{  \tilde p}_{H^1}^2
	\leq C (1+\aiminnorm{\boldsymbol u_t}^2 + \aiminnorm{\Delta\theta}^4\aiminnorm{p}^2 + \aiminnorm{\boldsymbol u}^2\aiminnorm{\nabla\boldsymbol u}^4 + \aiminnorm{\nabla\theta}^2\aiminnorm{\nabla\theta}_{H^1}^2\aiminnorm{\nabla\boldsymbol u}^2).
\end{equation}
Now, since all the terms in the right-hand side of \eqref{eq2.4.11p} are uniformly bounded by Propositions~\ref{prop2.4.2},\;\ref{prop2.5.2}-\ref{prop2.5.3} and estimate \eqref{eq2.4.45}, we then conclude that
\begin{equation}\label{eq2.4.11}
	\aiminnorm{\boldsymbol u(t)}_{H^2},\;\aiminnorm{  \tilde p }_{H^1} \leq C,\quad\forall\,t\geq  t_0+5.
\end{equation}
This also gives us the absorbing ball of $\boldsymbol u$ in $H^2(\Omega)$. We thus completed the proof of Proposition~\ref{prop2.5.3}.
\end{proof}

We now present an auxiliary lemma which is useful for obtaining the control of the time average of $\aiminnorm{(\boldsymbol u,\hat\theta)}_{H^3}^2$.
\begin{lemma}\label{lem2.5.1}
Let $\chi(\tau)$ be a $\mathcal C^1$-function in $\mathbb R$, $g\in H^1(\Omega)$, and $\theta$ satisfies \eqref{eq2.4.41}. Then
\begin{equation}\label{lem2.5.1eq1}
	\aiminnorm{\chi(\theta)g}_{H^1}^2 \leq C(\aiminnorm{\Delta\theta}^2 + 1)\aiminnorm{g}_{H^1}^2.
\end{equation}
\end{lemma}
\begin{proof}
By \eqref{eq2.4.42}$_1$ and the Sobolev embedding, we find that
	\begin{equation*}\begin{split}
		\aiminnorm{\chi(\theta)g}_{H^1}^2
		&\leq \aiminnorm{\chi(\theta)g}^2 + \aiminnorm{\chi(\theta)\nabla g}^2+ \aiminnorm{\chi'(\theta)(\nabla\theta) g}^2 \\
		&\leq C\aiminnorm{g}^2 + C\aiminnorm{\nabla g}^2 + C\aiminnorm{\nabla\theta}_4^2\aiminnorm{g}_4^2
		\leq C(\aiminnorm{\Delta\theta}^2 + 1)\aiminnorm{g}_{H^1}^2.
	\end{split}\end{equation*}
	This proves \eqref{lem2.5.1eq1}.
\end{proof}

We now utilize the classical elliptic regularity theory to obtain the control of the time average of $\aiminnorm{\hat\theta}_{H^3}^2$.
\begin{proposition}\label{prop2.5.5}
	Under the assumptions of Theorem~\ref{thm1.3}, we have
	\begin{equation}\label{eq2.4.16p}
		\int_t^{t+1}\aiminnorm{\hat\theta(s)}_{H^3}^2{\mathrm{d}s},\;\int_t^{t+1}\aiminnorm{\theta(s)}_{H^3}^2{\mathrm{d}s}\leq C,\quad\forall\,t\geq  t_0+5.
	\end{equation}
\end{proposition}

\begin{proof}
We first rewrite \eqref{eq2.3.2}$_2$ as
\begin{equation*}\begin{cases}
	\Delta\hat\theta = \kappa(\theta)^{-1}\big( \hat\theta_t + \boldsymbol u\cdot\nabla\hat\theta - \kappa(\theta)u_2 + \kappa'(\theta)\hat\theta_y\big),\\
	\hat\theta = 0,\quad\text{ on }\partial\Omega.
\end{cases}\end{equation*}
Therefore, we find that
\begin{equation}\begin{split}\label{eq2.4.12}
	\aiminnorm{\hat\theta}_{H^3}^2 \leq C
	\big( \aiminnorm{ \kappa(\theta)^{-1}\hat\theta_t }_{H^1}^2 + \aiminnorm{\kappa(\theta)^{-1}\boldsymbol u\cdot\nabla\hat\theta  }_{H^1}^2 + \aiminnorm{u_2}_{H^1}^2 + \aiminnorm{\frac{\kappa'(\theta)}{\kappa(\theta)}\hat\theta_y }_{H^1}^2\big).
\end{split}\end{equation}
We remark that we do not find an exact reference for the $H^3$ regularity for the Laplace operator in the channel domain and we sketch the proof of \eqref{eq2.4.12} here. First, we find a smooth domain $\widetilde\Omega$ such that
\[
(-1,2)_x\times(0,1)_y\subset \widetilde\Omega \subset (-2,3)_x\times(0,1)_y,
\]
and  choose a mollifier $\varphi$ with compact support contained in $\widetilde\Omega$ and equal to $1$ on $[0,1]_x\times[0,1]_y$. Then applying the elliptic regularity result (see e.g. \cite{Eva98}) to the function $\Delta(\varphi\hat\theta)$ with domain $\widetilde\Omega$ will yield \eqref{eq2.4.12}.

We now need to estimate the right-hand side of \eqref{eq2.4.12}.
As a preliminary, by the Sobolev embedding, we have
\begin{equation}\label{eq2.4.13.1}
\begin{split}
	\aiminnorm{\nabla(\boldsymbol u\cdot\nabla\hat\theta)}
	&\leq \aiminnorm{\nabla\boldsymbol u\cdot\nabla\hat\theta} + \aiminnorm{\boldsymbol u\cdot\nabla(\nabla\hat\theta)}\\
	&\leq \aiminnorm{\nabla\boldsymbol u}_4\aiminnorm{\nabla\hat\theta}_4 + \aiminnorm{\boldsymbol u}_\infty\aiminnorm{\Delta\hat\theta}\\
	&\leq C\aiminnorm{\boldsymbol u}_{H^2}\aiminnorm{\Delta\hat\theta}.
\end{split}
\end{equation}

We estimate each term in the right-hand side of \eqref{eq2.4.12} as follows.

\noindent By Lemma~\ref{lem2.5.1}, we have
\begin{equation}\begin{split}\label{eq2.4.13.2}
	\aiminnorm{ \kappa(\theta)^{-1}\hat\theta_t }_{H^1}^2
	\leq C(\aiminnorm{\Delta\theta}_4^2 + 1)\aiminnorm{\hat\theta_t}_{H^1}^2
	\leq C(\aiminnorm{\Delta\theta}^2 + 1)( \aiminnorm{\nabla\hat\theta_t}^2 + \aiminnorm{\hat\theta_t}^2);
\end{split}\end{equation}
By Lemma~\ref{lem2.5.1} and the Sobolev embedding, we have
\begin{equation}\begin{split}
	\aiminnorm{\kappa(\theta)^{-1}\boldsymbol u\cdot\nabla\hat\theta  }_{H^1}^2
	&\leq  C(\aiminnorm{\Delta\theta}_4^2 + 1)\aiminnorm{ \boldsymbol u\cdot\nabla\hat\theta  }_{H^1}^2\\
	&\leq  C(\aiminnorm{\Delta\theta}_4^2 + 1)\aiminnorm{ \nabla(\boldsymbol u\cdot\nabla\hat\theta)  }^2
	\leq  C(\aiminnorm{\Delta\theta}^2 + 1)\aiminnorm{\boldsymbol u}_{H^2}^2\aiminnorm{\Delta\hat\theta}^2,
\end{split}\end{equation}
where we have used \eqref{eq2.4.13.1}; obviously, we have
\begin{equation}
	\aiminnorm{u_2}_{H^1}^2 \leq \aiminnorm{\boldsymbol u}_{H^1}^2.
\end{equation}
By Lemma~\ref{lem2.5.1},
\begin{equation}\begin{split}\label{eq2.4.13.5}
	\aiminnorm{\frac{\kappa'(\theta)}{\kappa(\theta)} \hat\theta_y }_{H^1}^2
	\leq C(\aiminnorm{\Delta\theta}^2 + 1)\aiminnorm{\hat\theta_y}_{H^1}^2
	\leq C(\aiminnorm{\Delta\theta}^2 + 1)\aiminnorm{\Delta\hat\theta}^2.
\end{split}\end{equation}
Combining all the estimates \eqref{eq2.4.13.2}-\eqref{eq2.4.13.5}, we infer from \eqref{eq2.4.12} that
\begin{equation}\label{eq2.4.16}
	\aiminnorm{\hat\theta}_{H^3}^2 \leq
	C(\aiminnorm{\Delta\theta}^2 + 1)( \aiminnorm{\Delta\hat\theta}^2+\aiminnorm{\boldsymbol u}_{H^2}^2\aiminnorm{\Delta\hat\theta}^2  + \aiminnorm{\hat\theta_t}^2 + \aiminnorm{\nabla\hat\theta_t}^2)+C\aiminnorm{\boldsymbol u}_{H^1}^2,
\end{equation}
which, together with Propositions~\ref{prop2.5.1},~\ref{prop2.5.3}-\ref{prop2.5.4}, implies that \eqref{eq2.4.16p} is valid for $\hat\theta$.

In order to show  that \eqref{eq2.4.16p} is valid for $\theta$, we utilize \eqref{eq2.3.7} and the third identity in \eqref{eq2.3.30.1} to obtain
\[
\aiminnorm{ \Delta\theta }_{H^1}^2 \leq 2\aiminnorm{ \kappa^{-1}(\theta)\Delta\hat\theta}_{H^1}^2 + 2\aiminnorm{ \frac{\kappa'(\theta)}{\kappa^3(\theta)}\aiminabs{\nabla\hat\theta}^2 }^2_{H^1},
\]
which, together with Lemma~\ref{lem2.5.1} and the Sobolev embedding, implies that
\begin{equation}\begin{split}\label{eq2.4.16.1}
	\aiminnorm{ \Delta\theta }_{H^1}^2
	&\leq C(\aiminnorm{\Delta\theta}^2 + 1)\big(\aiminnorm{ \Delta\hat\theta }_{H^1}^2 + \aiminnorm{ \aiminabs{\nabla\hat\theta}^2 }^2_{H^1}\big)\\
	&\leq C(\aiminnorm{\Delta\theta}^2 + 1)\big(\aiminnorm{ \hat\theta }_{H^3}^2 + \aiminnorm{ \nabla\hat\theta}_4^2 + \aiminnorm{ \nabla\hat\theta\cdot\nabla(\nabla\hat\theta)} \big)\\
	&\leq C(\aiminnorm{\Delta\theta}^2 + 1)\big(\aiminnorm{ \hat\theta }_{H^3}^2 + \aiminnorm{ \Delta\hat\theta }^2 + \aiminnorm{\Delta\hat\theta}\aiminnorm{\Delta\hat\theta}_{H^1} \big).
\end{split}\end{equation}
Then, with Proposition~\ref{prop2.5.3} and \eqref{eq2.4.16p} for $\hat\theta$, we infer from \eqref{eq2.4.16.1} that \eqref{eq2.4.16p} is also valid for $\theta$. This completes the proof of Proposition~\ref{prop2.5.5}.
\end{proof}

\begin{proposition}\label{prop2.5.6}
	Under the assumptions of Theorem~\ref{thm1.3}, we have
	\begin{equation}\label{eq2.4.20}
		\int_t^{t+1}\big( \aiminnorm{\boldsymbol u(s)}_{H^3}^2 + \aiminnorm{\tilde p(s) }_{H^2}^2\big){\mathrm{d}s} \leq C,\quad\forall\,t\geq  t_0+5.
	\end{equation}
\end{proposition}
\begin{proof}
In order to find the control of the time average of $\aiminnorm{\boldsymbol u}_{H^3}^2$, we employ the regularity results for Stokes system \eqref{eq2.4.46} (see \cite{Tem01}) again and obtain that
\begin{equation}\begin{split}\label{eq2.4.17}
	\aiminnorm{\boldsymbol u}_{H^3}^2 + \aiminnorm{\tilde p }_{H^2}^2
	\leq \aiminnorm{\boldsymbol f}_{H^1}^2,
\end{split}\end{equation}
where $\bf f$ is that in \eqref{eq2.4.46}.
We now need to estimate $\aiminnorm{\boldsymbol f}_{H^1}$.
We first recall that
\begin{equation}\begin{split}\label{eq2.4.18.0}
\aiminnorm{\theta}_{H^2}, \aiminnorm{\hat\theta}_{H^2}, \aiminnorm{\boldsymbol u}_{H^2},\aiminnorm{\tilde p}_{H^1} \leq C,\quad\forall\,t\geq  t_0+5,
\end{split}\end{equation}
and similar to \eqref{eq2.4.13.1}, we also have
\begin{equation}\begin{split}\label{eq2.4.18.1}
	\aiminnorm{\nabla(\boldsymbol u\cdot\nabla\boldsymbol u)}
	\leq C\aiminnorm{\boldsymbol u}_{H^2}\aiminnorm{\Delta\boldsymbol u} \leq C\aiminnorm{\boldsymbol u}_{H^2}^2.
\end{split}\end{equation}
The estimates \eqref{eq2.4.18.2}-\eqref{eq2.4.18.6} below are valid only for $t\geq  t_0+5$ since we employed the uniform bounds \eqref{eq2.4.18.0}.

\noindent By Lemma~\ref{lem2.5.1} and Lemma~\ref{lem2.1.5}, we have
\begin{equation}\begin{split}\label{eq2.4.18.2}
	&\aiminnorm{ \frac{\nu'(\theta) }{\nu(\theta)}\tilde p\nabla\theta }_{H^1}^2
	=
	\aiminnorm{ \frac{\nu'(\theta) }{\nu(\theta)\kappa(\theta)}\tilde p\nabla\hat\theta }_{H^1}^2
	\leq C(\aiminnorm{\Delta\theta}^2 + 1)\aiminnorm{ \tilde p\nabla\hat\theta }_{H^1}^2\\
	&\quad\leq C\big[\aiminnorm{\tilde p}_{H^1}^2\aiminnorm{\nabla\hat\theta}_{H^1}^2(1+\aiminnorm{\tilde p}_{H^1}^2+\aiminnorm{\nabla\hat\theta}_{H^1}^2) + \epsilon (\aiminnorm{\tilde p}_{H^2}^2+\aiminnorm{\nabla\hat\theta}_{H^2}^2 ) \big]\\
	&\quad\leq C(C+ \epsilon (\aiminnorm{\tilde p}_{H^2}^2+\aiminnorm{\nabla\hat\theta}_{H^2}^2 ).
\end{split}\end{equation}
\noindent By Lemma~\ref{lem2.5.1} and Lemma~\ref{lem2.1.5}, we have
\begin{equation}\begin{split}
	&\aiminnorm{  \frac{\nu'(\theta) }{\nu(\theta)}\nabla\theta\nabla\boldsymbol u }_{H^1}^2
	=	\aiminnorm{  \frac{\nu'(\theta) }{\nu(\theta)\kappa(\theta)}\nabla\hat\theta\nabla\boldsymbol u }_{H^1}^2
	\leq C(\aiminnorm{\Delta\theta}_4^2 + 1)\aiminnorm{ \nabla\hat\theta\nabla\boldsymbol u}_{H^1}^2\\
	&\;\leq C\big[\aiminnorm{\nabla\hat\theta}_{H^1}^2\aiminnorm{\nabla\boldsymbol u}_{H^1}^2(1+\aiminnorm{\nabla\hat\theta}_{H^1}^2+\aiminnorm{\nabla\boldsymbol u}_{H^1}^2)+\epsilon (\aiminnorm{\nabla\hat\theta}_{H^2}^2+\aiminnorm{\nabla\boldsymbol u}_{H^2}^2)\big]\\
	&\;\leq C(C+ \epsilon (\aiminnorm{\nabla\hat\theta}_{H^2}^2 +\aiminnorm{\nabla\boldsymbol u}_{H^2}^2)).
\end{split}\end{equation}
\noindent By Lemma~\ref{lem2.5.1}, the Sobolev embedding, and \eqref{eq2.4.18.1}, we have
\begin{equation}\begin{split}
	\aiminnorm{ \frac{1}{\nu(\theta)}\boldsymbol u\cdot\nabla\boldsymbol u}_{H^1}^2
	&\leq C(\aiminnorm{\Delta\theta}_4^2 + 1)\aiminnorm{ \boldsymbol u\cdot\nabla\boldsymbol u  }_{H^1}^2\\
	&\leq C\aiminnorm{\nabla( \boldsymbol u\cdot\nabla\boldsymbol u ) }^2
	\leq C\aiminnorm{\boldsymbol u}_{H^2}^2 \leq C.
\end{split}\end{equation}
\noindent By Lemma~\ref{lem2.5.1}, we have
\begin{equation}\begin{split}\label{eq2.4.18.6}
	\aiminnorm{  \frac{1}{\nu(\theta)} \boldsymbol u_t   }_{H^1}^2 &+ \aiminnorm{  \frac{1}{\nu(\theta)} \theta\boldsymbol e_2   }_{H^1}^2 + \aiminnorm{  \frac{1}{\nu(\theta)} (1-y)\boldsymbol e_2 }_{H^1}^2\\
	&\leq C(\aiminnorm{\Delta\theta}^2 +1)( \aiminnorm{\boldsymbol u_t}_{H^1}^2 + \aiminnorm{\theta}_{H^1}^2 + \aiminnorm{1-y}_{H^1}^2)\\
	&\leq C(\aiminnorm{\Delta\theta}^2 +1)( \aiminnorm{\boldsymbol u_t}^2 + \aiminnorm{\nabla\boldsymbol u_t}^2+ \aiminnorm{\theta}_{H^1}^2 + 2  )\\
	&\leq C(1 + \aiminnorm{\nabla\boldsymbol u_t}^2 ).
\end{split}\end{equation}

Therefore,  combining the estimates \eqref{eq2.4.18.2}-\eqref{eq2.4.18.6} and choosing $\epsilon$ small enough, we find from \eqref{eq2.4.17} that
\begin{equation}\label{eq2.4.19}
	\aiminnorm{\boldsymbol u}_{H^3}^2 + \aiminnorm{\tilde p }_{H^2}^2
	\leq
	C + C\aiminnorm{\nabla \boldsymbol u_t}^2 + \aiminnorm{\hat\theta}_{H^3}^2,\quad\forall\,t\geq  t_0+5.
\end{equation}
which implies \eqref{eq2.4.20} by Propositions~\ref{prop2.5.2},~\ref{prop2.5.5}. This completes the proof of Proposition~\ref{prop2.5.6}.
\end{proof}

\section{Continuity  with respect to the initial data}\label{sec3}
The existence of global strong solutions for the Boussinesq system \eqref{eq1.1.3}-\eqref{eq1.1.4}  follows from the uniform estimates proved in Section~\ref{sec2}. In order to finish proving Theorem~\ref{thm1.2}, we are left to show the uniqueness and continuity of the strong solutions. The continuity of the strong solutions will be proved in Section~\ref{sec4} and
in this section, we are going to prove the uniqueness of the strong solutions. Note that all the  estimates above are valid on any interval $[0, t_1]$. The fact that they are stated for $t\geq t_0+5$  was meant to get bounds independent of the initial data to obtain the corresponding absorbing sets.

In this section,  we will show that for any fixed $t>0$, the mapping
\begin{equation}\label{eq3.1.1}
	(\boldsymbol u_0,\theta_0)=(\boldsymbol u(0),\theta(0))\mapsto (\boldsymbol u(t),\theta(t))	
\end{equation}
is Lipschitz continuous from $H^2(\Omega)$ into itself for all strong solutions. Suppose that for $i=1,2$, $(\boldsymbol u^{(i)},p^{(i)}, \theta^{(i)})$ are two strong solutions to the Boussinesq system \eqref{eq1.1.3}-\eqref{eq1.1.4}, with initial data $(\boldsymbol u_0^{(i)},\theta_0^{(i)})$ belonging to $H_0^1(\Omega)\cap H^2(\Omega)$. Let
\begin{equation}\label{eq3.1.0}
	\boldsymbol v=\boldsymbol u^{(1)} - \boldsymbol u^{(2)},\quad\quad \eta = \theta^{(1)} - \theta^{(2)},\quad\quad p=p^{(1)}-p^{(2)}.
\end{equation}
Then we have the following equations for $\boldsymbol v$ and $\eta$:
\begin{equation}\label{eq3.1.2}
	\begin{cases}
		\partial_t \boldsymbol v - {\mathrm{div}\,}(\nu(\theta^{(1)})\nabla\boldsymbol v) +\boldsymbol u^{(1)}\cdot\nabla\boldsymbol v + \nabla p
		\\
		 \hspace{100pt} = - {\mathrm{div}\,}(( \nu(\theta^{(2)}) - \nu(\theta^{(1)}) ) \nabla\boldsymbol u^{(2)})-\boldsymbol v\cdot\nabla \boldsymbol u^{(2)}  +\eta\boldsymbol e_2,\\
		 {\mathrm{div}\,}\boldsymbol v=0,\\
		 \partial_t\eta -{\mathrm{div}\,}(\kappa(\theta^{(1)})\nabla\eta) + \boldsymbol u^{(1)}\cdot\nabla\eta-v_2 \\
		 \hspace{100pt}= -{\mathrm{div}\,}( ( \kappa(\theta^{(2)}) - \kappa(\theta^{(1)}) ) \nabla\theta^{(2)})- \boldsymbol v\cdot\nabla\theta^{(2)}\\
		 \hspace{115pt}-\kappa'(\theta^{(1)} )\theta^{(1)}_y + \kappa'(\theta^{(2)} )\theta^{(2)}_y,
	\end{cases}
\end{equation}
which, by expanding the divergence terms in the left-hand sides of the equations, is equivalent to:
\begin{equation}\label{eq3.1.2p}
	\begin{cases}
		\partial_t \boldsymbol v - \nu(\theta^{(1)})\Delta\boldsymbol v - \nu'(\theta^{(1)})\nabla \theta^{(1)}\cdot\nabla\boldsymbol v +\boldsymbol u^{(1)}\cdot\nabla\boldsymbol v + \nabla p
		\\
		 \hspace{100pt} = - {\mathrm{div}\,}(( \nu(\theta^{(2)}) - \nu(\theta^{(1)}) ) \nabla\boldsymbol u^{(2)})-\boldsymbol v\cdot\nabla \boldsymbol u^{(2)}  +\eta\boldsymbol e_2,\\
		 {\mathrm{div}\,}\boldsymbol v=0,\\
		 \partial_t\eta -\kappa(\theta^{(1)})\Delta\eta -  \kappa'(\theta^{(1)})\nabla \theta^{(1)}\cdot\nabla\eta+ \boldsymbol u^{(1)}\cdot\nabla\eta -v_2\\
		 \hspace{100pt}= -{\mathrm{div}\,}( ( \kappa(\theta^{(2)}) - \kappa(\theta^{(1)}) ) \nabla\theta^{(2)})- \boldsymbol v\cdot\nabla\theta^{(2)}\\
		 \hspace{115pt}-\kappa'(\theta^{(1)} )\theta^{(1)}_y + \kappa'(\theta^{(2)} )\theta^{(2)}_y.
	\end{cases}
\end{equation}

We first show that the mapping \eqref{eq3.1.1} is Lipschitz continuous in the space $H_0^1(\Omega)$ and then extend the result to the space $H^2(\Omega)$. In this section, we denote by
$$\mathcal Q:=\mathcal Q\big(\aiminnorm{\boldsymbol u_0^{(1)}}_{H^2},\,\aiminnorm{\boldsymbol u_0^{(2)}}_{H^2},\, \aiminnorm{\theta_0^{(1)}}_{H^2},\,\aiminnorm{\theta_0^{(2)}}_{H^2}\big)$$
a positive constant which depends in an increasing manner on the $H^2$-norms of the initial data $\boldsymbol u_0^{(i)}$ and $\theta_0^{(i)}$ ($i=1,2$), but is independent of time $t$, and the constant $\mathcal Q$ may vary from line to line. By the global well-posedness result (Theorem~\ref{thm1.2}), we have
\begin{equation}\label{eq3.1.3}
\begin{split}
	\aiminnorm{ (\boldsymbol u^{(i)}(t),\theta^{(i)}(t)) }_{H^2},\,\aiminnorm{ (\boldsymbol u_t^{(i)}(t),\theta_t^{(i)}(t)) }_{L^2},\,\aiminnorm{ p^{(i)}(t) }_{H^1} \leq \mathcal Q,\quad\forall\, i=1,2,\quad\forall\,t\geq 0.\\
\end{split}	
\end{equation}
By \eqref{eq3.1.3}, we also have
\begin{equation}\label{eq3.1.4}
\begin{split}
	\aiminnorm{ (\boldsymbol v(t),\eta(t)) }_{H^2},\,\aiminnorm{ (\boldsymbol v_t(t),\eta_t(t)) },\,\aiminnorm{ p(t) }_{H^1} \leq \mathcal Q,\quad\forall\,t\geq 0.\\
\end{split}	
\end{equation}
By Remark~\ref{rmk2.0.1} and the Sobolev embedding, we also have
\begin{equation}\label{eq3.1.5}
	\aiminnorm{ \theta^{(i)} }_{L^\infty((0, t_1)\times\Omega)}\leq \mathcal Q,\quad\forall\,i=1,2,\quad\forall\, t_1>0.
\end{equation}
In the remaining of this section, we should bear in mind the estimates \eqref{eq3.1.3}-\eqref{eq3.1.5} which will be used frequently without referring to them.

\subsection{Continuity in $H_0^1(\Omega)$}
We are now going to establish the Lipschitz continuous result for the mapping \eqref{eq3.1.1} from $H^2(\Omega)$ to $H^1_0(\Omega)$.
Taking the inner product of  \eqref{eq3.1.2p}$_{1,3}$ with $(-\Delta\boldsymbol v,-\Delta\eta)$ in $L^2(\Omega)$ yields
\begin{equation}\begin{split}\label{eq3.2.1}
	\frac12\frac{\mathrm{d}}{\mathrm{d}t}\aiminnorm{\nabla\boldsymbol v}^2 + \int_\Omega\nu(\theta^{(1)})\aiminabs{\Delta\boldsymbol v}^2&{\mathrm{d}x}{\mathrm{d}y}
	= \int_\Omega {\mathrm{div}\,}(( \nu(\theta^{(2)}) - \nu(\theta^{(1)}) ) \nabla\boldsymbol u^{(2)})\cdot\Delta\boldsymbol v{\mathrm{d}x}{\mathrm{d}y}\\
&+\int_\Omega  \boldsymbol v\cdot\nabla \boldsymbol u^{(2)}\cdot\Delta\boldsymbol v{\mathrm{d}x}{\mathrm{d}y} +\int_\Omega \boldsymbol u^{(1)}\cdot\nabla\boldsymbol v\cdot\Delta\boldsymbol v{\mathrm{d}x}{\mathrm{d}y} \\
	&-\int_\Omega \nu'(\theta^{(1)})\nabla \theta^{(1)}\cdot\nabla\boldsymbol v \cdot\Delta \boldsymbol v{\mathrm{d}x}{\mathrm{d}y} - \int_\Omega\eta\boldsymbol e_2\cdot \Delta\boldsymbol v{\mathrm{d}x}{\mathrm{d}y}\\
	=&:I_1 + I_2 +I_3 + I_4 + I_5.
\end{split}\end{equation}
\begin{equation}\begin{split}\label{eq3.2.2}
	\frac12\frac{\mathrm{d}}{\mathrm{d}t}\aiminnorm{\nabla\eta}^2 + \int_\Omega \kappa(\theta^{(1)})\aiminabs{\Delta\eta}^2&{\mathrm{d}x}{\mathrm{d}y}
	=	\int_\Omega {\mathrm{div}\,}( ( \kappa(\theta^{(2)}) - \kappa(\theta^{(1)}) ) \nabla\theta^{(2)})\cdot\Delta\eta{\mathrm{d}x}{\mathrm{d}y}\\
	&+\int_\Omega \boldsymbol v\cdot\nabla\theta^{(2)}\cdot\Delta\eta{\mathrm{d}x}{\mathrm{d}y} + \int_\Omega \boldsymbol u^{(1)}\cdot\nabla\eta\cdot\Delta\eta{\mathrm{d}x}{\mathrm{d}y} \\
	& -\int_\Omega \kappa'(\theta^{(1)})\nabla \theta^{(1)}\cdot\nabla\eta \cdot\Delta \eta{\mathrm{d}x}{\mathrm{d}y} + \int_\Omega v_2\cdot\Delta\eta{\mathrm{d}x}{\mathrm{d}y}\\
	&+\int_\Omega (-\kappa'(\theta^{(1)} )\theta^{(1)}_y + \kappa'(\theta^{(2)} )\theta^{(2)}_y) \cdot\Delta\eta{\mathrm{d}x}{\mathrm{d}y}\\
	=&:J_1 + J_2 + J_3 + J_4+J_5+J_6.
\end{split}\end{equation}

As a preliminary result, we observe by \eqref{eq3.1.5}  that for any $\chi(\tau)\in \mathcal C(\mathbb R)$, we have the following inequalities:
\begin{equation}\label{eq3.2.3}
\begin{split}
	\aiminabs{\chi(\theta^{(2)}) - \chi(\theta^{(1)})} &\leq \mathcal Q \aiminabs{\theta^{(2)} - \theta^{(1)} }=\mathcal Q\aiminabs{\eta},\qquad\text{a.e. }(x,y)\in\Omega, \text{ a.e }t\in[0, t_1].\\
	\aiminabs{\chi(\theta^{(i)})}&\leq \mathcal Q,\quad\forall\,i=1,2,\qquad\text{a.e. }(x,y)\in\Omega, \text{ a.e }t\in[0, t_1].
\end{split}
\end{equation}
We also provide two useful lemmas.
\begin{lemma}\label{lem3.2.1}
	Let $\chi(\tau)\in\mathcal C^2(\mathbb R)$, $g\in H_0^1(\Omega)\cap H^2(\Omega)$, and $\theta^{(1)}, \theta^{(2)}$ be those in \eqref{eq3.1.0}. Then
	\begin{equation}\begin{split}\label{lem3.2.1.eq0}
		\aiminnorm{{\mathrm{div}\,}\big((\chi( \theta^{(2)})-\chi(\theta^{(1)}) )\nabla g\big)}^2
		\leq \mathcal Q\aiminnorm{\nabla\eta}\aiminnorm{\Delta\eta}\aiminnorm{\Delta g}^2 + \mathcal Q\aiminnorm{\nabla\eta}^2 \aiminnorm{\Delta g}^2.
	\end{split}\end{equation}
\end{lemma}
\begin{proof}
We first write
\begin{equation*}\begin{split}
	{\mathrm{div}\,}((\chi( \theta^{(2)})-\chi(\theta^{(1)}) )&\nabla g)=\big(\chi( \theta^{(2)})-\chi(\theta^{(1)})\big)\Delta g \\
	&+ \bigg[\big(\chi'( \theta^{(2)})-\chi'(\theta^{(1)})\big) \nabla\theta^{(2)} +  \chi'(\theta^{(1)})( \nabla\theta^{(2)} - \nabla\theta^{(1)}) \bigg]\nabla g,
\end{split}\end{equation*}
then by using \eqref{eq3.2.3}, we obtain
\begin{equation}\begin{split}\label{lem3.2.1.eq1}
	\aiminnorm{{\mathrm{div}\,}\big((\chi( \theta^{(1)})-\chi(\theta^{(2)}) )\nabla g\big)}^2
	\leq \mathcal Q\big( \aiminnorm{\eta\Delta g}^2 + \aiminnorm{ \eta \nabla\theta^{(2)} \nabla g}^2 + \aiminnorm{ \nabla\eta \nabla g}^2\big),
\end{split}\end{equation}
where $\eta=\theta^{(1)}-\theta^{(2)}$. By Agmon's inequality, we have
\begin{equation}
	\aiminnorm{\eta\Delta g}^2\leq \aiminnorm{\eta}_\infty^2 \aiminnorm{\Delta g}^2 \leq C\aiminnorm{\nabla\eta}\aiminnorm{\Delta\eta}\aiminnorm{\Delta g}^2;
\end{equation}
by H\"older's inequality and the Sobolev embedding, we have
\begin{equation}
	\aiminnorm{ \eta \nabla\theta^{(2)} \nabla g}^2
	\leq \aiminnorm{\eta}_8^2\aiminnorm{\nabla\theta^{(2)}}_8^2\aiminnorm{\nabla g}_4^2
	\leq C\aiminnorm{\nabla\eta}^2 \aiminnorm{\Delta\theta^{(2)}}^2 \aiminnorm{\Delta g}^2;
\end{equation}
by H\"older's inequality and the Sobolev embedding, we have
\begin{equation}\label{lem3.2.1.eq5}
	\aiminnorm{ \nabla\eta \nabla g}^2\leq \aiminnorm{\nabla\eta}_4^2 \aiminnorm{\nabla g}_4^2
	\leq C\aiminnorm{\nabla\eta}\aiminnorm{\Delta\eta}\aiminnorm{\Delta g}^2.
\end{equation}
Combining the estimates \eqref{lem3.2.1.eq1}-\eqref{lem3.2.1.eq5} together and utilizing \eqref{eq3.1.3} for $\aiminnorm{\Delta\theta^{(2)}}$ yield \eqref{lem3.2.1.eq0}. This completes the proof of Lemma~\ref{lem3.2.1}.
\end{proof}
\begin{lemma}\label{lem3.2.2}
	Fix $q$ such that $2\leq q<+\infty$ and let $\chi=\chi(\cdot)\in\mathcal C(\mathbb R)$, $g^{(1)}, g^{(2)}\in H^1(\Omega)$, and $\theta^{(1)}, \theta^{(2)}$ be those in \eqref{eq3.1.0}. Then
	\begin{equation}\begin{split}\label{lem3.2.2.eq0}
		\aiminnorm{ \chi( \theta^{(1)})g^{(1)}-\chi(\theta^{(2)})g^{(2)} }_q
		\leq  \mathcal Q\aiminnorm{ g^{(1)}}_{H^1} \aiminnorm{\nabla\eta} +  \mathcal Q\aiminnorm{ g^{(1)} - g^{(2)} }_q.
	\end{split}\end{equation}
\end{lemma}
\begin{proof}
	We first write
	\[
	\chi( \theta^{(1)})g^{(1)}-\chi(\theta^{(2)})g^{(2)} =
	(\chi( \theta^{(1)})-\chi(\theta^{(2)}))g^{(1)} + \chi( \theta^{(2)})(g^{(1)}-g^{(2)});
	\]
	then by \eqref{eq3.2.3}, we obtain
	\begin{equation*}
	\begin{split}
		\aiminnorm{ \chi( \theta^{(1)})g^{(1)}-\chi(\theta^{(2)})g^{(2)} }_q
		&\leq \mathcal Q\aiminnorm{\eta g^{(1)}}_q + \mathcal Q\aiminnorm{ g^{(1)} - g^{(2)} }_q\\
		&\leq \mathcal Q\aiminnorm{\eta}_{2q}\aiminnorm{ g^{(1)}}_{2q} +  \mathcal Q\aiminnorm{ g^{(1)} - g^{(2)} }_q,
	\end{split}		
	\end{equation*}
	which, together with the Sobolev embedding, implies \eqref{lem3.2.2.eq0}.
\end{proof}

We are now in position to return to \eqref{eq3.2.1}-\eqref{eq3.2.2} and we first estimate the right-hand side of \eqref{eq3.2.1} term by term.
For the first term $I_1$ in the right-hand side of \eqref{eq3.2.1}, we apply Lemma~\ref{lem3.2.1} with $\chi=\nu$, $g=\boldsymbol u^{(2)}$ and Young's inequality:
\begin{equation}\begin{split}\label{eq3.2.6}
	I_1
	&\leq \aiminnorm{{\mathrm{div}\,}\big((\nu( \theta^{(2)})-\nu(\theta^{(1)}) )\nabla \boldsymbol u^{(2)}\big)} \aiminnorm{ \Delta\boldsymbol v}\\
	&\leq C\aiminnorm{{\mathrm{div}\,}\big((\nu( \theta^{(2)})-\nu(\theta^{(1)}) )\nabla \boldsymbol u^{(2)}\big)}^2 + \frac{\underline\nu}{16}\aiminnorm{\Delta\boldsymbol v}^2\\
	&\leq \mathcal Q\aiminnorm{\nabla\eta}\aiminnorm{\Delta\eta}\aiminnorm{\Delta \boldsymbol u^{(2)}}^2 + \mathcal Q\aiminnorm{\nabla\eta}^2 \aiminnorm{\Delta \boldsymbol u^{(2)}}^2 + \frac{\underline\nu}{16}\aiminnorm{\Delta\boldsymbol v}^2\\
	&\leq \mathcal Q\aiminnorm{\nabla\eta}^2  + \frac{\underline\nu}{16}\aiminnorm{\Delta\boldsymbol v}^2 + \frac{\underline\kappa}{16}\aiminnorm{\Delta\eta}^2.
\end{split}\end{equation}
For the second term $I_2$, we use the Sobolev embedding and Young's inequality:
\begin{equation}
	\begin{split}
		I_2 &\leq \aiminnorm{\boldsymbol v}_4\aiminnorm{ \nabla\boldsymbol u^{(2)} }_4 \aiminnorm{\Delta\boldsymbol v} \leq \aiminnorm{\nabla\boldsymbol v}\aiminnorm{ \Delta\boldsymbol u^{(2)} }\aiminnorm{\Delta\boldsymbol v}\\
		&\leq \mathcal Q\aiminnorm{\nabla\boldsymbol v}^2 + \frac{\underline\nu}{16} \aiminnorm{\Delta\boldsymbol v}^2.
	\end{split}
\end{equation}
For the third and fourth terms $I_3$ and $I_4$, applying Lemma~\ref{lem2.1.3} with $f=\boldsymbol u^{(1)}$ for $I_3$, $f= \nabla\theta^{(1)}$ for $I_4$, and $g=\nabla\boldsymbol v$ implies that
\begin{equation}
	I_3\leq C\aiminnorm{\boldsymbol u^{(1)}}^2\aiminnorm{\nabla\boldsymbol u^{(1)}}^2\aiminnorm{\nabla\boldsymbol v}^2 +  \frac{\underline\nu}{16} \aiminnorm{\Delta\boldsymbol v}^2
	\leq \mathcal Q\aiminnorm{\nabla\boldsymbol v}^2 + \frac{\underline\nu}{16} \aiminnorm{\Delta\boldsymbol v}^2,
\end{equation}
\begin{equation}
	I_4 \leq C \aiminnorm{ \nabla\theta^{(1)}  }^2\aiminnorm{ \Delta\theta^{(1)} }^2 \aiminnorm{\nabla\boldsymbol v}^2  + \frac{\underline\nu}{16} \aiminnorm{\Delta\boldsymbol v}^2
	\leq \mathcal Q\aiminnorm{\nabla\boldsymbol v}^2 + \frac{\underline\nu}{16} \aiminnorm{\Delta\boldsymbol v}^2.
\end{equation}
By Poincar\'e's inequality and the Cauchy-Schwarz inequality, we estimate the last term $I_5$:
\begin{equation}
	I_5\leq \aiminnorm{\eta}\aiminnorm{\Delta\boldsymbol v}\leq C\aiminnorm{\nabla\eta}^2 + \frac{\underline\nu}{16} \aiminnorm{\Delta\boldsymbol v}^2.
\end{equation}

We then estimate the right-hand side of \eqref{eq3.2.2} term by term. Proceeding similarly as for the $I_i$ , we can obtain the following estimates for $J_i$, for all $i=1,\cdots,5$:
\begin{equation}\begin{split}
	J_1,\; J_3,\; J_4&\leq \mathcal Q \aiminnorm{\nabla \eta}^2  + \frac{\underline\nu}{16} \aiminnorm{\Delta\eta}^2,\\
	J_2,\; J_5&\leq C\aiminnorm{\nabla\boldsymbol v}^2+ \frac{\underline\nu}{16} \aiminnorm{\Delta\eta}^2.\\
\end{split}\end{equation}
We are left to estimate $J_6$. Using the Cauchy-Schwarz inequality and Young's inequality, and applying Lemma~\ref{lem3.2.2} with $q=2$, $\chi=\kappa'$, and $g^{(i)}=\theta_y^{(i)}$ for $i=1,2$ yield that
\begin{equation}\begin{split}\label{eq3.2.9}
	J_6 &\leq \aiminnorm{ \kappa'(\theta^{(1)} )\theta^{(1)}_y - \kappa'(\theta^{(2)} )\theta^{(2)}_y }\aiminnorm{\Delta\eta}\\
	&\leq \mathcal Q\big( \aiminnorm{ \theta^{(1)}_y }_{H^1} \aiminnorm{\nabla\eta} + \aiminnorm{ \theta^{(1)}_y- \theta^{(2)}_y }\big)\aiminnorm{\Delta\eta}\\
	&\leq \mathcal Q\aiminnorm{\nabla\eta}\aiminnorm{\Delta\eta}
	\leq\mathcal Q \aiminnorm{\nabla \eta}^2  + \frac{\underline\nu}{16} \aiminnorm{\Delta\eta}^2,
\end{split}\end{equation}

Now summing \eqref{eq3.2.1} and \eqref{eq3.2.2} and using all the estimates \eqref{eq3.2.6}-\eqref{eq3.2.9} for the $I_i$'s ($i=1,\cdots,5$) and the $J_j$'s ($j=1,\cdots,6$), and denoting
\begin{equation}\begin{split}\label{eq3.2.9-1}
y_1(t):=\aiminnorm{\nabla\boldsymbol v}^2 + \aiminnorm{\nabla\eta}^2,\\
\end{split}\end{equation}
we arrive at
\begin{equation}\begin{split}\label{eq3.2.10}
	\frac{\mathrm{d}}{\mathrm{d}t} y_1(t)  +\big(\underline\nu \aiminnorm{\Delta\boldsymbol v}^2 + \underline\kappa\aiminnorm{\Delta\eta}^2\big)
	\leq \mathcal Q y_1(t).
\end{split}\end{equation}
Therefore, by the usual Gronwall lemma, we find that
\begin{equation}\label{eq3.2.11}
y_1(t)\leq y_1(0)\exp\big\{\mathcal Q t\big\}.	
\end{equation}

We have thus proved that for fixed $t>0$, the mapping \eqref{eq3.1.1} is Lipschitz continuous from $H^2(\Omega)$ to $H^1(\Omega)$ with respect to the initial data $(\boldsymbol u(0),\theta(0))$.

Integrating \eqref{eq3.2.10} in time on $(0,t)$ gives
\begin{equation}\begin{split}\label{eq3.2.12}
	\int_0^t\big(\underline\nu \aiminnorm{\Delta\boldsymbol v}^2 + \underline\kappa\aiminnorm{\Delta\eta}^2\big){\mathrm{d}s}
	&\leq \int_0^t \mathcal Q y_1(s){\mathrm{d}s} + y_1(0)\\
	&\leq y_1(0)\int_0^t \mathcal Q \exp\big\{\mathcal Q s\big\} {\mathrm{d}s} + y_1(0)\\
	&= y_1(0)\exp\big\{\mathcal Q t\big\}.
\end{split}\end{equation}

\subsection{Continuity in $H^2(\Omega)$}
In order to show the continuity in $H^2(\Omega)$, we first need to show that the mapping
\begin{equation}\label{eq3.3.0}
	(\boldsymbol u_0,\theta_0)=(\boldsymbol u(0),\theta(0))\mapsto (\boldsymbol u_t(t),\theta_t(t))	
\end{equation}
is Lipschitz continuous from $H^2(\Omega)$ into $L^2(\Omega)$ for all solutions of the Boussinesq system \eqref{eq1.1.3}-\eqref{eq1.1.4}. Let us first find out the relation between $(\Delta\boldsymbol v,\Delta\eta)$ and $(\boldsymbol v_t,\eta_t)$.
By the classical elliptic regularity theory (see e.g. \cite{Eva98}), we  infer from \eqref{eq3.1.2p}$_3$ that
\begin{equation}\begin{split}\label{eq3.3.1}
	\underline\kappa\aiminnorm{\Delta\eta} \leq &C\big(\aiminnorm{\eta_t} + \aiminnorm{\nabla \theta^{(1)}\cdot\nabla\eta }  + \aiminnorm{\boldsymbol u^{(1)}\cdot\nabla\eta } + \aiminnorm{v_2}+ \aiminnorm{\boldsymbol v\cdot\nabla \theta^{(2)}}\\
	& + \aiminnorm{ {\mathrm{div}\,}( ( \kappa(\theta^{(2)}) - \kappa(\theta^{(1)}) ) \nabla\theta^{(2)})    } + \aiminnorm{ \kappa'(\theta^{(1)} )\theta^{(1)}_y - \kappa'(\theta^{(2)} )\theta^{(2)}_y }\big).
\end{split}\end{equation}
We utilize H\"older's inequality, Ladyzhenskaya's inequalities, Young's inequality, the Sobolev embedding and the bounds \eqref{eq3.1.3} to estimate the right-hand side of \eqref{eq3.3.1} term by term:
\begin{equation}\begin{split}\label{eq3.3.2}
	\aiminnorm{\nabla \theta^{(1)}\cdot\nabla\eta }
	\leq \aiminnorm{\nabla \theta^{(1)} }_4\aiminnorm{\nabla\eta }_4
	\leq C\aiminnorm{\Delta \theta^{(1)}}\aiminnorm{\nabla\eta}^{1/2}\aiminnorm{\Delta\eta}^{1/2}
	\leq \mathcal Q\aiminnorm{\nabla\eta} + \epsilon \aiminnorm{\Delta\eta};
\end{split}\end{equation}
\begin{equation}
	 \aiminnorm{\boldsymbol u^{(1)}\cdot\nabla\eta } \leq \aiminnorm{\boldsymbol u^{(1)} }_\infty \aiminnorm{\nabla\eta}
	 \leq \aiminnorm{ \boldsymbol u^{(1)}  }_{H^2}\aiminnorm{\nabla\eta}
	 \leq \mathcal Q\aiminnorm{\nabla\eta};
\end{equation}
\begin{equation}
	\aiminnorm{v_2}\leq \aiminnorm{\nabla\boldsymbol v};
\end{equation}
\begin{equation}
	\aiminnorm{\boldsymbol v\cdot\nabla \theta^{(2)}} \leq \aiminnorm{\boldsymbol v}_4\aiminnorm{ \nabla \theta^{(2)} }_4\leq \aiminnorm{\nabla\boldsymbol v}\aiminnorm{\Delta \theta^{(2)}}
	\leq \mathcal Q \aiminnorm{\nabla\boldsymbol v};
\end{equation}
applying Lemma~\ref{lem3.2.1} with $\chi=\kappa$ and $g=\theta^{(2)}$ and Young's inequality for the sixth-term in the right-hand side of \eqref{eq3.3.1} yields
\begin{equation}\begin{split}\label{eq3.3.4}
	\aiminnorm{ {\mathrm{div}\,}( ( \kappa(\theta^{(2)}) - \kappa(\theta^{(1)}) ) \nabla\theta^{(2)}) }
	&\leq \mathcal Q\aiminnorm{\nabla\eta}^{1/2}\aiminnorm{\Delta\eta}^{1/2}\aiminnorm{\Delta \theta^{(2)}} + \mathcal Q\aiminnorm{\nabla\eta} \aiminnorm{\Delta \theta^{(2)}}\\
	&\leq \mathcal Q \aiminnorm{\nabla\eta} + \epsilon \aiminnorm{\Delta\eta}.
\end{split}\end{equation}
Similar to \eqref{eq3.2.9}, applying Lemma~\ref{lem3.2.2} with $q=2$, $\chi=\kappa'$, and $g^{(i)}=\theta_y^{(i)}$ for $i=1,2$ yields
\begin{equation}\label{eq3.3.5}
	\aiminnorm{ \kappa'(\theta^{(1)} )\theta^{(1)}_y - \kappa'(\theta^{(2)} )\theta^{(2)}_y }
	\leq \mathcal Q\aiminnorm{\nabla\eta}.
\end{equation}

Combining all the estimates \eqref{eq3.3.2}-\eqref{eq3.3.5} and taking $\epsilon$ small enough, we conclude from \eqref{eq3.3.1} that
\begin{equation}\label{eq3.3.8}
	\aiminnorm{\Delta\eta} \leq \mathcal Q\big( \aiminnorm{\eta_t} + \aiminnorm{\nabla\eta} + \aiminnorm{\nabla\boldsymbol v} \big),\quad\forall\,t\geq 0.
\end{equation}

In order to find a similar relation for $\aiminnorm{\Delta\boldsymbol v}$ and $\aiminnorm{\boldsymbol v_t}$, we take the inner product of \eqref{eq3.1.2p}$_{1}$ with $\Delta\boldsymbol v$ in $L^2(\Omega)$ and note that by \eqref{eq1.1.10}, we have
\[
\int_\Omega \nabla p\cdot \Delta\boldsymbol v{\mathrm{d}x}{\mathrm{d}y} = 0,
\]
thanks to the free-boundary conditions; we arrive at
\begin{equation}\label{eq3.3.9.0}
	\int_\Omega \nu(\theta^{(1)})\aiminabs{\Delta\boldsymbol v}^2{\mathrm{d}x}{\mathrm{d}y} = \int_\Omega \boldsymbol f\cdot\Delta\boldsymbol v{\mathrm{d}x}{\mathrm{d}y},
\end{equation}
where
\[
\boldsymbol f=-\boldsymbol v_t - \nu'(\theta^{(1)})\nabla \theta^{(1)}\cdot\nabla\boldsymbol v-\boldsymbol u^{(1)}\cdot\nabla\boldsymbol v -\boldsymbol v\cdot\nabla \boldsymbol u^{(2)}  +\eta\boldsymbol e_2 - {\mathrm{div}\,}(( \nu(\theta^{(2)}) - \nu(\theta^{(1)}) ) \nabla\boldsymbol u^{(2)}).
\]
By the Cauchy-Schwarz inequality, we infer from \eqref{eq3.3.9.0} that
\begin{equation}\label{eq3.3.9}
	\underline\nu\aiminnorm{\Delta\boldsymbol v} \leq  \aiminnorm{\boldsymbol f},\quad\forall\,t\geq 0.
\end{equation}
Using Lemma~\ref{lem2.1.1} and Young's inequality, we estimate $\aiminnorm{\boldsymbol f}$ term by term:
\begin{equation}
\begin{split}
		\aiminnorm{\nu'(\theta^{(1)})\nabla \theta^{(1)}\cdot\nabla\boldsymbol v}
	&\leq C\aiminnorm{ \nabla \theta^{(1)}}_4\aiminnorm{ \nabla\boldsymbol v}_4\\
	&\leq C\aiminnorm{\Delta\theta^{(1)}} \aiminnorm{\nabla\boldsymbol v}^{1/2}  \aiminnorm{\Delta\boldsymbol v}^{1/2}
	\leq \mathcal Q \aiminnorm{\nabla\boldsymbol v}  + \frac{\underline\nu}{2} \aiminnorm{\Delta\boldsymbol v};
\end{split}
\end{equation}
\begin{equation}
	\aiminnorm{ \boldsymbol u^{(1)}\cdot\nabla\boldsymbol v} \leq \aiminnorm{ \boldsymbol u^{(1)}  }_{\infty}\aiminnorm{\nabla\boldsymbol v}
	\leq \aiminnorm{ \boldsymbol u^{(1)}  }_{H^2}\aiminnorm{\nabla\boldsymbol v}
	 \leq \mathcal Q\aiminnorm{\nabla\boldsymbol v};
\end{equation}
\begin{equation}
	\aiminnorm{ \boldsymbol v\cdot\nabla \boldsymbol u^{(2)} }\leq \aiminnorm{\boldsymbol v}_4\aiminnorm{ \nabla \boldsymbol u^{(2)} }_4\leq \aiminnorm{\nabla\boldsymbol v}\aiminnorm{\Delta \boldsymbol u^{(2)}}
	\leq \mathcal Q \aiminnorm{\nabla\boldsymbol v};
\end{equation}
\begin{equation}
	\aiminnorm{\eta} \leq \aiminnorm{\nabla\eta}.
\end{equation}
The last term in $\boldsymbol f$ is estimated similarly as for \eqref{eq3.3.4} by applying Lemma~\ref{lem3.2.1} with $\chi=\nu$ and $g=\boldsymbol u^{(2)}$; we find
\begin{equation}\label{eq3.3.13}
	\begin{split}
		\aiminnorm{ {\mathrm{div}\,}(( \nu(\theta^{(2)}) - \nu(\theta^{(1)}) ) \nabla\boldsymbol u^{(2)}) }
		\leq \mathcal Q\aiminnorm{\nabla\eta} + \frac{\underline\nu}{2}\aiminnorm{\Delta\eta}.
	\end{split}
\end{equation}
Combining all the estimates \eqref{eq3.3.9}-\eqref{eq3.3.13}, we conclude  that
\begin{equation}\label{eq3.3.14}
	\aiminnorm{\Delta\boldsymbol v} \leq \aiminnorm{\Delta\eta} +\frac{2}{\underline\nu}\aiminnorm{\boldsymbol v_t}+ \mathcal Q\big(\aiminnorm{\nabla\eta} + \aiminnorm{\nabla\boldsymbol v} \big),\quad\forall\,t\geq 0.
\end{equation}

With the relations \eqref{eq3.3.8},\,\eqref{eq3.3.14} at hand, in order to prove that the mapping \eqref{eq3.1.1} is Lipschitz continuous from $H^2(\Omega)$ into itself, we only need to prove that the mapping \eqref{eq3.3.0} is Lipschitz continuous from $H^2(\Omega)$ into $L^2(\Omega)$. Hence, taking the derivatives in \eqref{eq3.1.2} with respect to the time $t$ yields
\begin{equation}\begin{cases}\label{eq3.3.16}
	\partial_t \boldsymbol v_t - {\mathrm{div}\,}\big(\nu(\theta^{(1)})\nabla\boldsymbol v_t\big) + \boldsymbol u_t^{(1)}\cdot\nabla\boldsymbol v+ \boldsymbol u^{(1)}\cdot\nabla\boldsymbol v_t + \nabla p_t 	= {\mathrm{div}\,}\big(\nu'(\theta^{(1)})\theta^{(1)}_t\nabla\boldsymbol v\big) \\
	 \hspace{30pt}- {\mathrm{div}\,}\big(( \nu(\theta^{(2)}) - \nu(\theta^{(1)}) ) \nabla\boldsymbol u_t^{(2)}\big) -  {\mathrm{div}\,}\big(( \nu'(\theta^{(2)})\theta^{(2)}_t - \nu'(\theta^{(1)})\theta^{(1)}_t ) \nabla\boldsymbol u^{(2)}\big)\\
	 \hspace{100pt}-\boldsymbol v_t\cdot\nabla \boldsymbol u^{(2)}- \boldsymbol v\cdot\nabla \boldsymbol u_t^{(2)}  +\eta_t\boldsymbol e_2,\\
	 {\mathrm{div}\,} \boldsymbol v_t =0,\\
	 \partial_t\eta_t  -{\mathrm{div}\,}\big(\kappa(\theta^{(1)})\nabla\eta_t\big) + \boldsymbol u_t^{(1)}\cdot\nabla\eta + \boldsymbol u^{(1)}\cdot\nabla\eta_t
	 = {\mathrm{div}\,}\big(\kappa'(\theta^{(1)})\theta^{(1)}_t\nabla\eta\big)\\
	 \hspace{30pt}-{\mathrm{div}\,}\big( ( \kappa(\theta^{(2)}) - \kappa(\theta^{(1)}) ) \nabla\theta_t^{(2)}\big)
	 -{\mathrm{div}\,}\big( ( \kappa'(\theta^{(2)})\theta^{(2)}_t - \kappa'(\theta^{(1)})\theta^{(1)}_t ) \nabla\theta^{(2)}\big) \\
	 \hspace{100pt}- \boldsymbol v_t\cdot\nabla\theta^{(2)} - \boldsymbol v\cdot\nabla\theta^{(2)}_t + v_{2,t}\\
	 \hspace{30pt} -\kappa'(\theta^{(1)} )\theta^{(1)}_{y,t} + \kappa'(\theta^{(2)} )\theta^{(2)}_{y,t} -\kappa''(\theta^{(1)} )\theta^{(1)}_t\theta^{(1)}_y + \kappa''(\theta^{(2)} )\theta^{(2)}_t\theta^{(2)}_y.
\end{cases}\end{equation}

Taking the inner product of  \eqref{eq3.3.16} with $(\boldsymbol v_t,\eta_t)$ in $L^2(\Omega)$ and integrating by parts yields
\begin{equation}\label{eq3.3.17}
	\begin{split}
		\frac12\frac{\mathrm{d}}{\mathrm{d}t}\aiminnorm{\boldsymbol v_t}^2 + \int_\Omega\nu(\theta^{(1)})\aiminabs{\nabla\boldsymbol v_t}^2{\mathrm{d}x}{\mathrm{d}y}
		=-\int_\Omega \boldsymbol u_t^{(1)}\cdot\nabla\boldsymbol v\cdot\boldsymbol v_t{\mathrm{d}x}{\mathrm{d}y} - \int_\Omega\boldsymbol u^{(1)}\cdot\nabla\boldsymbol v_t\cdot\boldsymbol v_t{\mathrm{d}x}{\mathrm{d}y}\\
		-\int_\Omega \nu'(\theta^{(1)})\theta^{(1)}_t\nabla\boldsymbol v\cdot\nabla\boldsymbol v_t{\mathrm{d}x}{\mathrm{d}y}
		 +\int_\Omega ( \nu(\theta^{(2)}) - \nu(\theta^{(1)}) ) \nabla\boldsymbol u_t^{(2)} \cdot\nabla\boldsymbol v_t{\mathrm{d}x}{\mathrm{d}y}\\
		 +\int_\Omega ( \nu'(\theta^{(2)})\theta^{(2)}_t - \nu'(\theta^{(1)})\theta^{(1)}_t ) \nabla\boldsymbol u^{(2)} \cdot\nabla\boldsymbol v_t{\mathrm{d}x}{\mathrm{d}y}
		 -\int_\Omega \boldsymbol v_t\cdot\nabla \boldsymbol u^{(2)}\cdot\boldsymbol v_t {\mathrm{d}x}{\mathrm{d}y}\\
		 -\int_\Omega \boldsymbol v\cdot\nabla \boldsymbol u_t^{(2)}\cdot\boldsymbol v_t {\mathrm{d}x}{\mathrm{d}y} + \int_\Omega \eta_t\boldsymbol e_2\cdot\boldsymbol v_t{\mathrm{d}x}{\mathrm{d}y}\\
		 =:I_1 + \cdots + I_8,\hspace{150pt}
	\end{split}
\end{equation}
and
\begin{equation}\label{eq3.3.17p}
	\begin{split}
		\frac12\frac{\mathrm{d}}{\mathrm{d}t}\aiminnorm{\eta_t}^2 + \int_\Omega\kappa(\theta^{(1)})\aiminabs{\nabla\eta_t}^2{\mathrm{d}x}{\mathrm{d}y}
		=-\int_\Omega \boldsymbol u_t^{(1)}\cdot\nabla\eta\cdot\eta_t{\mathrm{d}x}{\mathrm{d}y} - \int_\Omega\boldsymbol u^{(1)}\cdot\nabla\eta_t\cdot\eta_t{\mathrm{d}x}{\mathrm{d}y}\\
		-\int_\Omega \kappa'(\theta^{(1)})\theta^{(1)}_t\nabla\eta\cdot\nabla\eta_t{\mathrm{d}x}{\mathrm{d}y}
		 +\int_\Omega ( \kappa(\theta^{(2)}) - \kappa(\theta^{(1)}) ) \nabla\theta_t^{(2)} \cdot\nabla\eta_t{\mathrm{d}x}{\mathrm{d}y}\\
		 +\int_\Omega ( \kappa'(\theta^{(2)})\theta^{(2)}_t - \kappa'(\theta^{(1)})\theta^{(1)}_t ) \nabla\theta^{(2)} \cdot\nabla\eta_t{\mathrm{d}x}{\mathrm{d}y}
		 -\int_\Omega \boldsymbol v_t\cdot\nabla \theta^{(2)}\cdot\eta_t {\mathrm{d}x}{\mathrm{d}y}\\
		 -\int_\Omega \boldsymbol v\cdot\nabla \theta_t^{(2)}\cdot\eta_t {\mathrm{d}x}{\mathrm{d}y} + \int_\Omega v_{2,t}\cdot\eta_t{\mathrm{d}x}{\mathrm{d}y}\\
		 +\int_\Omega \big(-\kappa'(\theta^{(1)} )\theta^{(1)}_{y,t} + \kappa'(\theta^{(2)} )\theta^{(2)}_{y,t}\big)\cdot\eta_t{\mathrm{d}x}{\mathrm{d}y}\\
		 +\int_\Omega \big( -\kappa''(\theta^{(1)} )\theta^{(1)}_t\theta^{(1)}_y + \kappa''(\theta^{(2)} )\theta^{(2)}_t\theta^{(2)}_y \big)\cdot\eta_t{\mathrm{d}x}{\mathrm{d}y}\\
		 =:J_1 + \cdots + J_{10}.\hspace{150pt}
	\end{split}
\end{equation}
Before we enter into the estimate of \eqref{eq3.3.17}-\eqref{eq3.3.17p}, let us first denote
\begin{equation}\label{eq3.3.18}
	g_2(t):=\sum_{i=1,2}\bigg( \aiminnorm{ \boldsymbol u^{(i)} }_{H^3}^2+\aiminnorm{ \theta^{(i)} }_{H^3}^2 + \aiminnorm{\boldsymbol u_t^{(i)}}_{H^1}^2 + \aiminnorm{\theta_t^{(i)}}_{H^1}^2 \bigg);
\end{equation}
note that for the strong solutions $(\boldsymbol u^{(i)}, \theta^{(i)})$, $i=1,2$, we have
\begin{equation}\label{eq3.3.19}
\int_0^t g_2(s){\mathrm{d}s}<+\infty,\quad\forall\,t\geq0.
\end{equation}

We now estimate the right-hand side of \eqref{eq3.3.17} term by term.
For the first term $I_1$, we use H\"older's inequality, the Sobolev embedding, and Young's inequality:
\begin{equation*}
\begin{split}
	I_1 &\leq \aiminnorm{ \boldsymbol u_t^{(1)} }_4\aiminnorm{\nabla\boldsymbol v}\aiminnorm{\boldsymbol v_t}_4
	\leq C\aiminnorm{\nabla\boldsymbol v}\aiminnorm{ \nabla\boldsymbol u_t^{(1)} }\aiminnorm{\nabla\boldsymbol v_t}\\
	&\leq C\aiminnorm{\nabla\boldsymbol v}^2\aiminnorm{ \nabla\boldsymbol u_t^{(1)} }^2 + \frac{\underline\nu}{32}\aiminnorm{\nabla\boldsymbol v_t}^2\\
	&\leq Cg_2(t)\aiminnorm{\nabla\boldsymbol v}^2+ \frac{\underline\nu}{32}\aiminnorm{\nabla\boldsymbol v_t}^2.
\end{split}
\end{equation*}
For the second term $I_2$, applying Lemma~\ref{lem2.1.3} with $f=\boldsymbol u^{(1)}$ and $g=\boldsymbol v_t$ gives
\begin{equation*}\begin{split}
	I_2 &\leq C\aiminnorm{ \boldsymbol u^{(1)} }^2\aiminnorm{ \nabla\boldsymbol u^{(1)} }^2 \aiminnorm{\boldsymbol v_t}^2 +   \frac{\underline\nu}{32}\aiminnorm{\nabla\boldsymbol v_t}^2\\
	&\leq \mathcal Q\aiminnorm{\boldsymbol v_t}^2 +  \frac{\underline\nu}{32}\aiminnorm{\nabla\boldsymbol v_t}^2.
\end{split}\end{equation*}
By H\"older's inequality, Ladyzhenskaya's inequality, and Young's inequality, the third term $I_3$ is estimated as follows:
\begin{equation*}\begin{split}
	I_3
	&\leq C\aiminnorm{  \theta^{(1)}_t }_4 \aiminnorm{\nabla\boldsymbol v}_4\aiminnorm{\nabla\boldsymbol v_t}
	\leq C\aiminnorm{ \theta^{(1)}_t }^{1/2}\aiminnorm{ \nabla\theta^{(1)}_t }^{1/2} \aiminnorm{\nabla\boldsymbol v}^{1/2} \aiminnorm{\Delta\boldsymbol v}^{1/2}\aiminnorm{\nabla\boldsymbol v_t} \\
	&\leq C\aiminnorm{ \theta^{(1)}_t }^2 \aiminnorm{\Delta\boldsymbol v}^2 + \aiminnorm{ \nabla\theta^{(1)}_t }^{2} \aiminnorm{\nabla\boldsymbol v}^2 +   \frac{\underline\nu}{32}\aiminnorm{\nabla\boldsymbol v_t}^2\\
	&\leq \mathcal Q \aiminnorm{\Delta\boldsymbol v}^2 + g_2(t)\aiminnorm{\nabla\boldsymbol v}^2  +   \frac{\underline\nu}{32}\aiminnorm{\nabla\boldsymbol v_t}^2.
\end{split}\end{equation*}
For the fourth term $I_4$ (the most problematic term), we use \eqref{eq3.2.3} and utilize Agmon's and Young's inequalities; we estimate $I_4$ as follows:
\begin{equation*}\begin{split}
	I_4
	&\leq C \aiminnorm{\eta}_\infty \aiminnorm{\nabla\boldsymbol u_t^{(2)} }\aiminnorm{ \nabla\boldsymbol v_t }
	\leq C\aiminnorm{\nabla\eta}^{1/2}\aiminnorm{\Delta\eta}^{1/2}  \aiminnorm{\nabla\boldsymbol u_t^{(2)} } \aiminnorm{ \nabla\boldsymbol v_t }\\
	&\leq C\aiminnorm{\nabla\eta}\aiminnorm{\Delta\eta} \aiminnorm{\nabla\boldsymbol u_t^{(2)} }^2 +  \frac{\underline\nu}{32}\aiminnorm{\nabla\boldsymbol v_t}^2\\
	&\leq Cg_2(t) \aiminnorm{\nabla\eta}\aiminnorm{\Delta\eta} +  \frac{\underline\nu}{32}\aiminnorm{\nabla\boldsymbol v_t}^2,\\
\end{split}\end{equation*}
which, together with the relation \eqref{eq3.3.8} between $\Delta\eta$ and $\eta_t$, implies that
\begin{equation*}\begin{split}
	I_4 &\leq \mathcal Q g_2(t) \aiminnorm{\nabla\eta} \big( \aiminnorm{\eta_t} + \aiminnorm{\nabla\eta} + \aiminnorm{\nabla\boldsymbol v} \big)+  \frac{\underline\nu}{32}\aiminnorm{\nabla\boldsymbol v_t}^2\\
	&\leq \mathcal Q g_2(t)\big( \aiminnorm{\nabla\eta}^2 + \aiminnorm{\eta_t}^2 + \aiminnorm{\nabla\boldsymbol v}^2 \big) +   \frac{\underline\nu}{32}\aiminnorm{\nabla\boldsymbol v_t}^2.\\
\end{split}\end{equation*}
For the fifth term $I_5$, using H\"older's inequality and Lemma~\ref{lem3.2.2} with $q=4$, $\chi=\nu'$, and $g^{(i)}=\theta_t^{(i)}$ for $i=1,2$ yield
\begin{equation*}\begin{split}
	I_5&\leq \aiminnorm{\nu'(\theta^{(2)})\theta^{(2)}_t - \nu'(\theta^{(1)})\theta^{(1)}_t}_4\aiminnorm{\nabla\boldsymbol u^{(2)}}_4 \aiminnorm{\nabla\boldsymbol v_t}\\
	&\leq \mathcal Q\big(\aiminnorm{\theta_t^{(1)} }\aiminnorm{\nabla\eta} + \aiminnorm{\theta_t^{(1)}-\theta_t^{(2)} }_4\big) \aiminnorm{\Delta\boldsymbol u^{(2)}} \aiminnorm{\nabla\boldsymbol v_t}\\
	&\leq \mathcal Q \aiminnorm{\nabla\eta}\aiminnorm{\nabla\boldsymbol v_t} + \mathcal Q \aiminnorm{\eta_t}_4\aiminnorm{\nabla\boldsymbol v_t},\\
\end{split}\end{equation*}
which, by exploiting Ladyzhenskaya's inequalities and Young's inequality, implies that
\begin{equation*}\begin{split}
	I_5&\leq\mathcal Q\aiminnorm{\nabla\eta}\aiminnorm{\nabla\boldsymbol v_t}+\mathcal Q\aiminnorm{\eta_t}^{1/2}\aiminnorm{\nabla\eta_t}^{1/2}\aiminnorm{\nabla\boldsymbol v_t}\\
	&\leq \mathcal Q(\aiminnorm{\nabla\eta}^2 + \aiminnorm{\eta_t}^2)+\frac{\underline\nu}{32}\aiminnorm{\nabla\boldsymbol v_t}^2 + \frac{\underline\kappa}{32}\aiminnorm{\nabla\eta_t}^2.
\end{split}\end{equation*}
For the sixth term $I_6$, by Ladyzhenskaya's inequalities and Young's inequality, we have
\begin{equation*}\begin{split}
	I_6 \leq \aiminnorm{ \nabla \boldsymbol u^{(2)} }\aiminnorm{\boldsymbol v_t}_4^2 \leq \mathcal Q \aiminnorm{\boldsymbol v_t}\aiminnorm{\nabla\boldsymbol v_t}
	\leq \mathcal Q \aiminnorm{\boldsymbol v_t}^2 +  \frac{\underline\nu}{32}\aiminnorm{\nabla\boldsymbol v_t}^2.
\end{split}\end{equation*}
Similarly as for the first term, we estimate $I_7$ as follows:
\begin{equation*}\begin{split}
	I_7 &\leq \aiminnorm{\boldsymbol v}_4 \aiminnorm{ \nabla\boldsymbol u_t^{(2)}} \aiminnorm{\boldsymbol v_t}_4
	\leq \aiminnorm{\nabla\boldsymbol v}  \aiminnorm{ \nabla\boldsymbol u_t^{(2)}} \aiminnorm{\nabla\boldsymbol v_t}\\
	&\leq  C\aiminnorm{ \nabla\boldsymbol u_t^{(2)}}^2 \aiminnorm{\nabla\boldsymbol v}^2 + \frac{\underline\nu}{32}\aiminnorm{\nabla\boldsymbol v_t}^2\\
	&\leq Cg_2(t) \aiminnorm{\nabla\boldsymbol v}^2 + \frac{\underline\nu}{32}\aiminnorm{\nabla\boldsymbol v_t}^2.
\end{split}\end{equation*}
The last term $I_8$ is easy and we have
\begin{equation*}\begin{split}
	I_8 \leq \aiminnorm{\eta_t}\aiminnorm{\boldsymbol v_t}\leq \frac12\aiminnorm{\eta_t}^2 + \frac12\aiminnorm{\boldsymbol v_t}^2.
\end{split}\end{equation*}
Combining all the estimates for $I_1,\cdots, I_8$, we conclude from \eqref{eq3.3.17} that
\begin{equation}\begin{split}\label{eq3.3.31}
	\frac12\frac{\mathrm{d}}{\mathrm{d}t}\aiminnorm{\boldsymbol v_t}^2 + \underline\nu\aiminnorm{\nabla\boldsymbol v_t}^2
	\leq &\mathcal Qg_2(t)\big( \aiminnorm{\boldsymbol v_t}^2 +  \aiminnorm{\eta_t}^2\big)
	+\mathcal Q   \aiminnorm{\Delta\boldsymbol v}^2\\
	&+\mathcal Q g_2(t)\big(\aiminnorm{\nabla\eta}^2  + \aiminnorm{\nabla\boldsymbol v}^2  \big)
	 + \frac{7\underline\nu}{32}\aiminnorm{\nabla\boldsymbol v_t}^2+\frac{\underline\kappa}{32}\aiminnorm{\nabla\eta_t}^2.
\end{split}\end{equation}

We then estimate the right-hand side of \eqref{eq3.3.17p} term by term. Proceeding similarly as for the $I_i$ , we can obtain the following estimates for the $J_i$, for all $i=1,\cdots,8$:
\begin{equation*}\begin{split}
	J_1, J_7 &\leq C g_2(t)\aiminnorm{\nabla\eta}^2 + \frac{\underline\kappa}{32}\aiminnorm{\nabla\eta_t}^2,\\
	J_2, J_6 &\leq \mathcal Q \aiminnorm{\eta_t}^2 + \frac{\underline\kappa}{32}\aiminnorm{\nabla\eta_t}^2,\\
	J_3 &\leq \mathcal Q \aiminnorm{\Delta\eta}^2 + g_2(t)\aiminnorm{\nabla\eta}^2 + \frac{\underline\kappa}{32}\aiminnorm{\nabla\eta_t}^2,\\
	J_4 &\leq  \mathcal Q g_2(t)\big( \aiminnorm{\nabla\eta}^2 + \aiminnorm{\eta_t}^2 + \aiminnorm{\nabla\boldsymbol v}^2 \big) +   \frac{\underline\nu}{32}\aiminnorm{\nabla\eta_t}^2,\\
	J_5 &\leq \mathcal Q(\aiminnorm{\nabla\eta}^2 + \aiminnorm{\eta_t}^2)+ \frac{\underline\kappa}{32}\aiminnorm{\nabla\eta_t}^2,\\
	J_8 &\leq \frac12\aiminnorm{\eta_t}^2 + \frac12\aiminnorm{\boldsymbol v_t}^2.
\end{split}\end{equation*}
We are left to estimate the last two terms $J_9$ and $J_{10}$.
Similarly to the proof of Lemmas~\ref{lem3.2.1}-\ref{lem3.2.2}, we need to add additional terms in $J_9$ and $J_{10}$ in order to extract $\eta$ which is the difference between $\theta^{(1)}$ and $\theta^{(2)}$. For $J_9$, by \eqref{eq3.2.3}, H\"older's inequality, and the Sobolev embedding, we find that
\begin{equation*}\begin{split}
	J_9 &\leq \mathcal Q\aiminnorm{\eta \theta_{y,t}^{(1)}\eta_t } + \mathcal Q\aiminnorm{ \eta_{y,t}\eta_t}\\
	&\leq \mathcal Q \aiminnorm{\eta}_4 \aiminnorm{ \theta_{y,t}^{(1)} } \aiminnorm{\eta_t}_4 + \mathcal Q\aiminnorm{\nabla\eta_t}\aiminnorm{\eta_t}\\
	&\leq \mathcal Q g_2(t)^{1/2} \aiminnorm{\nabla\eta}\aiminnorm{\nabla\eta_t} +\mathcal Q\aiminnorm{\nabla\eta_t}\aiminnorm{\eta_t}\\
	&\leq \mathcal Q g_2(t)\aiminnorm{\nabla\eta}^2 + \mathcal Q\aiminnorm{\eta_t}^2  + \frac{\underline\kappa}{32}\aiminnorm{\nabla\eta_t}^2.\\
\end{split}\end{equation*}
For $J_{10}$,  by \eqref{eq3.2.3}, H\"older's inequality, the Sobolev embedding, and Ladyzhenskaya's inequality, we find that
\begin{equation*}\begin{split}
	J_{10} &\leq \mathcal Q \aiminnorm{ \eta \theta_t^{(1)}\theta_y^{(1)} \eta_t} + \mathcal Q \aiminnorm{ \eta_t \theta_y^{(1)} \eta_t} + \mathcal Q \aiminnorm{ \theta_t^{(2)} \eta_y \eta_t}\\
	&\leq \mathcal Q \aiminnorm{\eta}_4\aiminnorm{ \theta_t^{(1)} }_4\aiminnorm{\theta_y^{(1)} }_4\aiminnorm{\eta_t}_4
	+\mathcal Q \aiminnorm{  \theta_y^{(1)} } \aiminnorm{ \eta_t }_4^2 + \mathcal Q \aiminnorm{\theta_t^{(2)} }_4\aiminnorm{\eta_y}\aiminnorm{\eta_t}_4\\
	&\leq \mathcal Q \aiminnorm{\nabla\eta}\aiminnorm{\nabla \theta_t^{(1)} }\aiminnorm{\nabla\theta_y^{(1)} }\aiminnorm{\nabla\eta_t} + \mathcal Q\aiminnorm{\eta_t}\aiminnorm{\nabla\eta_t} + \mathcal Q\aiminnorm{\nabla\theta_t^{(2)} }\aiminnorm{\eta_y}\aiminnorm{\nabla\eta_t}\\
	&\leq \mathcal Q g_2(t)^{1/2} \aiminnorm{\nabla\eta}\aiminnorm{\nabla\eta_t} + \mathcal Q\aiminnorm{\eta_t}\aiminnorm{\nabla\eta_t}  \\
	&\leq \mathcal Q g_2(t)\aiminnorm{\nabla\eta}^2 + \mathcal Q\aiminnorm{\eta_t}^2 + \frac{\underline\kappa}{32}\aiminnorm{\nabla\eta_t}^2.\\
\end{split}\end{equation*}
Combining all the estimates for $J_1,\cdots, J_{10}$, we arrive at
\begin{equation}\begin{split}\label{eq3.3.32}
	\frac12\frac{\mathrm{d}}{\mathrm{d}t}\aiminnorm{\eta_t}^2 + \underline\kappa\aiminnorm{\nabla\eta_t}^2
	\leq \mathcal Qg_2(t)\big( \aiminnorm{\boldsymbol v_t}^2 +  \aiminnorm{\eta_t}^2\big)
	+\mathcal Q \aiminnorm{\Delta\eta}^2 \\
	+\mathcal Q g_2(t)\big(\aiminnorm{\nabla\eta}^2  + \aiminnorm{\nabla\boldsymbol v}^2  \big)
	 + \frac{9\underline\kappa}{32}\aiminnorm{\nabla\eta_t}^2.
\end{split}\end{equation}

Summing \eqref{eq3.3.31} and \eqref{eq3.3.32} and denoting
\begin{equation}\begin{split}
	y_2(t):=\aiminnorm{\boldsymbol v_t}^2 + \aiminnorm{\eta_t}^2,
\end{split}\end{equation}
and dispensing with the terms involving $(\nabla\boldsymbol v_t,\nabla\eta_t)$, we arrive at
\begin{equation}\begin{split}\label{eq3.3.34}
	\frac{\mathrm{d}}{\mathrm{d}t} y_2(t) \leq \mathcal Qg_2(t)y_2(t) +  \mathcal Q\big( \aiminnorm{\Delta\eta}^2  + \aiminnorm{\Delta\boldsymbol v}^2 \big)
	+\mathcal Q g_2(t) y_1(t),
\end{split}\end{equation}
where $y_1(t)$ is defined in \eqref{eq3.2.9-1}. Now, applying the usual Gronwall lemma to \eqref{eq3.3.34}, we obtain
\begin{equation*}\begin{split}
	y_2(t) \leq y_2(0) \exp\bigg\{\mathcal Q \int_0^t g_2(s){\mathrm{d}s}\bigg\} + \mathcal Q\int_0^t\big( \aiminnorm{\Delta\eta}^2  + \aiminnorm{\Delta\boldsymbol v}^2 \big){\mathrm{d}s} + \mathcal Q\int_0^t g_2(s)y_1(s){\mathrm{d}s},
\end{split}\end{equation*}
which, by combining the estimates \eqref{eq3.2.11}-\eqref{eq3.2.12}, implies that
\begin{equation}\begin{split}\label{eq3.3.36}
	y_2(t) &\leq y_2(0) \exp\bigg\{\mathcal Q \int_0^t g_2(s){\mathrm{d}s}\bigg\} + \mathcal Q y_1(0) \exp\big\{\mathcal Q t \} \big( 1+\int_0^t g_2(s) {\mathrm{d}s}\big).\\
\end{split}\end{equation}
We can conclude from \eqref{eq3.3.36} that the mapping \eqref{eq3.3.0} is continuous from $H^2(\Omega)$ to $L^2(\Omega)$ as long as we have the following estimate for $y_2(0)$:
\begin{equation}\label{eq3.3.37}
	y_2(0) \leq \mathcal Q \big( \aiminnorm{ \boldsymbol v(0) }_{H^2}^2 + \aiminnorm{\eta(0)}_{H^2}^2 \big).
\end{equation}
To show \eqref{eq3.3.37}, we take the inner product of \eqref{eq3.1.2p} when $t=0$ with $(\boldsymbol v_t(0), \theta_t(0))$ and then employ Lemma~\ref{lem2.1.1} and the uniform estimates \eqref{eq3.1.3}; we are able to prove \eqref{eq3.3.37} and we omit the details here.

Now, going back to \eqref{eq3.3.8} and \eqref{eq3.3.14} and using \eqref{eq3.2.11},\,\eqref{eq3.3.36}, we can conclude the following result.
\begin{proposition}\label{prop3.2.1}
	Let $(\boldsymbol u(t),\theta(t))$ be the strong solutions of the Boussinesq system \eqref{eq1.1.3}-\eqref{eq1.1.4} with initial data $(\boldsymbol u(0),\theta(0))$ given by Theorem~\ref{thm1.2}. Then for  any fixed $t>0$, the mapping
		$$(\boldsymbol u(0),\theta(0))\mapsto (\boldsymbol u(t),\theta(t))$$
	is Lipschitz continuous from $H^2(\Omega)$ to $H^2(\Omega)$.
\end{proposition}

\section{The global attractor}\label{sec4}
In this section, we are going to conclude the proof of the main Theorem~\ref{thm1.3} and in particular show the existence of the global attractor. In order to achieve this goal, we utilize the following abstract result from \cite[Chapter I]{Tem97} about semigroups and the existence of their global attractors. One can also refer to \cite{Hal88, Lad91, BV92, CV02} for additional developments concerning the theory of global attractors.
\begin{theorem}\label{thm4.2}
Suppose that $X$ is a metric space with metric $d(\cdot,\cdot)$ and the semigroup $\aiminset{S(t)}_{t\geq 0}$ is a family of operators from $X$ into $X$ itself such that:
\begin{enumerate}[$\quad(i)$]
\item for each fixed $t>0$, $S(t)$ is continuous from $X$ into itself;
\item for some $t_0>0$, $S(t_0)$ is compact from $X$ into itself;
\item there exists a subset $B_0$ of $X$ which is bounded, and a subset $U$ of $X$ which is open, such that $B_0\subset U\subset X$, and $B_0$ is absorbing in $U$ for the semigroup, that is for any bounded subset $B\subset U$, there exists $t_0=t_0(B)>0$ such that
\[
S(t)B\subset B_0,\quad\quad \forall\,t>t_0(B).
\]
($B_0$ is also called the absorbing set of the semigroup in $U$).
\end{enumerate}
Then $\mathcal A:=\omega(B_0)$, the $\omega$-limit set of $B_0$, is a compact attractor which attracts all the bounded sets of $U$, that is for any bounded set  $B\subset U$,
\[
	\lim_{t\rightarrow+\infty}\text{dist}(S(t)B,\mathcal A) = 0,
\]
where $\text{dist}(B_1,B_2):=\sup_{x\in B_1}\inf_{y\in B_2}d(x,y)$ is the non-symmetric Hausdorff distance between subsets of $X$. Furthermore, the set $\mathcal A$ is the maximal bounded attractor in $U$ for the inclusion relation.

Suppose in addition that $X$ is a Banach space, $U$ is convex and
\begin{enumerate}[$\quad(i)$]
\setcounter{enumi}{3}
\item for any $x\in X$, $t\mapsto S(t)x$ from $\mathbb R_+$ to $X$ is continuous.
\end{enumerate}
Then $\mathcal A=\omega(B_0)$ is also connected.

If $U=X$, $\mathcal A$ is the global attractor of the semigroup $\aiminset{S(t)}_{t\geq 0}$ in $X$.
\end{theorem}

We now carry out the proof of Theorem~\ref{thm1.3} by checking all the conditions of Theorem~\ref{thm4.2} with
\[
	X = W= \{ (\boldsymbol u,\theta)\in H^2(\Omega)^3\,:\,{\mathrm{div}\,}\boldsymbol u=0,\; (u_2,\frac{\partial u_1}{\partial y},\theta )\big|_{y=0,1}=0,\text{ and }\int_\Omega u_1{\mathrm{d}x}{\mathrm{d}y} = 0 \}.
\]
Item $(iii)$, the existence of an absorbing ball in the space $W$, has been proved in Section~\ref{sec2}; item $(i)$ has been proved in Section~\ref{sec3}. Therefore, we are left to check items $(ii)$ and $(iv)$. Once they are proven, Theorem~\ref{thm1.3} will then be completely proved by the application of Theorem~\ref{thm4.2}.

Let us first check item $(iv)$. We need a useful result from Lions and Magenes \cite[Chapter I, Theorem 3.1]{LM72}, which implies that

\begin{lemma}\label{lem4.1.1}
	Let $X$ and $Y$ be two Hilbert spaces such that $X$ is separable with
	\[
	X\subset Y,\quad X\text{ dense in }Y\text{ with continuous injection.}
	\]
	Suppose that $u\in L^2(0, t_1; X)$ and $u_t\in L^2(0, t_1; Y)$. Then $u$ is almost everywhere equal to a function continuous from $[0, t_1]$ into $[X,Y]_{1/2}$.
\end{lemma}

For a definition of the interpolation space $[X,Y]_{1/2}$, see \cite[Chapter I, Section 2.1]{LM72}. We also infer from \cite[Chapter I, Theorem 9.6]{LM72} that
\begin{equation}\label{eq4.1.1}
	[ H^3(\Omega), H^1(\Omega)]_{1/2} = H^2(\Omega).
\end{equation}

Now, we are ready to show item $(iv)$. Given $(\boldsymbol u_0,\theta_0)\in W$, then it has been shown in Section~\ref{sec2} that for any fixed $ t_1>0$, the strong solutions $(\boldsymbol u,\theta)$ of the Boussinesq system \eqref{eq1.1.3}-\eqref{eq1.1.4} satisfy
\[
(\boldsymbol u, \theta)\in L^2(0, t_1; H^3(\Omega)),\quad\quad (\partial_t\boldsymbol u,\partial_t\theta)\in L^2(0, t_1; H^1(\Omega)).
\]
Then applying Lemma~\ref{lem4.1.1} with $X=H^3(\Omega)$ and $Y=H^1(\Omega)$ and using \eqref{eq4.1.1}, we find that
\begin{equation}
(\boldsymbol u, \theta)\in \mathcal C([0, t_1]; [H^3(\Omega), H^1(\Omega)]_{1/2}=H^2(\Omega)),
\end{equation}
which proves item $(iv)$.



We are now going to check item $(ii)$. Before we show the compactness of $S(t)$ for some $t>0$, let us first recall the compactness lemma of J.-L. Lions \cite{Lio69} (see also \cite{Aub63}).
\begin{lemma}\label{lem4.1.2}[Aubin-Lions]
	Let $X_0$, $X$ and $X_1$ be Banach spaces such that $X_0$ and $X_1$ are reflexive, $X_0\hookrightarrow X \hookrightarrow X_1$, each embedding is continuous and each space is dense in the next one. Assume also that $X_0$ is compactly embedded in $X$. For $0< t_1<\infty$, let
	\[
		Y:=\aiminset{ u\,:\, u\in L^2(0, t_1; X_0),\; \frac{{\mathrm{d}} u}{{\mathrm{d}t}}\in L^2(0, t_1; X_1) }.
	\]
	Then $Y$ is a Banach space equipped with the norm $\aiminnorm{u}_{L^2(0, t_1; X_0) } +\aiminnorm{u'}_{L^2(0, t_1; X_1) }$. Moreover, $Y$ is compactly embedded in $L^2(0, t_1; X)$.
\end{lemma}

Following \cite[pp. 176-177]{Ju07} line by line, we are able to prove the desired compactness result. We present the details here for the sake of completeness.
Because of the existence of an absorbing set for $\aiminset{S(t)}_{t\geq 0}$ in $W$, to prove the compactness of $\aiminset{S(t)}_{t\geq 0}$ in $W$, it is enough to consider a bounded subset $\mathfrak B$ of $W$. For any fixed $ t_1>0$, define the set $\mathfrak C_{ t_1}$ as a subset of the function space $L^2(0, t_1;L^2(\Omega))$:
\[
\mathfrak C_{ t_1} = \aiminset{ (\boldsymbol u, \theta)\,|\, (\boldsymbol u_0,\theta_0)\in \mathfrak B,\,(\boldsymbol u(t), \theta(t)) = S(t)(\boldsymbol u_0,\theta_0),\;t\in [0, t_1]}.
\]
Notice that for $\Omega$ bounded, $H^1(\Omega)$ is compactly embedded in $L^2(\Omega)$. If $(\boldsymbol u_0,\theta_0)\in \mathfrak B$, then it has been shown in Section~\ref{sec2} that for any fixed $ t_1>0$, the strong solution $(\boldsymbol u,\theta)$ satisfies
\[
(\boldsymbol u, \theta)\in L^2(0, t_1; H^3(\Omega)),\quad\quad (\partial_t\boldsymbol u,\partial_t\theta)\in L^2(0, t_1; H^1(\Omega)).
\]
Therefore, by Lemma~\ref{lem4.1.2} with
\[
X_0=H^3(\Omega),\quad\quad X=H^2(\Omega),\quad\quad X_1=H^{1}(\Omega),
\]
$\mathfrak C_{ t_1}$ is compact in $L^2(0, t_1; H^2(\Omega))$.

In order to show that for any fixed $t>0$, $S(t)$ is a compact operator in $W$, we take any bounded sequence $\aiminset{(\boldsymbol u_{0,n},\theta_{0,n})}_{0}^{\infty}\subset\mathfrak B$ and we want to extract, for any fixed $t>0$, a convergent subsequence from $\aiminset{S(t)(\boldsymbol u_{0,n},\theta_{0,n})}_0^\infty$.

Since $\aiminset{(\boldsymbol u_n, \theta_n)}_0^\infty\subset\mathfrak C_{ t_1}$, by Lemma~\ref{lem4.1.2}, there exists a function $(\bar{\boldsymbol u},\bar\theta)$:
\[
(\bar{\boldsymbol u},\bar\theta) \in L^2(0, t_1; H^2(\Omega)),
\]
and a subsequence of $\aiminset{S(\cdot)(\boldsymbol u_{0,n},\theta_{0,n})}_0^\infty$, still denoted as $\aiminset{S(\cdot)(\boldsymbol u_{0,n},\theta_{0,n})}_0^\infty$ for simplicity of notation, such that
\begin{equation}\label{eq4.1.10}
	\lim_{n\rightarrow \infty}\int_0^{ t_1}\aiminnorm{  S(t)(\boldsymbol u_{0,n},\theta_{0,n}) - (\bar{\boldsymbol u}(t),\bar\theta(t))}_{H^2}^2 {\mathrm{d}t} = 0.
\end{equation}
By Riesz Lemma and \eqref{eq4.1.10}, it then follows that there exists a subsequence of $\aiminset{S(\cdot)(\boldsymbol u_{0,n},\theta_{0,n})}_0^\infty$, still denoted as $\aiminset{S(\cdot)(\boldsymbol u_{0,n},\theta_{0,n})}_0^\infty$ for simplicity of notation, such that
\begin{equation}\label{eq4.1.11}
	\lim_{n\rightarrow \infty}\aiminnorm{  S(t)(\boldsymbol u_{0,n},\theta_{0,n}) - (\bar{\boldsymbol u}(t),\bar\theta(t))}_{H^2}  = 0,\quad\text{ a.e. }t\in (0, t_1).
\end{equation}
Fix any $t\in(0, t_1)$, and by \eqref{eq4.1.11}, we can select a $t_*\in(0,t)$ such that
\[
	\lim_{n\rightarrow \infty}\aiminnorm{  S(t_*)(\boldsymbol u_{0,n},\theta_{0,n}) - (\bar{\boldsymbol u}(t_*), \bar\theta(t_*))}_{H^2}  = 0.
\]
Then by continuity of the map $S(t-t_*)$ in $W$, we have
\begin{equation}\begin{split}
	S(t)(\boldsymbol u_{0,n},\theta_{0,n}) &= S(t-t_*)S(t_*)(\boldsymbol u_{0,n},\theta_{0,n})\\
	&\rightarrow S(t-t_*)(\bar{\boldsymbol u}(t_*), \bar\theta(t_*)),\quad\text{ in }W.
\end{split}\end{equation}
Hence for any $t>0$,  $\aiminset{S(t)(\boldsymbol u_{0,n},\theta_{0,n})}_0^\infty$ contains a subsequence which is convergent in $W$, which proves the compactness of the operator $S(t)$ in $W$. The proof of Theorem~\ref{thm1.3} is then completed.

\section*{Acknowledgments}
This work was partially supported by the National Science Foundation under the grant NSF DMS-1206438, and by the Research Fund of Indiana University. 

\bibliographystyle{amsalpha}

\begin{thebibliography}{ACW11}

\bibitem[ACW11]{ACW11}
Dhanapati Adhikari, Chongsheng Cao, and Jiahong Wu, \emph{Global regularity
  results for the 2d boussinesq equations with vertical dissipation}, Journal
  of Differential Equations \textbf{251} (2011), no.~6, 1637--1655.

\bibitem[Aub63]{Aub63}
J.-P. Aubin, \emph{Un th\'eor\`eme de compacit\'e}, C. R. Acad. Sci. Paris
  \textbf{256} (1963), 5042--5044.

\bibitem[BV92]{BV92}
A.~V. Babin and M.~I. Vishik, \emph{Attractors of evolution equations}, Studies
  in Mathematics and its Applications, vol.~25, North-Holland Publishing Co.,
  Amsterdam, 1992, Translated and revised from the 1989 Russian original by
  Babin. \MR{1156492 (93d:58090)}

\bibitem[Cha06]{Cha06}
Dongho Chae, \emph{Global regularity for the 2{D} {B}oussinesq equations with
  partial viscosity terms}, Adv. Math. \textbf{203} (2006), no.~2, 497--513.

\bibitem[CT06]{CT06}
C.~Cao and E.~S. Titi, \emph{Global well-posedness of the three-dimensional
  viscous primitive equations of large scale ocean andatmosphere dynamics},
  Annals of Mathematics \textbf{166} (2007), 245--267.

\bibitem[CV02]{CV02}
Vladimir~V. Chepyzhov and Mark~I. Vishik, \emph{Attractors for equations of
  mathematical physics}, American Mathematical Society Colloquium Publications,
  vol.~49, American Mathematical Society, Providence, RI, 2002.

\bibitem[CW13]{CW13}
Chongsheng Cao and Jiahong Wu, \emph{Global regularity for the two-dimensional
  anisotropic boussinesq equations with vertical dissipation}, Archive for
  Rational Mechanics and Analysis \textbf{208} (2013), no.~3, 985--1004
  (English).

\bibitem[Eva98]{Eva98}
L.~C. Evans, \emph{Partial {D}ifferential {E}quations}, vol.~19, Amer. Math.
  Soc., Providence, RI, 1998.

\bibitem[FMT87]{FMT87}
C.~Foias, O.~Manley, and R.~Temam, \emph{Attractors for the {B}\'enard problem:
  existence and physical bounds on their fractal dimension}, Nonlinear Anal.
  \textbf{11} (1987), no.~8, 939--967.

\bibitem[FP67]{FP67}
C.~Foias and G.~Prodi, \emph{Sur le comportement global des solutions
  non-stationnaires des \'equations de {N}avier-{S}tokes en dimension {$2$}},
  Rend. Sem. Mat. Univ. Padova \textbf{39} (1967), 1--34.

\bibitem[Hal88]{Hal88}
Jack~K. Hale, \emph{Asymptotic behavior of dissipative systems}, Mathematical
  Surveys and Monographs, vol.~25, American Mathematical Society, Providence,
  RI, 1988. \MR{941371 (89g:58059)}

\bibitem[HL05]{HL05}
Thomas~Y. Hou and Congming Li, \emph{Global well-posedness of the viscous
  {B}oussinesq equations}, Discrete Contin. Dyn. Syst. \textbf{12} (2005),
  no.~1, 1--12.

\bibitem[Ju01]{Ju01}
Ning Ju, \emph{Existence of global attractor for the three-dimensional modified
  {N}avier-{S}tokes equations}, Nonlinearity \textbf{14} (2001), no.~4,
  777--786.

\bibitem[Ju07]{Ju07}
\bysame, \emph{The global attractor for the solutions to the 3{D} viscous
  primitive equations}, Discrete Contin. Dyn. Syst. \textbf{17} (2007), no.~1,
  159--179.

\bibitem[Lad91]{Lad91}
Olga Ladyzhenskaya, \emph{Attractors for semigroups and evolution equations},
  Lezioni Lincee. [Lincei Lectures], Cambridge University Press, Cambridge,
  1991.

\bibitem[LB96]{LB96}
Sebasti{\'a}n~A. Lorca and Jos{\'e}~Luiz Boldrini, \emph{Stationary solutions
  for generalized {B}oussinesq models}, J. Differential Equations \textbf{124}
  (1996), no.~2, 389--406.

\bibitem[LB99]{LB99}
Sebasti\'an~A. Lorca and Jos\'e~Luiz Boldrini, \emph{The initial value problem
  for a generalized boussinesq model}, Nonlinear Analysis: Theory, Methods \&
  Applications \textbf{36} (1999), no.~4, 457--480.

\bibitem[Lio69]{Lio69}
J.-L. Lions, \emph{Quelques m\'ethodes de r\'esolution des probl\`emes aux
  limites non lin\'eaires}, Dunod, 1969.

\bibitem[LM72]{LM72}
J.-L. Lions and E.~Magenes, \emph{Non-homogeneous boundary value problems and
  applications. {V}ol. {I}}, Springer-Verlag, New York, 1972, Translated from
  the French by P. Kenneth, Die Grundlehren der mathematischen Wissenschaften,
  Band 181.

\bibitem[LPZ13]{LPZ13}
Huapeng Li, Ronghua Pan, and Weizhe Zhang, \emph{Initial boundary value problem
  for 2d boussinesq equations with temperature-dependent heat diffusion},
  preprint, 2013.

\bibitem[LPZ11]{LPZ11}
Ming-Jun Lai, Ronghua Pan, and Kun Zhao, \emph{Initial boundary value problem
  for two-dimensional viscous {B}oussinesq equations}, Arch. Ration. Mech.
  Anal. \textbf{199} (2011), no.~3, 739--760.

\bibitem[MZ97]{MZ97}
A.~Miranville and M.~Ziane, \emph{On the dimension of the attractor for the
  {B}\'enard problem with free surfaces}, Russian J. Math. Phys. \textbf{5}
  (1997), no.~4, 489--502 (1998).

\bibitem[MZ13]{MZ13}
Changxing Miao and Xiaoxin Zheng, \emph{On the global well-posedness for the
  boussinesq system with horizontal dissipation}, Communications in
  Mathematical Physics \textbf{321} (2013), no.~1, 33--67 (English).

\bibitem[SZ13]{SZ13}
Yongzhong Sun and Zhifei Zhang, \emph{Global regularity for the
  initial-boundary value problem of the 2-d boussinesq system with variable
  viscosity and thermal diffusivity}, Journal of Differential Equations
  \textbf{255} (2013), no.~6, 1069--1085.

\bibitem[Tem97]{Tem97}
Roger Temam, \emph{Infinite-dimensional dynamical systems in mechanics and
  physics}, second ed., Applied Mathematical Sciences, vol.~68,
  Springer-Verlag, New York, 1997.

\bibitem[Tem01]{Tem01}
\bysame, \emph{Navier-{S}tokes {E}quations, {T}heory and {N}umerical
  {A}nalysis}, AMS Chelsea Publishing, Providence, RI, 2001, Reprint of the
  1984 edition.

\bibitem[TS82]{TS82}
D~L. Turcotte and G.~Schubert, \emph{{G}eodynamics {A}pplications of
  {C}ontinuum {P}hysics to {G}eological {P}roblems}, John Wiley and Sons, 1982.

\bibitem[Wang05]{Wang05}
Xiaoming Wang, \emph{A note on long time behavior of solutions to the
  {B}oussinesq system at large {P}randtl number}, Nonlinear partial
  differential equations and related analysis, Contemp. Math., vol. 371, Amer.
  Math. Soc., Providence, RI, 2005, pp.~315--323.

\bibitem[Wang07]{Wang07}
\bysame, \emph{Asymptotic behavior of the global attractors to the {B}oussinesq
  system for {R}ayleigh-{B}\'enard convection at large {P}randtl number}, Comm.
  Pure Appl. Math. \textbf{60} (2007), no.~9, 1293--1318.

\bibitem[WZ11]{WZ11}
Chao Wang and Zhifei Zhang, \emph{Global well-posedness for the 2-d boussinesq
  system with the temperature-dependent viscosity and thermal diffusivity},
  Advances in Mathematics \textbf{228} (2011), no.~1, 43--62.

\bibitem[Zia98]{Zia98}
M.~Ziane, \emph{On the two-dimensional {N}avier-{S}tokes equations with the
  free boundary condition}, Appl. Math. Optim. \textbf{38} (1998), no.~1,
  1--19.

\end{thebibliography}
\providecommand{\bysame}{\leavevmode\hbox to3em{\hrulefill}\thinspace}
\providecommand{\MR}{\relax\ifhmode\unskip\space\fi MR }
\providecommand{\MRhref}[2]{%
  \href{http://www.ams.org/mathscinet-getitem?mr=#1}{#2}
}
\providecommand{\href}[2]{#2}

\end{document}